\documentclass[a4paper,11pt]{amsart}
\title{Bounded deviations in higher genus II: minimal laminations}
\date{\today}

\author{Pierre-Antoine Guihéneuf}
\address{Pierre-Antoine Guihéneuf: Sorbonne Université, Université Paris Cité, CNRS, IMJ-PRG, F-75005 Paris, France --- 
IRL Jean-Christophe Yoccoz CNRS / IMPA, Estr. Dona Castorina, 110
Jardim Botânico, Rio de Janeiro, Brasil}
\email{pierre-antoine.guiheneuf@imj-prg.fr}
\thanks{P.-A.\ G.\ thanks the Jean-Christophe Yoccoz international laboratory CNRS/IMPA for the semester in Brazil during which the ideas of this work were born.}

\author{F\'abio Armando Tal}
\address{F\'abio Armando Tal: Instituto de Matem\'atica e Estat\'istica da Universidade de S\~ao Paulo, R. do Mat\~ao, 1010
- Vila Universitaria, S\~ao Paulo, Brasil}
\email{fabiotal@ime.usp.br}

\usepackage[utf8]{inputenc}
\usepackage[english]{babel}
\usepackage[T1]{fontenc}  
\usepackage{amsfonts}
\usepackage{amsmath}
\usepackage{amssymb}
\usepackage{amsthm}

\DeclareFontFamily{U}{mathx}{}
\DeclareFontShape{U}{mathx}{m}{n}{<-> mathx10}{} 
\DeclareSymbolFont{mathx}{U}{mathx}{m}{	n}
\DeclareMathAccent{\widehat}{0}{mathx}{"70}
\DeclareMathAccent{\widecheck}{0}{mathx}{"71}

\usepackage{enumitem}
\setlist{noitemsep}
\usepackage{tikz}
\usepackage[pdftex,colorlinks=true,linkcolor=blue,citecolor=blue,urlcolor=blue]{hyperref}

\newtheorem{lemma}{Lemma}[section]
\newtheorem{theorem}[lemma]{Theorem}
\newtheorem{theo}{Theorem}
\newtheorem{corollary}[theo]{Corollary}

\newtheorem{prop}[lemma]{Proposition}

\newtheorem{claim}[lemma]{Claim}
\theoremstyle{definition}
\newtheorem{definition}[lemma]{Definition}
\theoremstyle{remark}
\newtheorem{rem}[lemma]{Remark}

\newcommand{\F}{\mathcal{F}}
\newcommand{\X}{\mathcal{X}}
\newcommand{\Hy}{\mathbf{H}}
\newcommand{\N}{\mathbf{N}}

\newcommand{\R}{\mathbf{R}}

\newcommand{\T}{\mathbf{T}}

\newcommand{\G}{\mathcal{G}}

\newcommand{\Q}{\mathbf{Q}}
\newcommand{\Z}{\mathbf{Z}}
\newcommand{\M}{\mathcal{M}^\textrm{erg}_{\vartheta>0}}
\newcommand{\varep}{\varepsilon}
\newcommand{\Homeo}{\operatorname{Homeo}}

\newcommand{\Fix}{\operatorname{Fix}}
\newcommand{\rot}{\operatorname{rot}}
\newcommand{\rote}{\operatorname{rot}_{\mathrm{erg}}}

\newcommand{\conv}{\operatorname{conv}}
\newcommand{\diam}{\operatorname{diam}}
\newcommand{\inte}{\operatorname{int}}
\newcommand{\dom}{\operatorname{dom}}
\newcommand{\Id}{\operatorname{Id}}
\newcommand{\fil}{\operatorname{fill}}
\newcommand{\dd}{\,\mathrm{d}}
\newcommand{\cl}{\mathcal{N}}
\newcommand{\wt}{\widetilde}
\newcommand{\wh}{\widehat}
\newcommand{\wc}{\widecheck}
\newcommand{\pr}{\operatorname{pr}}
\newcommand{\Me}{\mathcal{M}^{\mathrm{erg}}}
\newcommand{\Merg}{\mathcal{M}^{\mathrm{erg}}_{\vartheta>0}}

\begin{document}

\maketitle

\begin{abstract}
This article follows and completes \cite{paper1PAF}, where we study the problem of bounded deviations for homeomorphisms of closed surfaces of genus $\ge 2$. This second part deals with bounded deviations relative to geodesic minimal laminations that are not reduced to a closed geodesic. The combination of both articles generalises to the higher genus case most of the bounded deviations results already known for the torus.
\end{abstract}

\tableofcontents

\section{Introduction}

This article is the second part of our study of bounded deviations for homeomorphisms of closed hyperbolic surfaces of genus $g \ge 2$.  In Part~I \cite{paper1PAF} we treated the case of bounded deviations with respect to closed geodesics, which is the higher genus analogue of the case of torus homeomorphisms whose rotation set is a segment with rational slope.  
Here we address the complementary situation: bounded deviations relative to \emph{minimal geodesic laminations that are not reduced to a closed geodesic}. These laminations play the role of the ``irrational–slope'' directions in the rotation theory of the torus.

For homeomorphisms of $\T^2$ isotopic to the identity, the geometry of the rotation set provides sharp control on the long-term displacement of orbits. In the case where the rotation set is a nondegenerate segment with irrational slope and containing a rational point  (note that by \cite{lct1}, this rational point has to be an endpoint of the segment), two fundamental bounded deviation statements are known: bounded deviations in the direction orthogonal to the segment \cite{zbMATH07488214} and bounded deviations in the direction parallel to it
\cite{zbMATH07867510}.  Our aim here is to establish higher genus counterparts of these two phenomena in the setting of the tracking laminations introduced in
\cite{alepablo}.

To better illustrate the situation in higher genus surfaces, let us describe the case of nonwandering flows. A classical theorem (\cite{zbMATH03467479}, see also \cite[Theorem 3.1.7]{zbMATH01321275} and similar statements of Section 3.1) states that in this case, if the flow has a finite number of singularities, there exists a decomposition into finitely many integrable and quasi-minimal open pieces, plus some closure of heteroclinic orbits (orbits whose $\alpha$ and $\omega$-limit sets are made of singularities): an integrable piece is an annulus foliated by periodic orbits; a quasi-minimal piece is a sub-surface with genus in which any orbit having non fixed points in its accumulation set is dense. In the present article we are interested in the counterpart of quasi-minimal pieces for general surface homeomorphisms homotopic to the identity.

We postpone the more technical aspects of the definition of tracking laminations for the next section (see Subsection~\ref{SecTrack}), but let us give a brief description. Let $f$ be an homeomorphism of a closed surface $S$, isotopic to the identity. For every ergodic $f$-invariant probability measure $\mu$ one can define an (homotopical) rotation speed, and if $\mu$ has positive rotation speed, typical points for $\mu$ follow a so-called \emph{tracking geodesic}.  
The collection of all these geodesics forms a closed set  $\Lambda_\mu$, and the equivalence classes introduced in \cite{alepablo} provide a canonical decomposition of the rotational dynamics of $f$. In that context, if $\mu$ does not belong to a chaotic class, then $\Lambda_\mu$ is a minimal geodesic lamination of $T^1S$ (and in particular the associated tracking geodesics are simple). We are interested here in the case where $\mu$ is not in a chaotic class and $\Lambda_\mu$ is minimal non-closed, meaning it is not made of a single closed geodesic. In this case, one can associate to the class $\cl_i$ of $\mu$ a subsurface $S_i$ that plays a role analogous to the quasi-minimal pieces for the flows described above. The boundary of this connected sub-surface is made of a finite and disjoint union of simple closed geodesics (and $S_i$ is characterized as minimal for inclusion for these properties).

\subsection{Main results}

Let $S$ be a compact boundaryless orientable connected surface, and denote $\Homeo_0(S)$ the set of homeomorphisms of $S$ that are homotopic to the identity. Let $f\in \Homeo_0(S)$. 
We denote $\wt S$ the universal cover of the surface $S$ and $\wt f$ the (unique) lift of $f$ to $\wt S$ that commutes with deck transformations. The surface $S$ is equipped with a metric of constant curvature $-1$ that induces a distance $d$, which lifts to a distance on $\wt S$ that we also denote by $d$. 

Denote $\Me(f)$ the set of $f$-invariant ergodic Borel probability measures. 

\subsubsection*{Orientability of the lamination}

As a warm-up, we prove that the laminations associated to minimal non-closed classes are orientable.

\begin{theo}\label{TheoLaminMinim}
Let $S$ be a compact boundaryless hyperbolic surface and $f\in \Homeo_0(S)$.
Let $\cl_i$ be a minimal non-closed class of $\M(f)$ and $\Lambda_i$ the associated lamination.
Then $\Lambda_i$ is orientable.
\end{theo}

It is a classical result that there is no minimal non-closed orientable lamination on a finite punctured sphere, hence the surface $S_i$ associated to such a class $\cl_i$ must have genus. This answers \cite[Remark 6.10]{alepablo}.
\bigskip

Recall that $\wt f$, the canonical lift of $f$ to the universal cover $\wt S$ of $S$, extends continuously with the identity to the boundary at infinity $\partial\wt S$ of $\wt S$.
In the sequel, for $\wt A\subset \wt S$ and $R>0$, we denote 
\[V_R({\wt A}) = \big\{\wt x\in \wt S\mid\exists \wt a\in\wt A : d(\wt x,\wt a)\le R\big\}\]
the $R$-neighbourhood of $\wt A$. We also denote $i:A\to S$ the inclusion map.
We say that two disjoint geodesics $\wt\gamma, \wt\gamma'$ of $\wt S$ \emph{have the same orientation} if either $L(\wt\gamma)\subset L(\wt\gamma')$ or $L(\wt\gamma')\subset L(\wt\gamma)$. Beware: the relation ``having the same orientation'' is not transitive.

\subsubsection*{Bounded deviations transverse to the lamination}

The following result is a higher genus counterpart of a theorem of Silva Salomão and Tal \cite{zbMATH07488214}, that states that for a torus homeomorphism having a segment with irrational slope and containing a rational point as a rotation set, there are bounded deviations in the direction orthogonal to the rotation set\footnote{Under a non-wandering hypothesis for \cite{zbMATH07488214}, that should be possible to remove using the same proof strategy as ours.}. 

\begin{theo}\label{ThmBndDevIrrat}
Let $S$ be a compact boundaryless hyperbolic surface and $f\in \Homeo_0(S)$.
Let $\cl_i$ be a minimal non-closed class of $\M(f)$, $\Lambda_i$ the associated lamination and $S_i$ the associated surface.

Then for any $\mu\in\cl_i$ and any $\mu$-typical point $z\in S$, there exists $C_0>0$ and a set $X\subset S$ with empty interior and satisfying $i_*(\pi_1(S_i)) \subset i_*(\pi_1(X))$ such that the following is false (see Figure~\ref{FigThmBndDevIrrat}):
There exists $\wt y_0$ belonging to a connected component $\wt {\mathcal X}^-$ of the complement of $\wt X$ (the union of the lifts of $X$) and $n_0\in\N$ such that $\wt f^{n_0}(\wt y_0)$ belongs to a connected component $\wt {\mathcal X}^+$ of the complement $\wt X$, with $\wt {\mathcal X}^- \subset L(V_{C_0}(\wt\gamma_{\wt z}))$ and $\wt {\mathcal X}^+ \subset R(V_{C_0}(\wt\gamma_{\wt z}))$.
\end{theo}

\begin{figure}
\begin{center}

\tikzset{every picture/.style={line width=0.75pt}} 

\begin{tikzpicture}[x=0.75pt,y=0.75pt,yscale=-1,xscale=1]

\draw [color={rgb, 255:red, 208; green, 2; blue, 27 }  ,draw opacity=1 ] [dash pattern={on 0.84pt off 2.51pt}]  (280.77,154.32) .. controls (296.38,152.55) and (322.96,195.85) .. (309.55,199.26) ;
\draw [color={rgb, 255:red, 155; green, 155; blue, 155 }  ,draw opacity=1 ] [dash pattern={on 0.84pt off 2.51pt}]  (279.37,85.37) .. controls (266.54,84.62) and (262.99,136.4) .. (279.18,137.19) ;
\draw  [fill={rgb, 255:red, 74; green, 74; blue, 74 }  ,fill opacity=0.1 ][line width=1.5]  (49.43,143.1) .. controls (49.8,54.8) and (153.93,92.6) .. (199.43,93.1) .. controls (244.93,93.6) and (349.3,54.3) .. (349.43,143.1) .. controls (349.55,231.89) and (252.3,191.8) .. (199.43,193.1) .. controls (146.55,194.39) and (49.05,231.39) .. (49.43,143.1) -- cycle ;
\draw  [draw opacity=0][fill={rgb, 255:red, 74; green, 144; blue, 226 }  ,fill opacity=0.15 ] (199.43,93.1) .. controls (235.27,93.2) and (261.27,83.6) .. (284.04,85.63) .. controls (341.87,84) and (349.65,122.03) .. (349.04,146.63) .. controls (346.27,175.8) and (341.67,203.2) .. (278.54,201.13) .. controls (249.87,197.8) and (218.27,192.2) .. (199.43,193.1) .. controls (176.07,194.2) and (175.47,95) .. (199.43,93.1) -- cycle ;
\draw  [color={rgb, 255:red, 189; green, 16; blue, 224 }  ,draw opacity=0.2 ][fill={rgb, 255:red, 189; green, 16; blue, 224 }  ,fill opacity=0.1 ] (478.72,51.5) .. controls (503.61,73.5) and (507.39,192.39) .. (478.28,230.61) .. controls (443.61,192.39) and (446.5,75.94) .. (478.72,51.5) -- cycle ;
\draw  [draw opacity=0][fill={rgb, 255:red, 176; green, 248; blue, 28 }  ,fill opacity=0.5 ] (388.21,173.59) .. controls (382.67,160.21) and (381.84,133.66) .. (384.51,124.04) .. controls (388.58,125.29) and (389.58,126.21) .. (391.83,124.04) .. controls (394.33,123.71) and (395.25,122.71) .. (396.25,121.13) .. controls (400.17,126.54) and (414.42,120.96) .. (417.42,121.54) .. controls (412.08,133.96) and (425.67,136.88) .. (429.75,145.63) .. controls (417.17,146.88) and (411.33,156.21) .. (409.33,167.04) .. controls (402.33,165.38) and (398.08,170.29) .. (397.33,173.88) .. controls (393.58,170.88) and (391.25,172.38) .. (388.21,173.59) -- cycle ;
\draw [color={rgb, 255:red, 189; green, 16; blue, 224 }  ,draw opacity=1 ]   (478.72,51.5) .. controls (474.28,127.5) and (473.39,151.06) .. (478.28,230.61) ;
\draw [color={rgb, 255:red, 155; green, 155; blue, 155 }  ,draw opacity=1 ]   (431.08,149.46) .. controls (428.51,134.36) and (409.9,137.59) .. (419.13,117.28) ;
\draw [color={rgb, 255:red, 155; green, 155; blue, 155 }  ,draw opacity=1 ]   (408.6,171.6) .. controls (410.21,156.97) and (420.05,142.97) .. (434.36,146.21) ;
\draw [color={rgb, 255:red, 155; green, 155; blue, 155 }  ,draw opacity=1 ]   (397.91,178.96) .. controls (395.02,168.27) and (407.59,163.9) .. (413.44,169.28) ;
\draw [color={rgb, 255:red, 155; green, 155; blue, 155 }  ,draw opacity=1 ]   (400.05,178.67) .. controls (398.51,172.97) and (394.21,171.28) .. (388.21,173.59) ;
\draw [color={rgb, 255:red, 155; green, 155; blue, 155 }  ,draw opacity=1 ]   (394.97,118.97) .. controls (400.58,128.81) and (410.36,118.67) .. (424.67,121.9) ;
\draw [color={rgb, 255:red, 155; green, 155; blue, 155 }  ,draw opacity=1 ]   (388.51,120.36) .. controls (389.44,126.51) and (396.67,123.59) .. (398.21,116.97) ;
\draw [color={rgb, 255:red, 155; green, 155; blue, 155 }  ,draw opacity=1 ]   (383.59,123.59) .. controls (388.05,125.28) and (391.74,126.82) .. (393.28,120.21) ;
\draw  [draw opacity=0][fill={rgb, 255:red, 248; green, 231; blue, 28 }  ,fill opacity=0.5 ] (553.84,103.26) .. controls (560.45,116.15) and (563.43,142.54) .. (561.55,152.35) .. controls (557.39,151.43) and (556.32,150.6) .. (554.25,152.94) .. controls (551.79,153.48) and (550.95,154.55) .. (550.09,156.21) .. controls (545.74,151.13) and (531.99,157.85) .. (528.96,157.51) .. controls (533.26,144.7) and (519.49,142.9) .. (514.71,134.51) .. controls (527.15,132.24) and (532.2,122.47) .. (533.32,111.51) .. controls (540.43,112.6) and (544.26,107.35) .. (544.72,103.72) .. controls (548.7,106.41) and (550.91,104.72) .. (553.84,103.26) -- cycle ;
\draw [color={rgb, 255:red, 155; green, 155; blue, 155 }  ,draw opacity=1 ]   (513.07,130.8) .. controls (516.86,145.64) and (535.15,140.91) .. (527.6,161.9) ;
\draw [color={rgb, 255:red, 155; green, 155; blue, 155 }  ,draw opacity=1 ]   (533.67,106.9) .. controls (533.27,121.61) and (524.59,136.36) .. (510.07,134.31) ;
\draw [color={rgb, 255:red, 155; green, 155; blue, 155 }  ,draw opacity=1 ]   (543.73,98.7) .. controls (547.48,109.12) and (535.31,114.5) .. (529.05,109.61) ;
\draw [color={rgb, 255:red, 155; green, 155; blue, 155 }  ,draw opacity=1 ]   (541.62,99.16) .. controls (543.62,104.71) and (548.05,106.05) .. (553.84,103.26) ;
\draw [color={rgb, 255:red, 155; green, 155; blue, 155 }  ,draw opacity=1 ]   (551.53,158.25) .. controls (545.15,148.9) and (536.22,159.8) .. (521.7,157.75) ;
\draw [color={rgb, 255:red, 155; green, 155; blue, 155 }  ,draw opacity=1 ]   (557.86,156.34) .. controls (556.44,150.28) and (549.47,153.78) .. (548.48,160.5) ;
\draw [color={rgb, 255:red, 155; green, 155; blue, 155 }  ,draw opacity=1 ]   (561.55,152.35) .. controls (556.97,151.02) and (554.12,150.16) .. (553.12,156.88) ;
\draw  [line width=1.5]  (382.17,141.17) .. controls (382.17,91.46) and (422.46,51.17) .. (472.17,51.17) .. controls (521.87,51.17) and (562.17,91.46) .. (562.17,141.17) .. controls (562.17,190.87) and (521.87,231.17) .. (472.17,231.17) .. controls (422.46,231.17) and (382.17,190.87) .. (382.17,141.17) -- cycle ;
\draw [color={rgb, 255:red, 208; green, 2; blue, 27 }  ,draw opacity=1 ]   (418.5,140.79) .. controls (458.5,110.79) and (493.07,179.87) .. (533.07,149.87) ;
\draw [shift={(533.07,149.87)}, rotate = 323.13] [color={rgb, 255:red, 208; green, 2; blue, 27 }  ,draw opacity=1 ][fill={rgb, 255:red, 208; green, 2; blue, 27 }  ,fill opacity=1 ][line width=0.75]      (0, 0) circle [x radius= 1.34, y radius= 1.34]   ;
\draw [shift={(418.5,140.79)}, rotate = 323.13] [color={rgb, 255:red, 208; green, 2; blue, 27 }  ,draw opacity=1 ][fill={rgb, 255:red, 208; green, 2; blue, 27 }  ,fill opacity=1 ][line width=0.75]      (0, 0) circle [x radius= 1.34, y radius= 1.34]   ;
\draw  [fill={rgb, 255:red, 255; green, 255; blue, 255 }  ,fill opacity=1 ][line width=1.5]  (139.29,149.77) .. controls (127.81,153.89) and (116.02,156.31) .. (99.7,149.57) .. controls (109.68,133.05) and (129.57,132.94) .. (139.29,149.77) -- cycle ;
\draw [line width=1.5]    (89.4,143) .. controls (104.48,157.72) and (135.48,157.39) .. (149.4,143) ;

\draw  [fill={rgb, 255:red, 255; green, 255; blue, 255 }  ,fill opacity=1 ][line width=1.5]  (299.49,150.37) .. controls (288.01,154.49) and (276.22,156.91) .. (259.9,150.17) .. controls (269.88,133.65) and (289.77,133.54) .. (299.49,150.37) -- cycle ;
\draw [line width=1.5]    (249.6,143.6) .. controls (264.68,158.32) and (295.68,157.99) .. (309.6,143.6) ;

\draw [color={rgb, 255:red, 74; green, 144; blue, 226 }  ,draw opacity=1 ]   (199.43,93.1) .. controls (217.49,93.53) and (220.16,193.31) .. (199.43,193.1) ;
\draw [color={rgb, 255:red, 74; green, 144; blue, 226 }  ,draw opacity=1 ] [dash pattern={on 0.84pt off 2.51pt}]  (199.43,93.1) .. controls (175.63,93.63) and (175.06,193.63) .. (199.43,193.1) ;
\draw  [line width=1.5]  (49.43,143.1) .. controls (49.8,54.8) and (153.93,92.6) .. (199.43,93.1) .. controls (244.93,93.6) and (349.3,54.3) .. (349.43,143.1) .. controls (349.55,231.89) and (252.3,191.8) .. (199.43,193.1) .. controls (146.55,194.39) and (49.05,231.39) .. (49.43,143.1) -- cycle ;
\draw [color={rgb, 255:red, 155; green, 155; blue, 155 }  ,draw opacity=1 ]   (279.37,85.37) .. controls (293.88,84.62) and (294.1,137.29) .. (279.18,137.19) ;
\draw  [color={rgb, 255:red, 155; green, 155; blue, 155 }  ,draw opacity=1 ] (246.54,109.51) .. controls (266.54,99.51) and (292.1,95.51) .. (322.07,109.5) .. controls (352.05,123.49) and (329.43,162.18) .. (322.07,169.5) .. controls (314.71,176.82) and (262.99,196.4) .. (242.99,166.4) .. controls (222.99,136.4) and (226.54,119.51) .. (246.54,109.51) -- cycle ;
\draw [color={rgb, 255:red, 208; green, 2; blue, 27 }  ,draw opacity=1 ]   (280.77,154.32) .. controls (253.97,156.72) and (72.28,196.49) .. (69.08,141.29) .. controls (65.88,86.09) and (221.73,104.2) .. (277.88,106.09) .. controls (334.03,107.97) and (339.48,151.29) .. (307.08,167.69) .. controls (274.68,184.09) and (235.56,148.03) .. (241.88,132.09) .. controls (248.2,116.14) and (305.62,124.98) .. (306.28,124.49) ;
\draw [shift={(306.28,124.49)}, rotate = 323.13] [color={rgb, 255:red, 208; green, 2; blue, 27 }  ,draw opacity=1 ][fill={rgb, 255:red, 208; green, 2; blue, 27 }  ,fill opacity=1 ][line width=0.75]      (0, 0) circle [x radius= 1.34, y radius= 1.34]   ;
\draw [color={rgb, 255:red, 208; green, 2; blue, 27 }  ,draw opacity=1 ]   (265.63,176.2) .. controls (304.1,176.67) and (298.57,201.02) .. (309.55,199.26) ;
\draw [shift={(265.63,176.2)}, rotate = 0.7] [color={rgb, 255:red, 208; green, 2; blue, 27 }  ,draw opacity=1 ][fill={rgb, 255:red, 208; green, 2; blue, 27 }  ,fill opacity=1 ][line width=0.75]      (0, 0) circle [x radius= 1.34, y radius= 1.34]   ;

\draw (415,143.24) node [anchor=north west][inner sep=0.75pt]  [font=\small,color={rgb, 255:red, 178; green, 3; blue, 25 }  ,opacity=1 ]  {$\wt{y}_{0}$};
\draw (535,156) node [anchor=north] [inner sep=0.75pt]  [font=\small,color={rgb, 255:red, 178; green, 3; blue, 25 }  ,opacity=1 ]  {$\wt{f}^{n_{0}}(\wt{y}_{0})$};
\draw (474.56,81.87) node [anchor=south east] [inner sep=0.75pt]  [font=\small,color={rgb, 255:red, 177; green, 5; blue, 212 }  ,opacity=1 ]  {$V_{C}(\wt{\gamma })$};
\draw (411.44,115.48) node [anchor=south] [inner sep=0.75pt]  [font=\small,color={rgb, 255:red, 143; green, 210; blue, 5 }  ,opacity=1 ]  {$\wt{\mathcal X}^{-}$};
\draw (537.44,104.49) node [anchor=south east] [inner sep=0.75pt]  [font=\small,color={rgb, 255:red, 186; green, 173; blue, 0 }  ,opacity=1 ]  {$\wt{\mathcal X}^{+}$};
\draw (231.62,124.09) node [anchor=south east] [inner sep=0.75pt]  [color={rgb, 255:red, 155; green, 155; blue, 155 }  ,opacity=1 ]  {$X$};
\draw (307.41,128) node [anchor=north] [inner sep=0.75pt]  [font=\small,color={rgb, 255:red, 178; green, 3; blue, 25 }  ,opacity=1 ]  {$y_{0}$};
\draw (262.79,178) node [anchor=north] [inner sep=0.75pt]  [font=\small,color={rgb, 255:red, 178; green, 3; blue, 25 }  ,opacity=1 ]  {${f}^{n_{0}}(y_{0})$};

\end{tikzpicture}
\caption{The statement of Theorem~\ref{ThmBndDevIrrat}. The sets $\wt X^-$ and $\wt X^+$ are complement of the lift of $X$. Theorem~\ref{ThmBndDevIrrat} prevents from having a trajectory such as the one of $y_0$.}\label{FigThmBndDevIrrat}
\end{center}
\end{figure}
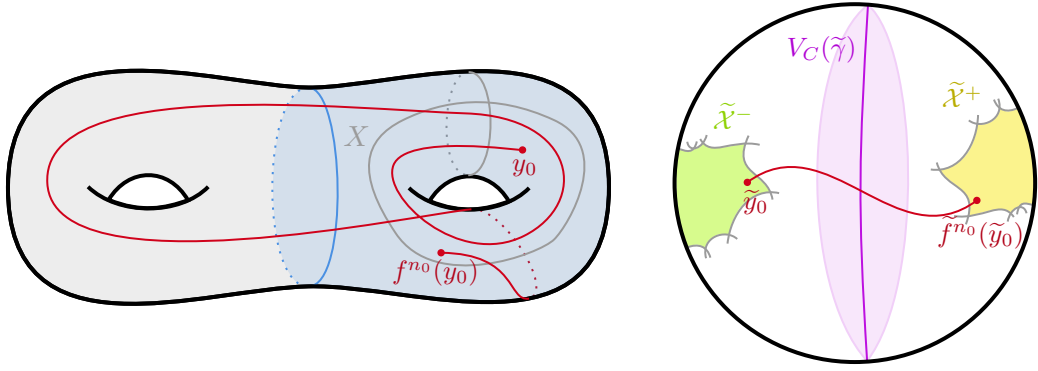

\begin{corollary}\label{CoroBndedDevIrrat}
Let $S$ be a compact boundaryless hyperbolic surface and $f\in \Homeo_0(S)$. Let $\cl_i$ be a minimal non-closed class of $\M(f)$, $\Lambda_i$ the associated lamination and $S_i$ the associated surface. Let $\gamma$ be a boundary component of this surface --- by \cite[Section 6.2]{alepablo} this is a closed geodesic. 
Then there exists $N>0$ such that an orbit of $\wt f$ cannot cross geometrically more than $N$ different lifts of $\gamma$.
\end{corollary}

This corollary, combined with \cite{paper1PAF}, implies a criterion of elliptic action on the fine curve graph for elements of $\Homeo_0(S)$ in terms of the ergodic rotation set \cite[Corollary C]{paper1PAF}. 

The following is an improvement of Theorem~\ref{ThmBndDevIrrat} under the additional assumption that the set of contractible fixed points of $f$ is inessential.

\begin{theo}\label{ThmBndDevIrrat2}
Let $S$ be a compact boundaryless hyperbolic surface and $f\in \Homeo_0(S)$ such that the set of contractible fixed points of $f$ is inessential.
Let $\cl_i$ be a minimal non-closed class of $\M(f)$.

Then for any $\mu\in\cl_i$ and any $\mu$-typical point $z\in S$, there exists ${\overline{C}}>0$ such that the following is false:
There exists $\wt y_0\in L(V_{{\overline{C}}}(\wt\gamma_{\wt z}))$ and $n_0\in \N$ such that $\wt f^{n_0}(\wt y_0)\in R(V_{{\overline{C}}}(\wt\gamma_{\wt z}))$.

Moreover, for any $n\in\Z$, we have $\wt f^n(\wt z)\in V_{\overline C}(\wt\gamma_{\wt z})$.
\end{theo}

Note that this theorem improves the case of surface flows with a finite number of singularities, that was studied in the 90's \cite[Theorem~2]{anosovweil1}, \cite[Theorem~3.1]{anosovweil2}, \cite[Theorem~4.11]{zbMATH06859894}.

\begin{rem}
The hypothesis about inessentialness of the fixed point set is necessary in Theorem~D: One can easily produce counterexamples, even for area-preserving flows where the result does not hold, by adapting Koropecki-Tal technique \cite{zbMATH06345227, pa}.
\end{rem}

\subsubsection*{Bounded deviations in the direction of the lamination}

Our next result is a higher genus counterpart of a result of Liu and Tal \cite{zbMATH07867510} stating bounded deviation in the direction of the rotation set. More precisely, let $f\in\Homeo_0(\T^2)$ having a segment with irrational slope and containing a rational point as a rotation set (by \cite[Theorem~C]{lct1}, this rational point is one of the ends of this segment).
Up to taking a power of $f$ and choosing an appropriate lift, we can suppose that $\rot(\wt f) = \{t v \mid t\in [0,1]\}$ for some $v\in \R^2\setminus\{0\}$ having irrational slope. The main theorem of \cite{zbMATH07867510} is that there exists $C>0$ such that for any $\wt x\in \R^2$ and any $n\in\Z$, we have  
\[\langle \wt f^n(\wt x) - \wt x,\, v\rangle \ge -C.\]
In higher genus, the global idea is the same: orbits cannot fellow travel in the opposite direction of the oriented lamination.
More formally, the result is stated in terms of projections on the tracking geodesic.
Let $\pr_{\wt\gamma}$ denote the orthogonal projection on the geodesic $\wt\gamma$. 

\begin{theo}\label{TheoBndedDirLam}
Let $z$ be a typical point for an ergodic measure $\mu$ that belongs to the class of a minimal non-closed lamination. Let $\wt z$ be a lift of $z$ to $\wt S$ and $\wt\gamma_{\wt z}$ the oriented tracking geodesic of $\wt z$. 

Then there exists $C>0$ such that for any $\wt y_0\in \wt S$ and any $n\in\N$, we have (see Figure~\ref{FigTheoBndedDirLam})
\[\pr_{\wt\gamma_{\wt z}}\big(\wt f^n(\wt y_0)\big) - \pr_{\wt\gamma_{\wt z}}(\wt y_0) \ge -C.\]
\end{theo}

By passing to the limit among the minimal lamination $\Lambda_i$ associated to the class $\cl_i$ of the measure $\mu$, the result remains true by replacing the geodesic $\wt\gamma_{\wt z}$ by any lift of a geodesic of $\Lambda_i$. 

\begin{figure}
\begin{center}

\tikzset{every picture/.style={line width=0.75pt}} 

\begin{tikzpicture}[x=0.75pt,y=0.75pt,yscale=-1,xscale=1]

\draw  [draw opacity=0][fill={rgb, 255:red, 208; green, 2; blue, 27 }  ,fill opacity=0.08 ] (369.8,95.68) .. controls (345.51,124.54) and (310.94,123.68) .. (292.94,84.82) .. controls (318.08,75.4) and (348.65,81.11) .. (369.8,95.68) -- cycle ;
\draw [color={rgb, 255:red, 208; green, 2; blue, 27 }  ,draw opacity=1 ]   (292.94,84.82) .. controls (309.22,121.4) and (344.94,126.54) .. (369.8,95.68) ;
\draw  [draw opacity=0][fill={rgb, 255:red, 208; green, 2; blue, 27 }  ,fill opacity=0.08 ] (366.65,246.54) .. controls (340.94,261.11) and (326.08,263.4) .. (295.8,256.82) .. controls (312.08,220.54) and (343.42,221.77) .. (366.65,246.54) -- cycle ;
\draw [color={rgb, 255:red, 208; green, 2; blue, 27 }  ,draw opacity=1 ]   (295.8,256.82) .. controls (310.37,222.25) and (343.8,220.82) .. (366.65,246.54) ;
\draw [color={rgb, 255:red, 80; green, 0; blue, 139 }  ,draw opacity=1 ]   (330.4,80.8) .. controls (322.65,144.82) and (323.22,190.25) .. (331.08,258.71) ;
\draw [shift={(324.88,165.58)}, rotate = 89.64] [fill={rgb, 255:red, 80; green, 0; blue, 139 }  ,fill opacity=1 ][line width=0.08]  [draw opacity=0] (8.04,-3.86) -- (0,0) -- (8.04,3.86) -- (5.34,0) -- cycle    ;
\draw [color={rgb, 255:red, 74; green, 144; blue, 226 }  ,draw opacity=1 ]   (330.4,80.8) .. controls (336.04,90.14) and (322.71,83.92) .. (323.6,93.92) .. controls (324.49,103.92) and (338.71,105.7) .. (338.04,125.47) .. controls (337.38,145.25) and (307.16,152.59) .. (304.27,183.03) .. controls (301.38,213.47) and (315.38,228.36) .. (338.04,220.8) .. controls (360.71,213.25) and (327.38,247.69) .. (328.27,251.91) .. controls (329.16,256.14) and (335.16,253.91) .. (331.08,258.71) ;
\draw [color={rgb, 255:red, 74; green, 144; blue, 226 }  ,draw opacity=1 ]   (308.89,212.73) .. controls (302.57,200.83) and (302.67,187.04) .. (306.04,174.1) ;
\draw [shift={(303.76,190.46)}, rotate = 87.39] [fill={rgb, 255:red, 74; green, 144; blue, 226 }  ,fill opacity=1 ][line width=0.08]  [draw opacity=0] (8.04,-3.86) -- (0,0) -- (8.04,3.86) -- (5.34,0) -- cycle    ;
\draw    (331.02,121.31) .. controls (329.22,152.74) and (329.03,175.59) .. (333.89,223.66) ;
\draw [shift={(334.19,226.64)}, rotate = 264.11] [fill={rgb, 255:red, 0; green, 0; blue, 0 }  ][line width=0.08]  [draw opacity=0] (7.14,-3.43) -- (0,0) -- (7.14,3.43) -- (4.74,0) -- cycle    ;
\draw [shift={(331.19,118.31)}, rotate = 93.4] [fill={rgb, 255:red, 0; green, 0; blue, 0 }  ][line width=0.08]  [draw opacity=0] (7.14,-3.43) -- (0,0) -- (7.14,3.43) -- (4.74,0) -- cycle    ;
\draw [color={rgb, 255:red, 155; green, 155; blue, 155 }  ,draw opacity=1 ]   (349.06,242.24) .. controls (373.39,190.91) and (403.73,167.57) .. (399.73,151.91) .. controls (395.73,136.24) and (346.06,129.57) .. (354.73,99.91) ;
\draw [shift={(354.73,99.91)}, rotate = 286.28] [color={rgb, 255:red, 155; green, 155; blue, 155 }  ,draw opacity=1 ][fill={rgb, 255:red, 155; green, 155; blue, 155 }  ,fill opacity=1 ][line width=0.75]      (0, 0) circle [x radius= 2.01, y radius= 2.01]   ;
\draw [shift={(373.3,200.36)}, rotate = 304.57] [fill={rgb, 255:red, 155; green, 155; blue, 155 }  ,fill opacity=1 ][line width=0.08]  [draw opacity=0] (7.14,-3.43) -- (0,0) -- (7.14,3.43) -- (4.74,0) -- cycle    ;
\draw [shift={(373.9,132.03)}, rotate = 212.47] [fill={rgb, 255:red, 155; green, 155; blue, 155 }  ,fill opacity=1 ][line width=0.08]  [draw opacity=0] (7.14,-3.43) -- (0,0) -- (7.14,3.43) -- (4.74,0) -- cycle    ;
\draw [shift={(349.06,242.24)}, rotate = 295.36] [color={rgb, 255:red, 155; green, 155; blue, 155 }  ,draw opacity=1 ][fill={rgb, 255:red, 155; green, 155; blue, 155 }  ,fill opacity=1 ][line width=0.75]      (0, 0) circle [x radius= 2.01, y radius= 2.01]   ;
\draw  [line width=1.5]  (230,170) .. controls (230,120.29) and (270.29,80) .. (320,80) .. controls (369.71,80) and (410,120.29) .. (410,170) .. controls (410,219.71) and (369.71,260) .. (320,260) .. controls (270.29,260) and (230,219.71) .. (230,170) -- cycle ;

\draw (332.36,169.2) node [anchor=west] [inner sep=0.75pt]  [xscale=1.2,yscale=1.2]  {$C$};
\draw (325.04,128.79) node [anchor=east] [inner sep=0.75pt]  [color={rgb, 255:red, 80; green, 0; blue, 139 }  ,opacity=1 ,xscale=1.2,yscale=1.2]  {$\wt{\gamma }_{\wt{z}}$};
\draw (301.79,195.2) node [anchor=east] [inner sep=0.75pt]  [color={rgb, 255:red, 74; green, 144; blue, 226 }  ,opacity=1 ,xscale=1.2,yscale=1.2]  {$I_{\wt{\F}}^{\Z}(\wt{z})$};
\draw (414.19,144.8) node [anchor=west] [inner sep=0.75pt]  [color={rgb, 255:red, 101; green, 101; blue, 101 }  ,opacity=1 ,xscale=1.2,yscale=1.2]  {$I_{\wt{\F}}^{[ 0,n_{0}]}(\wt{y}_{0})$};
\draw (233.19,133.6) node [anchor=east] [inner sep=0.75pt]  [color={rgb, 255:red, 255; green, 255; blue, 255 }  ,opacity=1 ,xscale=1.2,yscale=1.2]  {$I_{\wt{\F}}^{[ 0,n_{0}]}(\wt{y}_{0})$};
\end{tikzpicture}

\caption{Theorem~\ref{TheoBndedDirLam} asserts that such a grey orbit cannot exist: an orbit cannot meet first the top red domain and then the bottom red domain (delimited by geodesics that are orthogonal to $\wt\gamma_{\wt z}$ and at a distance $C$).}\label{FigTheoBndedDirLam}
\end{center}
\end{figure}
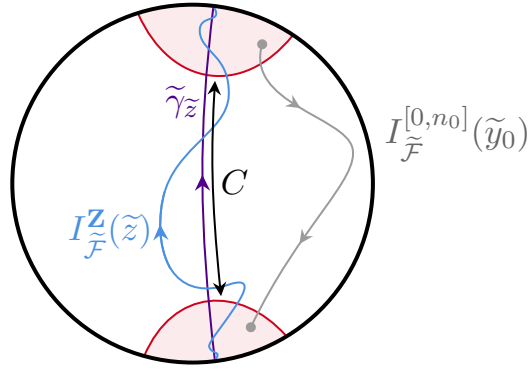

Let us finish with a corollary to Theorem~\ref{TheoBndedDirLam}.

\begin{corollary}\label{CoroBndedDirLam}
Suppose that $f\in\Homeo_0(S)$ and that there exist $\rho\in\rote(f) \setminus \R H_1(S,\Q)$ and a neighbourhood $V$ of $\rho$ such that $\rote (f) \cap V \subset \R \rho$. 
Then $\rote (f) \cap \R \rho \subset \R_+\rho$. 

If there exist $\rho\in\rot(f) \setminus \R H_1(S,\Q)$ and a neighbourhood $V$ of $\rho$ such that $\rot (f) \cap V \subset \R \rho$, then $\rot (f)$ is included in a half-space $\{\ell>0\}$, where $\ell\in H_1(S,\Q)^*$.
\end{corollary}

This gives a constraint about the possibilities for the shape of a rotation set, similar to the one of \cite[Theorem C]{lct1}. For instance, this shows that there is no $f\in \Homeo_0(S)$ whose rotation set is a segment with irrational direction and containing 0 in its interior. 

\bigskip

The collection of results in this paper will also be used in the forthcoming work \cite{GLCPT} to get pseudo-foliations with irrational direction on the surface for area-preserving homeomorphisms without horseshoe. They could also be used more generally to get invariant sub-surfaces associated to minimal non-closed classes, possibly under some non-wandering or measure preserving condition. 
They also could be used for studying the rotational dynamics of $C^r$ ($r\ge 1$) generic diffeomorphisms of $S$ as it was done in \cite{Addas-Zanata_2022} for the torus case.

\subsection{Proof strategy and plan of the paper}

The next section is devoted to the setting of some notations (the reader is advised to consult \cite{paper1PAF} for more precise preliminaries, in particular about the forcing theory of Le Calvez and Tal \cite{lct1, lct2}).

The first result we prove is the fact that laminations associated to classes are orientable (Theorem~\ref{SecLaminOrient}), it will occupy the whole Section~\ref{SecLaminOrient}. The first step is to prove in Lemma~\ref{LemAlphaSimple} that given a point $z$ that is typical for a measure of a minimal non-closed class, the transverse trajectory $I^\Z_{\wt \F}(\wt z)$ does not meet any leaf of $\wt\F$ twice or more (and in particular is simple). Denoting $\wt B$ the set of leaves met by this trajectory, we prove that the different images of $\wt B$ under deck transformations are well ordered (Lemmas~\ref{LemAlphaNotCross} and \ref{LemBandDisjoint}). Theorem~\ref{SecLaminOrient} then follows rather easily from these facts.

The next section is way longer and dedicated to the proof of results of bounded deviations in directions crossing the one of the lamination (Section~\ref{SecBoundCross}). The first step is to prove that the path $I^\Z_{\wt\F}(\wt z)$ cannot intersect $\wt\F$-transversally any other path (Proposition~\ref{LastPropBndDevIrrat}), improving the already proved fact that the path $I^\Z_{\wt\F}(\wt z)$ has no $\wt\F$-transverse intersection (Proposition~\ref{PropPasInterTrans}); it occupies Subsection~\ref{SubSecNoTrans}.

Subsection~\ref{SubSecEssen} uses the concept of fully essential points of \cite{zbMATH06294042, zbMATH06908424} to get an $f$-invariant open set $A_\varep$, made of points having an iterate coming $\varep$-close to the typical point $z$, and such that ${i}_*(\pi_1(S_\Lambda))\subset {i}_*(\pi_1(A_\varep))$ (Lemma~\ref{LemEssentialPoints}). The set $X$ of Theorem~\ref{ThmBndDevIrrat} is then taken as a skeleton of the set $A_\varep$. 
We then prove that the hypotheses of Theorem~\ref{ThmBndDevIrrat} implies that there is a point $\wt y\in \wt X$ on the left of $V_{C_0}(\check\gamma_{\check z})$ whose image $\wt f^{n_0}(\wt y)\in \wt X$ is on the right of $V_{C_0}(\wt\gamma_{\wt z})$. The fact that the distance to $\wt\gamma_{\wt z}$ is large enough combined with a compactness argument ensures that some iterate $\wt f^{n_-}(\wt y)$ is close to\footnote{Recall that $X$ is made of points whose iterates come close to $z$} $T_-\wt f^{m_-}(\wt z)$ for some $T_-\in\G\setminus\{\Id\}$; similarly $\wt f^{n_0+n_+}(\wt y)$ is close to $T_+\wt f^{m_+}(\wt z)$. The fact that the transverse trajectories of $\wt z$ and $T_-\wt z$ are neither equivalent nor accumulate one in the other (and the same for $\wt z$ and $T_+\wt z$), together with compactness arguments, implies that the trajectories of $\wt z$ and $\wt y$ have an $\wt \F$-transverse intersection (Proposition~\ref{PropExistInterTransIrrat}) contradicting Proposition~\ref{PropPasInterTrans}.
This whole strategy is inspired by the strategy of \cite{zbMATH07488214}.
 
If the fixed point set of $f$ is inessential, Theorem~\ref{ThmBndDevIrrat} can be improved into Theorem~\ref{ThmBndDevIrrat2}: this is the object of Section~\ref{SubSecThmBndDevIrrat2}. It uses the same strategy as before, but we have to carefully analyse what happens to the connected components of the complement of the set $X$. We use the hypothesis about inessentialness of fixed points twice: first through \cite[Lemma~2.9]{paper1PAF}, and second by the fact that orbits with big deviations have to visit twice some small open set far away from fixed points (Lemma~\ref{LemMeetPhiTwice}). We advise the reader to read first Case~\ref{Ca4} of the proof of Theorem~\ref{ThmBndDevIrrat2} that corresponds to the case where everything goes wrong in the previous strategy of the proof of Theorem~\ref{ThmBndDevIrrat}.

Finally, we prove Theorem~\ref{TheoBndedDirLam} in Section~\ref{SecBoundDirec}. The idea is to build two approximations of the trajectory of a $\mu$-typical point $z$, one whose associated geodesic (that lift closed geodesics) cross the tracking geodesic $\wt\gamma_{\wt z}$ from left to right, the second one from right to left (Subsection~\ref{SubsecApprox}). A crucial ingredient here is Atkinson theorem. The proof of Theorem~\ref{TheoBndedDirLam} (made in Subsection~\ref{SubsecTheoBndedDirLam}) then follows from the fact that the tracking geodesic of $\wt z$ crosses the axes of these approximations syndetically (Lemma~\ref{LemInterAlphaiEmpty}).

\section{Notations and preliminaries}

We suppose the reader has a certain familiarity with the forcing theory \cite{lct1, lct2} and its classical notations. We refer to Section 2 of \cite{paper1PAF} for more details about notations and preliminaries. 

Let us recall the main result the forcing theory is based on, obtained as a combination of \cite{lecalvezfoliations} and \cite{bguin2016fixed}. 

Given an isotopy $I$ from the identity to $f$, we denote $\dom(I) := S \backslash \Fix(I)$.

\begin{theorem}\label{ThExistIstop}
Let $S$ be a surface and $f\in\Homeo_0(S)$.
Then there exist an identity isotopy $I$ for $f$ and a transverse topological oriented singular foliation $\F$ of $S$ with $\dom(\F) = \dom(I)$, such that:
For any $z\in \dom(\F)$, there exists an $\F$-transverse path denoted by $\big(I_\F^t(z)\big)_{t\in[0,1]}$, linking $z$ to $f(z)$, that is homotopic in $\dom(\F)$, relative to its endpoints, to the arc $(I^t(z))_{t\in[0,1]}$.
\end{theorem}

This allows us to define the path $I_{\F}^\Z (x)$ as the concatenation of the paths $\big(I_\F^t(f^n(z))\big)_{t\in[0,1]}$ for $n\in\Z$.

Recall that $\wt S$ is the universal cover of the surface $S$. We will denote $\G$ the group of deck transformations of the cover $\wt s\to S$. 
We will denote $\wh\dom(I)$ the universal cover of $\dom(I)$, and $\wh f$ the lift of $f$ to $\wh\dom(I)$.

\subsection{Tracking geodesics}\label{SecTrack}

All our results will involve the notion of tracking geodesic introduced in \cite{alepablo}. Let us recall the main results, already stated in \cite{paper1PAF} --- we refer to \cite{alepablo} for more details.

Let $S$ be a compact boundaryless orientable connected surface, and denote $\Homeo_0(S)$ the set of homeomorphisms of $S$ that are homotopic to the identity. Let $f\in \Homeo_0(S)$. 
We denote $\wt S$ the universal cover of the surface $S$ and $\wt f$ the (unique) lift of $f$ to $\wt S$ that commutes with deck transformations. The surface $S$ is equipped with a metric of constant curvature $-1$ that induces a distance $d$, which lifts to a distance on $\wt S$ that we also denote by $d$. 

Denote $\Me(f)$ the set of $f$-invariant ergodic Borel probability measures. 
The following is a direct consequence of Kingman's subadditive ergodic theorem \cite[Lemma~1.6]{alepablo}.

\begin{lemma}\label{LemErgoRotSpeed}
Let $\mu\in\Me(f)$. Then there exists a constant \(\vartheta_\mu\in\R_+\) --- called the \emph{rotation speed} of $\mu$ --- such that
\[\lim\limits_{n \to +\infty}\frac{1}{n}d\big(\wt z, \wt f^n(\wt z)\big) = \lim\limits_{n \to +\infty}\frac{1}{n}d\big(\wt z, \wt f^{-n}(\wt z)\big) = \vartheta_\mu,\]
for \(\mu\)-almost every point \(z \in S\).	
\end{lemma}

We denote by \(\M(f)\) the set of $\mu\in\Me(f)$ such that $\vartheta_\mu>0$. As usual, we will parametrise geodesics by arclength. Points that are typical for some ergodic measure of $\M(f)$ follow a so-called \emph{tracking geodesic} \cite[Theorem~B]{alepablo}.

\begin{theorem}\label{DefTrackGeod}
Let $\mu\in\M(f)$. Then $\mu$-a.e.~$z\in S$ admits a \emph{tracking geodesic} $\gamma$: for each lift \(\wt z\) of \(z\), there exists a lift $\wt \gamma$ of $\gamma$ such that:
\begin{equation}\label{eq:trackingequation}
\lim\limits_{n \to +\infty}\frac{1}{n}d\big(\wt f^n(\wt z), \wt \gamma(n \vartheta_\mu)\big) = \lim\limits_{n \to +\infty}\frac{1}{n}d\big(\wt f^{-n}(\wt z), \wt \gamma(-n \vartheta_\mu)\big) = 0.
\end{equation}	
\end{theorem}

The geodesic associated to a $\mu$-typical $z\in S$ will be denoted by $\gamma_z$ and the one associated to the lift $\wt z$ will be denoted $\wt\gamma_{\wt z}$, and parametrised such that $d(\wt z,\wt\gamma_{\wt z}) = d(\wt z,\wt\gamma_{\wt z}(0))$. 

Theorem~\ref{DefTrackGeod} allows us to define a set of geodesics associated to an ergodic measure of $\M(f)$ \cite[Theorem~ C]{alepablo}. 

\begin{theorem}\label{theoExistTrackSet}
To any $\mu\in\M(f)$ is associated a set $\dot\Lambda_\mu\subset T^1 S$ that is invariant under the geodesic flow on $T^1 S$, and such that for $\mu$-a.e.~$z\in S$, we have 
\[\overline{\dot\gamma_z(\R)} = \dot\Lambda_\mu.\]
\end{theorem}

When the geodesics of $\dot\Lambda_\mu$ are simple, the lamination $\dot\Lambda_\mu$ is minimal \cite[Theorem D]{alepablo}:

\begin{theorem}\label{Thifsimpletheoremintro}
Suppose $\mu \in \M(f)$ is such that $\gamma_x$ is simple for $\mu$-a.e.\ $x \in S$.  Then $\dot\Lambda_\mu$ is a minimal geodesic lamination.
\end{theorem}

We denote $\Lambda_\mu$ the projection of $\dot\Lambda_\mu$ on $S$. 
Let us define an equivalence relation $\sim$ on $\M(f)$ by: $\mu_1\sim \mu_2$ if one of the following is true:
\begin{itemize}
\item $\dot\Lambda_{\mu_1} = \dot\Lambda_{\mu_2}$;
\item There exist $\nu_1,\dots,\nu_m\in\M(f)$ such that $\nu_1=\mu_1$, $\nu_m=\mu_2$ and for all $1\le i<m$, there exists two geodesics of $\dot\Lambda_{\nu_i}$ and $\dot\Lambda_{\nu_{i+1}}$ that intersect transversally. 
\end{itemize} 

We denote $(\cl_i)_i$ the classes of this equivalence relation. By \cite{alepablo} (Section 6.2, and in particular Lemmas 6.7 and 6.8), to any class $\cl_i$ is associated a surface $S_i\subset S$ whose boundary is made of a finite collection of closed geodesics and minimal for inclusion among such surfaces such that for any $i$ and any $\mu\in \cl_i$ we have $\dot\Lambda_\mu \subset T^1 S_i$. If $\inte(S_i)\neq\emptyset$, then $S_i$ is open. Moreover, the surfaces $S_i$ are pairwise disjoint.

\begin{definition}\label{DefClassMeas}
There are three types of classes: classes $\cl_i$ such that $\bigcup_{\mu\in\cl_i}\dot\Lambda_\mu$:
\begin{itemize}[nosep]
\item is a single closed geodesic are called \emph{closed} classes;
\item is a minimal lamination that is not a closed geodesic are called \emph{minimal non-closed} classes;
\item has transverse intersection are called \emph{chaotic} classes.
\end{itemize}
\end{definition}

Note that a similar theory was developed for surface flows, leading to the definition of \emph{geodesic framework}, see \cite[Section 4.2]{zbMATH06859894}. 

\section{Laminations are orientable}\label{SecLaminOrient}

Consider $\mu\in\Merg(f)$, with $\mu\in \cl_i$, with $i\in I^1$. By \cite[Theorem F]{alepablo}, $\Lambda_i$ is a minimal geodesic lamination. If $\Lambda_i$ is a simple closed geodesic, there is nothing to prove. So, suppose that $\Lambda_i$ is made of non-closed geodesics. 

Consider $z$ a point that is $\mu$-typical, and $\wt z$ a lift of $z$ to $\wt S$. Let {$I$ be a maximal isotopy between the identity and $f$ and let} $\F$ be a transverse foliation for $f$ (see \cite[Section~2]{paper1PAF}), $\wt \F$ a lift of $\F$ to $\wt S$ and denote $\beta_0 = I^\Z_{\F}(z)$ the transverse trajectory of $z$. Let $\wt\beta_0$ be a lift of $\beta_0$ to $\wt S$.

We equip the set of geodesics of $\wt S$ with a distance $d$ that comes from a distance on $\partial\wt S\times\partial\wt S$, by the map that to a geodesic associates the couple of its endpoints in $\partial\wt S$.

\begin{lemma}\label{LemGeodNotAccuPerOrb}
There exists $d_0>0$ such that for $\mu$-a.e.~$z\in S$ and any geodesic $\wt\gamma$ of $\wt S$ lifting a closed geodesic $\gamma$ of $S$ and satisfying $d(\wt\gamma,\wt\gamma_{\wt z})\le d_0$ crosses $\wt\Lambda_i$; moreover in this case $\gamma$ is included in $S_i$.
\end{lemma}

In particular, any geodesic satisfying $d(\wt\gamma,\wt\gamma_{\wt z})\le d_0$ cannot be the tracking geodesic of some $f$-periodic point

\begin{proof}
Suppose by contradiction that there are geodesics $\wt\gamma_{p}$ lifting closed geodesics and accumulating on $\wt\gamma_{\wt z}$ without crossing $\wt\Lambda_\mu$.
Recall that $\gamma_z$ is dense in the geodesic lamination $\Lambda_i$ as this latter is  minimal \cite[Theorem~D]{alepablo}. Denote $\wt\Lambda_i$ the lift of $\Lambda_i$ to $\wt S$. Following \cite[Lemma 4.4]{casson}, the complement of $\wt\Lambda_i$ is made of ideal polygons and unions of core regions with crowns. The geodesic $\gamma_{p}$ cannot be included in an ideal polygon of the complement as it would force one of its lifts to have an endpoint in common with $\wt \gamma_{\wt z}$, which is impossible as $\gamma_{p}$ is a closed geodesic and $\gamma_{z}$ is not. 
Hence, $\gamma_{p}$ is included in a crown and close to a geodesic that is a boundary component of this crown, and one easily sees that it implies the conclusion of the lemma.
\end{proof}

\begin{lemma}\label{LemReturnDeckTrans}
For $\mu$-a.e.\ $z\in S$, there exists $T\in\G$ and $n\ge 0$ such that $n$ is large, $\wt f^n(\wt z)$ is close to $T\wt z$ and the geodesic axis of $T$ is close to $\wt\gamma_{\wt z}$.
\end{lemma}

The proof of this lemma uses the arguments of  \cite[Lemma 5.10]{alepablo}. Recall that $\wt\gamma$ is the map that at some point $\wt z\in \wt S$ whose projection on $S$ is $\mu$-typical associates its tracking geodesic in $\wt S$. Its extremities are $\alpha(\wt z)$ and $\omega(\wt z)$. 
Denoting $\pr_{\wt\gamma}$ the orthogonal projection on the geodesic $\wt\gamma$, a tracking geodesic is parametrised by $\R$, such that $\wt\gamma_{\wt z}(0) = \pr_{\wt\gamma_{\wt z}}({\wt z})$. In the latter we will sometimes identify a parametrised geodesic with $\R$.

The proof will be reused in the proof of Proposition~\ref{PropPasInterTrans}. Note that it does not use any hypothesis on the tracking geodesic $\gamma_z$.

\begin{proof}
We fix a regular fundamental domain $D$ of $S$ in $\wt S$ (regular in the sense that its boundary has 0 $\mu$-measure) and define a map $S\ni x\mapsto \wt x\in\wt S$ accordingly.
By Lusin theorem, there exists a set $A$ of positive $\mu$-measure such that the restriction to $A$ of the map $\Theta$ that satisfies $\Theta(x)=\wt\gamma_{\wt x}$ is continuous\footnote{By \cite[Theorem~B]{alepablo}, the map $\wt x\mapsto \wt\gamma_{\wt x}$ is measurable.}. 
Let $z$ be a $\mu$-typical point such that $\wt z\in \inte (D)$. {Since $z$ is typical, there is some positive integer $m$ such that $f^m(z)$ belongs to $A\cap\inte{D}$ and since the map $\Theta$ must also be continuous on points of $f^{-m}(A)$, we can assume from the start that $z$ belongs to $A$}.
Let $W$ be a small open chart on $S$ around $z$, that is included in the projection of the interior of $D$.
Let $n>0$ (to be chosen large thereafter) such that $f^n(z)\in A\cap W$. Let $T$ be the deck transformation such that if $\wt W$ is the lift of $W$ satisfying $\wt z\in\wt W$, then $\wt f^n(\wt z)\in T \wt W$.

Let us explain why if $W$ is chosen small enough, and if $n$ is large enough, then the axis of the deck transformation $T$ is close to $\wt\gamma_{\wt z}$. 
Indeed, on the one hand, if $n$ is large enough, then $\pr_{\wt\gamma_{\wt z}}(\wt f^n(\wt z))$ is large; in particular $T\wt z$, which belongs to the same fundamental domain as $\wt f^n(\wt z)$, is close to $\omega(\wt z)$. 
On the other hand (recall that $\wt\gamma_{\wt z}$ is parametrised such that $\pr_{\wt\gamma_{\wt z}}(\wt z) = 0$),
\begin{equation}\label{eqAxisBeta}
-\pr_{\wt\gamma_{T^{-1}\widetilde f^n(\wt z)}}(T^{-1}\wt z) = -\pr_{\wt\gamma_{\widetilde f^n(\wt z)}}(\wt z) = \pr_{\wt\gamma_{\wt z}}(\widetilde f^n(\wt z)) \text{ is large.}
\end{equation}
If $W$ is small enough, then by continuity of $x\mapsto\wt\gamma_{\wt x}$ on $A$, we have that the (unparametrised) geodesics $\wt\gamma_{T^{-1}\widetilde f^n(\wt z)}$ and $\wt\gamma_{\wt z}$ are close; moreover $T^{-1}\widetilde f^n(\wt z)$ and $\wt z$ lie in the same fundamental domain; this implies that the parametrizations of the geodesics $\wt\gamma_{T^{-1}\widetilde f^n(\wt z)}$ and $\wt\gamma_{\wt z}$ are at finite distance. 
By \eqref{eqAxisBeta} we deduce that $-\pr_{\wt\gamma_{\wt z}}(T^{-1}\wt z)$ is large, in other words that $T^{-1}\wt z$ is close to $\alpha(\wt z)$. Hence, $T$ maps $\wt z$ close to $\omega(\wt z)$ and $T^{-1}$ maps $\wt z$ close to $\alpha(\wt z)$; by elementary hyperbolic geometry we deduce that the axis of $T$ is close to $\wt\gamma_{\wt z}$ for $n$ large enough and $W$ small enough. Note that the same reasoning also works for negative times (with big absolute value).
\end{proof}

\begin{prop}\label{PropPasInterTrans}
The trajectory $\beta_0$ has no $\F$-transverse self-intersection.
\end{prop}

\begin{proof}
In this proof we will use the proof of Lemma~\ref{LemReturnDeckTrans}, and in particular the Lusin set $A$ of continuity of $x\mapsto \wt\gamma_{\wt x}$.

Suppose by contradiction that $\beta_0$ does have an $\F$-transverse self-intersection; we will prove that this forces the existence of periodic points with tracking geodesics close to the one $\gamma_z$ of $\beta_0$. 

By hypothesis, there exist $t_1<t_2$ and $s_1<s_2$ such that the trajectories $I^{[t_1,t_2]}_\F(z)$ and $I^{[s_1,s_2]}_\F(z)$ intersect $\F$-transversally in $S$. This implies that there exists a deck transformation $T_0$ such that $\wt I^{[t_1,t_2]}_{\wt \F}(\wt z)$ and $T_0\wt I^{[s_1,s_2]}_{\wt \F}(\wt z)$ intersect $\wt \F$-transversally in $\wt S$. 

Choosing $W$ a small enough neighbourhood of $z$, we get that for any points $\wt z', \wt z''\in \wt W$ we have that $\wt I^{[t_1-1,t_2+1]}_{\wt \F}(\wt z')$ and $T_0\wt I^{[s_1-1,s_2+1]}_{\wt \F}(\wt z'')$ intersect $\wt \F$-transversally (see \cite[Lemma~2.10]{paper1PAF}). 
By ergodicity of $\mu$, one can choose $n_-<0< n_+$ with $|n_-|,|n_+|\gg 1$, such that $f^{n_-}(z), f^{n_+}(z)\in A\cap W$. Let $T_-$ and $T_+$ be the deck transformations such that $T_-^{-1}\wt f^{n_-}(\wt z)\in \wt W$ and $T_+^{-1}\wt f^{n_+}(\wt z)\in \wt W$. 
By Lemma~\ref{LemReturnDeckTrans} and its proof, the axes of $T_-$ and $T_+$ are close to $\wt\gamma_{\wt z}$. 
We also know that 
$\wt I^{[t_1-1,t_2+1]}_{\wt \F}\big(T_-^{-1}\wt f^{n_-}(\wt z)\big)$ and $T_0\wt I^{[s_1-1,s_2+1]}_{\wt \F}\big(T_+^{-1}\wt f^{n_+}(\wt z)\big)$ intersect $\wt \F$-transversally; in other words, denoting $\wt z_0 = \wt f^{n_-}(\wt z)$, we have $\wt I^{[t_1-1,t_2+1]}_{\wt \F}\big(\wt z_0\big)$ and $T_-T_0T_+^{-1}\wt I^{[s_1-1,s_2+1]}_{\wt \F}\big(\wt f^{n_+-n_-}(\wt z_0)\big)$ intersect $\wt \F$-trans\-versally. By \cite[Theorem~2.12]{paper1PAF}, there exists $\wt z_p\in \wt S$ and $\tau\in\N$ such that $\wt f^\tau(\wt z_p) = T_-T_0T_+^{-1} \wt z_p$: the point $z_p$ is $\tau$-periodic and turns around $S$ by the deck transformation $T_-T_0T_+^{-1}$.

By Lemma~\ref{LemReturnDeckTrans}, the axes of $T_-$ and $T_+$ can be supposed to be close to $\wt\gamma_{\wt z}$ (with different directions), hence if $-n_-$  and $n_+$ are chosen large enough (recall that $T_0$ is fixed), then the axis $\wt\gamma_{\wt z_p}$ of $T_-T_0T_+^{-1}$ is also close to $\wt\gamma_{\wt z}$. 
This is a contradiction with Lemma~\ref{LemGeodNotAccuPerOrb}.
\end{proof}

\begin{lemma}\label{LemAlphaSimple}
The transverse trajectory $\wt \beta_0$ does not meet any leaf of $\wt\F$ twice or more (and in particular is simple).
\end{lemma}

\begin{proof}
If this is not true, then the trajectory $\wt\beta_0$ has to draw a simple transverse loop $\wt\Gamma$. By recurrence of $\beta_0$, and the fact that this trajectory is proper in $\wt S$, we deduce that there exists two deck transformations $T_1,T_2$ such that $\beta_0$ also draws $T_1\wt\Gamma$ and $T_2\wt\Gamma$, and we can choose $T_1,T_2$ such that $\wt\Gamma$, $T_1\wt\Gamma$ and $T_2\wt\Gamma$ do not intersect, and that neither of them is included in the bounded connected component of the complement of another one. 

This configuration contradicts \cite[Lemma~2.20]{paper1PAF}.
\end{proof}

We denote by $\wt B$ the set of leaves met by $\wt\beta_0$. By Lemma~\ref{LemAlphaSimple}, this set is a topological plane. 
The following lemma says that the bands $(T\wt B)_{T\in\G}$ are ordered.

\begin{lemma}\label{LemAlphaNotCross}
For any deck transformation $T$ of $\wt S$, if $T\wt\gamma_{\wt z}\subset R(\wt\gamma_{\wt z})$, then  $T\wt \beta_0 \subset \wt B \cup R(\wt B)$. 
Similarly, if $T\wt\gamma_{\wt z}\subset L(\wt\gamma_{\wt z})$, then  $T\wt \beta_0 \subset \wt B \cup L(\wt B)$
\end{lemma}

Note that because $\Lambda_\mu$ is a lamination, for any deck transformation $T\in\G\setminus\{\Id_{\wt S}\}$, either $T\wt\gamma_{\wt z}\subset L(\wt\gamma_{\wt z})$ or $T\wt\gamma_{\wt z}\subset R(\wt\gamma_{\wt z})$.

This lemma can be seen as a consequence of \cite[Proposition 5]{lct1} but as its proof is rather simple we include it for the sake of completeness.

\begin{proof}
We fist show that the trajectory $T\wt \beta_0$ does not cross $\wt B$.
Suppose it is false: the trajectory $T\wt \beta_0|_{[s_0,s_1]}$ crosses $\wt B$. For example $T\wt\beta_0(s_0)\in R(\tilde B)$ and $T\wt\beta_0(s_1)\in L(\tilde B)$. By Proposition~\ref{PropPasInterTrans}, the paths $\wt\beta_0$ and $T\wt\beta_0$ do not intersect transversally. 

As $T\wt\beta_0(s_0)\notin \tilde B$ and $T\wt\beta_0(s_1)\notin \tilde B$, this implies (see \cite[Lemma~2.17]{paper1PAF}) that there exists $t_0\in\R$ such that either $\wt\beta_0|_{(-\infty, t_0]}$ is $\F$-equivalent to a subpath of $T\beta_0|_{[s_0,s_1]}$, or $\wt\beta_0|_{[t_0,+\infty)}$ is $\F$-equivalent to a subpath of $T\beta_0|_{[s_0,s_1]}$. In both cases, $\beta_0$ accumulates in itself, that is impossible by \cite[Proposition~2.4]{paper1PAF}. 

Now, suppose that $T\wt\gamma_{\wt z}\subset R(\wt\gamma_{\wt z})$.
As $\wt\beta_0$ is simple (Lemma~\ref{LemAlphaSimple}), it has well defined left and right. 
Because of the geodesic tracking property, for $t$ large enough, we have that $T\wt\beta_0(t)\in R(\wt\beta_0)$. 

By \cite[Proposition~2.14]{paper1PAF}, if the recurrent transverse paths $\wt\beta_0$ and $T_0\wt\beta_0$ were equivalent at $+\infty$, then it would imply that $\beta_0$ is equivalent at $+\infty$ to a transverse path that is $T_0$-invariant. By \cite[Proposition~2.18]{paper1PAF}, this would imply that there is a sequence $(t_k)\in\N^\N$, with $t_k\to +\infty$, such that the points $\wt\beta_0(t_k)$ stay at finite distance to the geodesic axis of $T_0$. This is impossible as the tracking geodesic of $z$ is not closed and hence cannot have an endpoint in common with the geodesic axis of $T_0$.
We have proved that the paths $\wt\beta_0$ and $T_0\wt\beta_0$ are not equivalent at $+\infty$.

Hence, we cannot have $T\wt\beta_0\sim_{+\infty}\wt\beta_0$, in particular there exists $t$ arbitrarily large such that $T\wt\beta_0(t)\notin \wt B$. The fact that $T\wt\beta_0(t)\in R(\wt\beta_0)$ then implies that $T\wt\beta_0(t)\in R(\wt B)$. Because $T\wt \beta_0$ does not cross $\wt B$, we deduce that $T\wt\beta_0\cap L(\wt B) = \emptyset$.
\end{proof}

\begin{lemma}\label{LemBandDisjoint}
For any deck transformation $T$ of $\wt S$ such that $\wt\gamma_{\wt z}$ and $T\wt\gamma_{\wt z}$ have different orientations, we have $\wt B \cap T\wt B = \emptyset$.
\end{lemma}

\begin{proof}
If $\wt\beta_0\cap T\wt\beta_0 = \emptyset$, the lemma is true because of orientations reasons (this is where we use the fact that that $\wt\gamma_{\wt z}$ and $T\wt\gamma_{\wt z}$ have different orientations).

Suppose that $\wt\beta_0(t_1) = T\wt\beta_0 (s_1)$. By \cite[Proposition~2.14]{paper1PAF}, there exist $s_0<s_1<s_2$ such that $T\wt\beta_0(s_0),T\wt\beta_0(s_2)\notin \wt B$ and $T\wt\beta_0\big((s_0,s_2)\big)\subset \wt B$.
By Lemma~\ref{LemAlphaNotCross}, this implies that $T\wt\beta_0(s_0)$ and $T\wt\beta_0(s_2)$ lie in the same side of $\wt B$, in other words $T\wt\beta_0|_{[s_0,s_2]}$ visits $\wt B$. 
By pushing along the foliation, we can then transform $T\wt\beta_0$ into another $\wt\F$-equivalent curve $\wt\beta$ such that $\wt \beta|_{[s_0,s_2]} \cap\wt\beta_0 = \emptyset$. We can do the same thing for any intersection point of $T\wt\beta_0$ with $\wt\beta_0$, to build a path $\wt\beta$ that is $\wt\F$-equivalent to $T\wt\beta_0$ and disjoint from $\wt\beta_0$. 
By applying the same reasoning as at the beginning of the proof, we deduce that the set of leaves met by $\wt\beta$ does not meet $\wt\beta_0$, which implies that the set of leaves met by $T\wt\beta_0$ does not meet $\wt\beta_0$.
\end{proof}

\begin{figure}
\begin{center}

\tikzset{every picture/.style={line width=0.75pt}} 

\begin{tikzpicture}[x=0.75pt,y=0.75pt,yscale=-1,xscale=1]

\draw  [draw opacity=0][fill={rgb, 255:red, 245; green, 166; blue, 35 }  ,fill opacity=0.15 ] (317.51,31.08) .. controls (306.13,67.42) and (279.46,135.06) .. (276.63,166.17) .. controls (273.83,190.17) and (280.38,244.17) .. (280.71,273.48) .. controls (257.63,253.42) and (237.44,191.04) .. (247.88,157.67) .. controls (261.73,119.46) and (283.88,49.67) .. (317.51,31.08) -- cycle ;
\draw  [draw opacity=0][fill={rgb, 255:red, 245; green, 166; blue, 35 }  ,fill opacity=0.15 ] (248.75,25.71) .. controls (271.11,64.29) and (276.67,119.24) .. (273.84,150.34) .. controls (271.04,174.34) and (255.63,243.17) .. (217.25,259.71) .. controls (230.38,223.92) and (224.86,187.17) .. (235.29,153.79) .. controls (249.15,115.59) and (236.71,52.69) .. (248.75,25.71) -- cycle ;
\draw  [color={rgb, 255:red, 144; green, 19; blue, 254 }  ,draw opacity=1 ][fill={rgb, 255:red, 144; green, 19; blue, 254 }  ,fill opacity=0.2 ] (259.29,151.79) .. controls (270.57,153.43) and (268.93,155.79) .. (267.84,165.07) .. controls (266.75,174.34) and (268.75,177.07) .. (258.93,176.52) .. controls (249.11,175.98) and (247.48,171.61) .. (248.2,164.16) .. controls (248.93,156.7) and (248.02,150.16) .. (259.29,151.79) -- cycle ;
\draw [color={rgb, 255:red, 74; green, 144; blue, 226 }  ,draw opacity=1 ]   (217.25,259.71) .. controls (257.25,229.71) and (266.31,99.89) .. (248.75,25.71) ;
\draw [shift={(256.1,142.7)}, rotate = 94.57] [fill={rgb, 255:red, 74; green, 144; blue, 226 }  ,fill opacity=1 ][line width=0.08]  [draw opacity=0] (8.04,-3.86) -- (0,0) -- (8.04,3.86) -- (5.34,0) -- cycle    ;
\draw [color={rgb, 255:red, 74; green, 144; blue, 226 }  ,draw opacity=1 ]   (280.71,273.48) .. controls (241.78,222.96) and (265.91,91.08) .. (317.51,31.08) ;
\draw [shift={(267.21,144.59)}, rotate = 100.82] [fill={rgb, 255:red, 74; green, 144; blue, 226 }  ,fill opacity=1 ][line width=0.08]  [draw opacity=0] (8.04,-3.86) -- (0,0) -- (8.04,3.86) -- (5.34,0) -- cycle    ;
\draw [color={rgb, 255:red, 245; green, 166; blue, 35 }  ,draw opacity=1 ]   (232.89,165.42) .. controls (267.29,161.82) and (269.11,176.37) .. (273.84,143.64) ;
\draw  [draw opacity=0][fill={rgb, 255:red, 74; green, 144; blue, 226 }  ,fill opacity=1 ] (252.81,159.88) .. controls (252.81,158.93) and (253.58,158.16) .. (254.53,158.16) .. controls (255.47,158.16) and (256.24,158.93) .. (256.24,159.88) .. controls (256.24,160.82) and (255.47,161.59) .. (254.53,161.59) .. controls (253.58,161.59) and (252.81,160.82) .. (252.81,159.88) -- cycle ;
\draw  [draw opacity=0][fill={rgb, 255:red, 74; green, 144; blue, 226 }  ,fill opacity=1 ] (262.03,165.52) .. controls (262.03,164.57) and (262.8,163.81) .. (263.74,163.81) .. controls (264.69,163.81) and (265.46,164.57) .. (265.46,165.52) .. controls (265.46,166.47) and (264.69,167.23) .. (263.74,167.23) .. controls (262.8,167.23) and (262.03,166.47) .. (262.03,165.52) -- cycle ;
\draw  [line width=1.5]  (150.54,148.35) .. controls (150.54,79.4) and (206.44,23.5) .. (275.4,23.5) .. controls (344.35,23.5) and (400.25,79.4) .. (400.25,148.35) .. controls (400.25,217.31) and (344.35,273.21) .. (275.4,273.21) .. controls (206.44,273.21) and (150.54,217.31) .. (150.54,148.35) -- cycle ;
\draw [color={rgb, 255:red, 208; green, 2; blue, 27 }  ,draw opacity=0.5 ]   (239,267.5) .. controls (252.79,232.27) and (266.29,94.27) .. (276.79,23.52) ;
\draw [shift={(260.28,150.19)}, rotate = 277.48] [fill={rgb, 255:red, 208; green, 2; blue, 27 }  ,fill opacity=0.5 ][line width=0.08]  [draw opacity=0] (8.04,-3.86) -- (0,0) -- (8.04,3.86) -- (5.34,0) -- cycle    ;

\draw (248.53,169.83) node [anchor=north east] [inner sep=0.75pt]  [color={rgb, 255:red, 144; green, 19; blue, 254 }  ,opacity=1 ,xscale=1.2,yscale=1.2]  {$\wt{W}$};
\draw (251,158.16) node [anchor=south east] [inner sep=0.75pt]  [color={rgb, 255:red, 74; green, 144; blue, 226 }  ,opacity=1 ,xscale=1.2,yscale=1.2]  {$\wt{z}$};
\draw (270.95,166.81) node [anchor=west] [inner sep=0.75pt]  [color={rgb, 255:red, 74; green, 144; blue, 226 }  ,opacity=1 ,xscale=1.2,yscale=1.2]  {$T\wt{f}^{n}(\wt{z})$};
\draw (241.6,72.37) node [anchor=east] [inner sep=0.75pt]  [color={rgb, 255:red, 245; green, 166; blue, 35 }  ,opacity=1 ,xscale=1.2,yscale=1.2]  {$\wt{B}$};
\draw (298.61,84.87) node [anchor=west] [inner sep=0.75pt]  [color={rgb, 255:red, 245; green, 166; blue, 35 }  ,opacity=1 ,xscale=1.2,yscale=1.2]  {$T\wt{B}$};

\end{tikzpicture}

\caption{Proof of Theorem~\ref{TheoLaminMinim}: such a red path cannot be the image of the transverse trajectory of $\tilde z$ by some deck transformation, otherwise it would have to meet either $\wt B$ or $T\wt B$.}\label{FigTheoLaminMinim}
\end{center}
\end{figure}
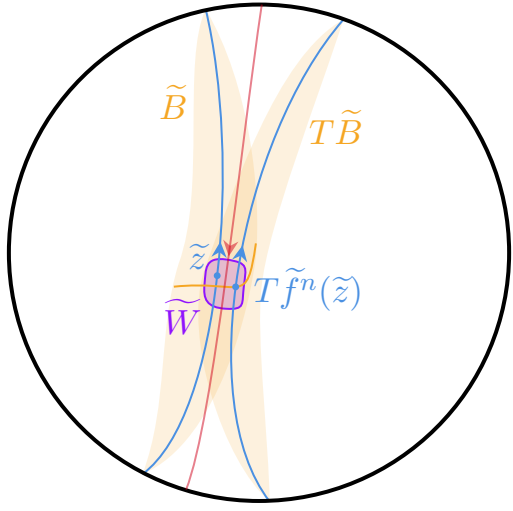

\begin{proof}[Proof of Theorem~\ref{TheoLaminMinim}]
This proof is depicted in Figure~\ref{FigTheoLaminMinim}.
Suppose $\Lambda_i$ is not orientable. 
Let $\wt W \subset \wt B$ be a local chart of $\wt\F$ around $\wt z$ (recall that $z$ is a $\mu$-typical point). By recurrence of the trajectory in $S$ (but not in $\wt S$), there exists $n\in\Z$ and a deck transformation $T_1\neq\mathrm{Id}$ such that $T_1 \wt f^n(\wt z) \in \wt W$. 
In particular, $T_1 \wt f^n(\wt z) \in \wt B$; by Lemma~\ref{LemBandDisjoint} this forces $\wt\gamma_{\wt z}$ and $T_1\wt\gamma_{\wt z}$ to have the same orientation.
As the lamination $\Lambda_i$ is minimal (see \cite[Theorem~D]{alepablo}) and not orientable, there exists a deck transformation $T_2$ such that $T_2\wt\gamma_{\wt z}$ lies between $\wt\gamma_{\wt z}$ and $T_1\wt\gamma_{\wt z}$, and such that it has not the same orientation as $\wt\gamma_{\wt z}$ (and $T_1\wt\gamma_{\wt z}$). But the trajectory $T_2\wt\beta_0$ cannot meet neither $\wt\beta_0$ nor $T_1\wt\beta_0$ (by Lemma~\ref{LemBandDisjoint}) nor $\wt \phi_{T_1\wt f^n(\wt z)}$ (still by Lemma~\ref{LemBandDisjoint}). The fact that $\wt\beta_0 \cup T_1\wt\beta_0 \cup \wt \phi_{T_1\wt f^n(\wt z)}$ separates both ends of $T_2\wt\beta_0$ leads to a contradiction.
\end{proof}

\section{Bounded deviations in directions crossing the one of the lamination}\label{SecBoundCross}

\subsection{Non-existence of transverse intersections}\label{SubSecNoTrans}

Recall that $z$ is a $\mu$-typical point and that we denote $\wt\beta_0 = I^\Z_{\wt F}(\wt z)$.

\begin{lemma}\label{LemNotEquivFiniteTime}
Let $T_0\in\G\setminus\{\Id\}$. There exist $\varep_0>0$ and $t'_0<0<t'_1$ such that for any $\wt y\in B(\wt z,\varep_0)$, the trajectories ${I^{[t'_0,0]}_{\wt \F}(\wt y)}$ and $I^{[0,t'_1]}_{\wt \F}(\wt y)$ are not $\wt\F$-equivalent to a subpath of $T_0\wt\beta_0$.
\end{lemma}

\begin{proof}
As in the proof of Lemma~\ref{LemAlphaNotCross}, using \cite[Proposition~2.14]{paper1PAF}, and \cite[Proposition~2.18]{paper1PAF} we deduce that the transverse paths $\wt\beta_0$ and $T_0\wt\beta_0$ are not equivalent at $+\infty$.

Hence, $\beta_0$ and $T_0\beta_0$ are not equivalent at $+\infty$, similarly they are not equivalent at $-\infty$. By \cite[Proposition~2.4]{paper1PAF}, one of them cannot accumulate in the other. This implies that there exists $t''_0<0<t''_1$ such that ${I^{[t''_0,0]}_{\wt \F}(\wt z)}$ and $I^{[0,t''_1]}_{\wt \F}(\wt z)$ are not $\wt\F$-equivalent to a subpath of $T_0\wt\beta_0$.

It remains to get a similar property for $\wt y$ instead of $\wt z$. By \cite[Lemma~2.10]{paper1PAF}, if $\varep_0$ is small enough, then for any $\wt y\in B(\wt z,\varep_0)$, the transverse path $I^{[t''_0,t''_1]}_{\wt \F}(\wt z)$ is equivalent to a subpath of $I^{[t''_0-1,t''_1+1]}_{\wt \F}(\wt y)$. Combined with the above fact, this proves the lemma for $t'_0 = t''_0-1$ and $t'_1 = t''_1+1$.
\end{proof}

Fix once for all $T_0, T_1\in\G\setminus\{\Id\}$ such that $T_0\wt\gamma_{\wt z}\subset L(\wt\gamma_{\wt z})$ and $T_1\wt\gamma_{\wt z}\subset R(\wt\gamma_{\wt z})$, and such that $\wt\gamma_{\wt z}$, $T_0\wt\gamma_{\wt z}$ and $T_1\wt\gamma_{\wt z}$ have the same orientation. If there exists one $T'_0\in\G\setminus\{\Id\}$ such that $T'_0\wt\gamma_{\wt z}\subset L(\wt\gamma_{\wt z})$ and $\wt\phi_{\wt z} \subset T'_0\wt B$, we choose $T_0 = T'_0$; by Lemma~\ref{LemBandDisjoint} the geodesics $\wt\gamma_{\wt z}$ and $T_0\wt\gamma_{\wt z}$ still have the same orientation. Similarly, If there exists one $T'_1\in\G\setminus\{\Id\}$ such that $T'_1\wt\gamma_{\wt z}\subset R(\wt\gamma_{\wt z})$ and $\wt\phi_{\wt z} \subset T'_1\wt B$, we choose $T_1 = T'_1$.

Let $\varep_0>0$ be given by Lemma~\ref{LemNotEquivFiniteTime} that works for both $T_0$ and $T_1$. We moreover suppose that $B(\wt z,\varep_0)$ is a trivialisation neighbourhood of $\wt \F$ such that for any $\wt y\in B(\wt z,\varep_0)$, the trajectory $I^{[-1,1]}_{\wt\F}(\wt y)$ crosses all the leaves meeting $B(\wt z,\varep_0)$. Let $t_0', t_1'$ be given by Lemma~\ref{LemNotEquivFiniteTime}.

\begin{figure}
\begin{center}

\tikzset{every picture/.style={line width=0.75pt}} 

\begin{tikzpicture}[x=0.75pt,y=0.75pt,yscale=-1,xscale=1]

\draw [color={rgb, 255:red, 245; green, 166; blue, 35 }  ,draw opacity=1 ]   (323.81,197.15) .. controls (328.31,172.15) and (361.56,178.4) .. (357.31,148.65) ;
\draw [shift={(340.55,176.73)}, rotate = 327.7] [fill={rgb, 255:red, 245; green, 166; blue, 35 }  ,fill opacity=1 ][line width=0.08]  [draw opacity=0] (7.14,-3.43) -- (0,0) -- (7.14,3.43) -- (4.74,0) -- cycle    ;
\draw [color={rgb, 255:red, 74; green, 144; blue, 226 }  ,draw opacity=1 ]   (283.08,61.52) .. controls (288.72,70.86) and (280.99,76.72) .. (281.88,86.72) .. controls (282.77,96.72) and (299.08,101.92) .. (299.48,114.32) .. controls (299.88,126.72) and (269.57,110.68) .. (266.68,141.12) .. controls (263.79,171.56) and (288.44,177.2) .. (287.88,192.32) .. controls (287.32,207.44) and (259.08,199.92) .. (261.08,209.92) .. controls (263.08,219.92) and (259.96,217.78) .. (255.88,222.58) ;
\draw [color={rgb, 255:red, 74; green, 144; blue, 226 }  ,draw opacity=1 ]   (279.77,176.47) .. controls (270.77,166.87) and (264.72,155.4) .. (266.68,141.12) ;
\draw [shift={(267.9,157.15)}, rotate = 68.64] [fill={rgb, 255:red, 74; green, 144; blue, 226 }  ,fill opacity=1 ][line width=0.08]  [draw opacity=0] (8.04,-3.86) -- (0,0) -- (8.04,3.86) -- (5.34,0) -- cycle    ;
\draw [color={rgb, 255:red, 74; green, 144; blue, 226 }  ,draw opacity=1 ]   (341.06,65.1) .. controls (346.71,74.44) and (339.81,80.35) .. (341.73,90.93) .. controls (343.65,101.51) and (361.66,96.45) .. (362.06,108.85) .. controls (362.46,121.25) and (342.7,107.91) .. (339.81,138.35) .. controls (336.92,168.79) and (372.98,160.76) .. (374.25,172.4) .. controls (375.52,184.04) and (358.98,191.67) .. (369.16,194.04) .. controls (379.34,196.4) and (371.59,197.83) .. (377.34,202.58) ;
\draw [color={rgb, 255:red, 74; green, 144; blue, 226 }  ,draw opacity=1 ]   (339.81,138.35) .. controls (341.53,123.56) and (344.95,118.48) .. (358.07,114.87) ;
\draw [shift={(346.47,120.44)}, rotate = 126.22] [fill={rgb, 255:red, 74; green, 144; blue, 226 }  ,fill opacity=1 ][line width=0.08]  [draw opacity=0] (8.04,-3.86) -- (0,0) -- (8.04,3.86) -- (5.34,0) -- cycle    ;
\draw [color={rgb, 255:red, 74; green, 144; blue, 226 }  ,draw opacity=1 ]   (274.84,231.69) .. controls (275.95,222.8) and (307.29,226.58) .. (305.51,205.91) .. controls (303.73,185.25) and (310.09,163.95) .. (332.4,173.47) .. controls (354.7,182.99) and (348.17,196.14) .. (343.95,206.14) .. controls (339.73,216.14) and (355.51,214.58) .. (354.17,224.36) ;
\draw  [line width=1.5]  (216,148) .. controls (216,98.29) and (256.29,58) .. (306,58) .. controls (355.71,58) and (396,98.29) .. (396,148) .. controls (396,197.71) and (355.71,238) .. (306,238) .. controls (256.29,238) and (216,197.71) .. (216,148) -- cycle ;
\draw [color={rgb, 255:red, 74; green, 144; blue, 226 }  ,draw opacity=1 ]   (332.4,173.47) .. controls (316.1,166.56) and (308.26,176.31) .. (306.17,188.8) ;
\draw [shift={(312.27,175.31)}, rotate = 326.8] [fill={rgb, 255:red, 74; green, 144; blue, 226 }  ,fill opacity=1 ][line width=0.08]  [draw opacity=0] (8.04,-3.86) -- (0,0) -- (8.04,3.86) -- (5.34,0) -- cycle    ;
\draw  [draw opacity=0][fill={rgb, 255:red, 0; green, 0; blue, 0 }  ,fill opacity=1 ] (354.02,162.56) .. controls (354.02,161.55) and (354.84,160.73) .. (355.84,160.73) .. controls (356.85,160.73) and (357.67,161.55) .. (357.67,162.56) .. controls (357.67,163.56) and (356.85,164.38) .. (355.84,164.38) .. controls (354.84,164.38) and (354.02,163.56) .. (354.02,162.56) -- cycle ;
\draw (269,130.4) node [anchor=east] [inner sep=0.75pt]  [color={rgb, 255:red, 74; green, 144; blue, 226 }  ,opacity=1 ,xscale=1.2,yscale=1.2]  {$I_{\wt{\F}}^{\Z}(\wt{z})$};
\draw (348,76.53) node [anchor=south west] [inner sep=0.75pt]  [color={rgb, 255:red, 74; green, 144; blue, 226 }  ,opacity=1 ,xscale=1.2,yscale=1.2]  {$T'I_{\wt{\F}}^{\Z}(\wt{z})$};
\draw (345,225) node [anchor=north west][inner sep=0.75pt]  [color={rgb, 255:red, 74; green, 144; blue, 226 }  ,opacity=1 ,xscale=1.2,yscale=1.2]  {$T'T_{0} I_{\wt{\F}}^{\Z}(\wt{z})$};
\draw (359.67,159.16) node [anchor=south west] [inner sep=0.75pt]  [xscale=1.2,yscale=1.2]  {$T'\wt{z}$};
\draw (325.25,189.87) node [anchor=north] [inner sep=0.75pt]  [color={rgb, 255:red, 216; green, 138; blue, 9 }  ,opacity=1 ,xscale=1.2,yscale=1.2]  {$\wt{\phi }_{T'\wt{z}}$};
\end{tikzpicture}

\caption{Proof of Lemma~\ref{LemIfFarThenNotEquiv}, first case: when the geodesics $\wt\gamma_{\wt z}$ and $T'T_0\wt\gamma_{\wt z}$ do not have the same orientation.}\label{FigLemIfFarThenNotEquiv0}
\end{center}
\end{figure}
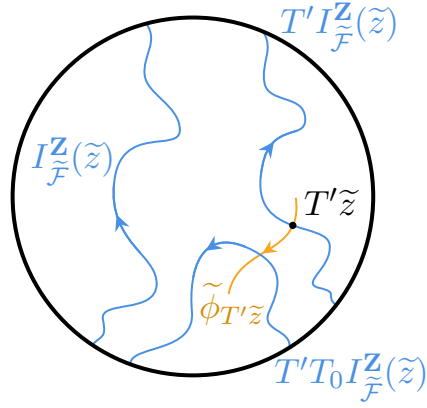

\begin{lemma}\label{LemIfFarThenNotEquiv}
Let $C_1 = \max\big(d(\wt z, \wt\gamma_{\wt z}) ,\,d(\wt z, T_0\wt\gamma_{\wt z}),\, d(\wt z, T_1\wt\gamma_{\wt z}) \big)$.
If $T'\in\G$ is such that $d(T'\wt z,\, \wt\gamma_{\wt z})>C_1$ and that $T'\wt z\in R(\wt\gamma_{\wt z})$ (resp. $T'\wt z\in L(\wt\gamma_{\wt z})$), then for any $\wt y\in B(T'\wt z, \varep_0)$, there exist $t_0\in [t_0', 1]$ and $t_1\in [-1, t_1']$ such that $\wt\phi_{I^{t_0}_{\wt \F}(\wt y)}\subset R(\wt B)$ and $\wt\phi_{I^{t_1}_{\wt \F}(\wt y)}\subset R(\wt B)$ (resp. $\wt\phi_{I^{t_0}_{\wt \F}(\wt y)}\subset L(\wt B)$ and $\wt\phi_{I^{t_1}_{\wt \F}(\wt y)}\subset L(\wt B)$).
\end{lemma}

\begin{figure}
\begin{center}
\tikzset{every picture/.style={line width=0.75pt}} 

\begin{tikzpicture}[x=0.75pt,y=0.75pt,yscale=-1.2,xscale=1.2]

\draw  [color={rgb, 255:red, 245; green, 166; blue, 35 }  ,draw opacity=0.3 ][fill={rgb, 255:red, 245; green, 166; blue, 35 }  ,fill opacity=0.15 ] (291.84,23.3) .. controls (282.99,65.3) and (291.25,115.71) .. (288.41,146.82) .. controls (285.61,170.82) and (275.27,221.87) .. (278.99,265.87) .. controls (266.97,252.88) and (219.75,162.72) .. (220.6,140.15) .. controls (222.32,113.01) and (264.7,55.01) .. (291.84,23.3) -- cycle ;
\draw  [color={rgb, 255:red, 245; green, 166; blue, 35 }  ,draw opacity=0.3 ][fill={rgb, 255:red, 245; green, 166; blue, 35 }  ,fill opacity=0.15 ] (225.44,21.5) .. controls (247.81,60.08) and (269.54,115.33) .. (266.7,146.44) .. controls (263.9,170.44) and (219.61,225.14) .. (194.44,249.5) .. controls (207.57,213.71) and (196.55,179.82) .. (206.99,146.44) .. controls (220.84,108.24) and (213.41,48.48) .. (225.44,21.5) -- cycle ;
\draw [color={rgb, 255:red, 74; green, 144; blue, 226 }  ,draw opacity=1 ]   (194.44,249.5) .. controls (225.89,211.14) and (244.89,104.72) .. (225.44,21.5) ;
\draw [shift={(231.27,134.47)}, rotate = 96.3] [fill={rgb, 255:red, 74; green, 144; blue, 226 }  ,fill opacity=1 ][line width=0.08]  [draw opacity=0] (8.04,-3.86) -- (0,0) -- (8.04,3.86) -- (5.34,0) -- cycle    ;
\draw  [color={rgb, 255:red, 245; green, 166; blue, 35 }  ,draw opacity=0.3 ][fill={rgb, 255:red, 245; green, 166; blue, 35 }  ,fill opacity=0.15 ] (320.7,35.87) .. controls (309.31,72.2) and (302.95,112.73) .. (300.11,143.83) .. controls (297.31,167.83) and (306.93,227.12) .. (307.27,256.44) .. controls (288.13,235.01) and (246.13,174.72) .. (244.99,140.72) .. controls (244.13,103.01) and (278.13,71.3) .. (320.7,35.87) -- cycle ;
\draw [color={rgb, 255:red, 74; green, 144; blue, 226 }  ,draw opacity=1 ]   (307.27,256.44) .. controls (268.33,205.92) and (269.1,95.87) .. (320.7,35.87) ;
\draw [shift={(280.57,141.02)}, rotate = 94.02] [fill={rgb, 255:red, 74; green, 144; blue, 226 }  ,fill opacity=1 ][line width=0.08]  [draw opacity=0] (8.04,-3.86) -- (0,0) -- (8.04,3.86) -- (5.34,0) -- cycle    ;
\draw  [line width=1.5]  (132.54,142.87) .. controls (132.54,73.92) and (188.44,18.02) .. (257.4,18.02) .. controls (326.35,18.02) and (382.25,73.92) .. (382.25,142.87) .. controls (382.25,211.83) and (326.35,267.73) .. (257.4,267.73) .. controls (188.44,267.73) and (132.54,211.83) .. (132.54,142.87) -- cycle ;
\draw [color={rgb, 255:red, 74; green, 144; blue, 226 }  ,draw opacity=1 ]   (278.99,265.87) .. controls (256.41,197.58) and (254.99,111.3) .. (291.84,23.3) ;
\draw [shift={(263.78,139.57)}, rotate = 93.53] [fill={rgb, 255:red, 74; green, 144; blue, 226 }  ,fill opacity=1 ][line width=0.08]  [draw opacity=0] (8.04,-3.86) -- (0,0) -- (8.04,3.86) -- (5.34,0) -- cycle    ;
\draw [color={rgb, 255:red, 245; green, 166; blue, 35 }  ,draw opacity=1 ]   (299.29,53.92) .. controls (292.24,68.75) and (295.88,108.36) .. (302.96,117.13) ;
\draw  [draw opacity=0][fill={rgb, 255:red, 74; green, 144; blue, 226 }  ,fill opacity=1 ] (293.65,77.11) .. controls (293.65,76.01) and (294.54,75.12) .. (295.64,75.12) .. controls (296.73,75.12) and (297.62,76.01) .. (297.62,77.11) .. controls (297.62,78.21) and (296.73,79.1) .. (295.64,79.1) .. controls (294.54,79.1) and (293.65,78.21) .. (293.65,77.11) -- cycle ;
\draw [color={rgb, 255:red, 245; green, 166; blue, 35 }  ,draw opacity=1 ]   (300.96,181.9) .. controls (289.57,195.75) and (284.02,224.27) .. (288.41,232.61) ;
\draw  [draw opacity=0][fill={rgb, 255:red, 74; green, 144; blue, 226 }  ,fill opacity=1 ] (285.48,215.92) .. controls (285.48,214.82) and (286.37,213.93) .. (287.47,213.93) .. controls (288.57,213.93) and (289.46,214.82) .. (289.46,215.92) .. controls (289.46,217.02) and (288.57,217.91) .. (287.47,217.91) .. controls (286.37,217.91) and (285.48,217.02) .. (285.48,215.92) -- cycle ;

\draw (218.31,61.7) node [anchor=east] [inner sep=0.75pt]  [color={rgb, 255:red, 245; green, 166; blue, 35 }  ,opacity=1 ,xscale=1.2,yscale=1.2]  {$\wt{B}$};
\draw (303,122.76) node [anchor=west] [inner sep=0.75pt]  [color={rgb, 255:red, 245; green, 166; blue, 35 }  ,opacity=1 ,xscale=1.2,yscale=1.2]  {$T'\wt{B}$};
\draw (282.47,40.38) node [anchor=south east] [inner sep=0.75pt]  [font=\footnotesize,color={rgb, 255:red, 245; green, 166; blue, 35 }  ,opacity=1 ,xscale=1.2,yscale=1.2]  {$T'T_{0}\wt{B}$};
\draw (220.93,190.91) node [anchor=east] [inner sep=0.75pt]  [color={rgb, 255:red, 74; green, 144; blue, 226 }  ,opacity=1 ,xscale=1.2,yscale=1.2]  {$I_{\wt{\F}}^{\Z}(\wt{z})$};
\draw (282.46,153.68) node [anchor=west] [inner sep=0.75pt]  [color={rgb, 255:red, 74; green, 144; blue, 226 }  ,opacity=1 ,xscale=1.2,yscale=1.2]  {$T'I_{\wt{\F}}^{\Z}(\wt{z})$};
\draw (295.19,206.64) node [anchor=west] [inner sep=0.75pt]  [color={rgb, 255:red, 245; green, 166; blue, 35 }  ,opacity=1 ,xscale=1.2,yscale=1.2]  {$\wt{\phi }_{I_{\wt{\F}}^{t_{0}}( T'\wt{z})}$};
\draw (298.74,87.7) node [anchor=west] [inner sep=0.75pt]  [color={rgb, 255:red, 245; green, 166; blue, 35 }  ,opacity=1 ,xscale=1.2,yscale=1.2]  {$\wt{\phi }_{I_{\wt{\F}}^{t_{1}}( T'\wt{z})}$};
\draw (275.29,241.58) node [anchor=north east] [inner sep=0.75pt]  [font=\footnotesize,color={rgb, 255:red, 74; green, 144; blue, 226 }  ,opacity=1 ,xscale=1.2,yscale=1.2]  {$T'T_{0} I_{\wt{\F}}^{\Z}(\wt{z})$};
\end{tikzpicture}
\caption{Proof of Lemma~\ref{LemIfFarThenNotEquiv}: the leaves $\wt\phi_{I^{t_0}_{\wt\F}(T'\wt z)}$ and $\wt\phi_{I^{t_1}_{\wt\F}(T'\wt z)}$ are included in $R(T'T_0\wt B)$ and hence, by ordering of the bands, in $R(\wt B)$.}\label{FigLemIfFarThenNotEquiv}
\end{center}
\end{figure}
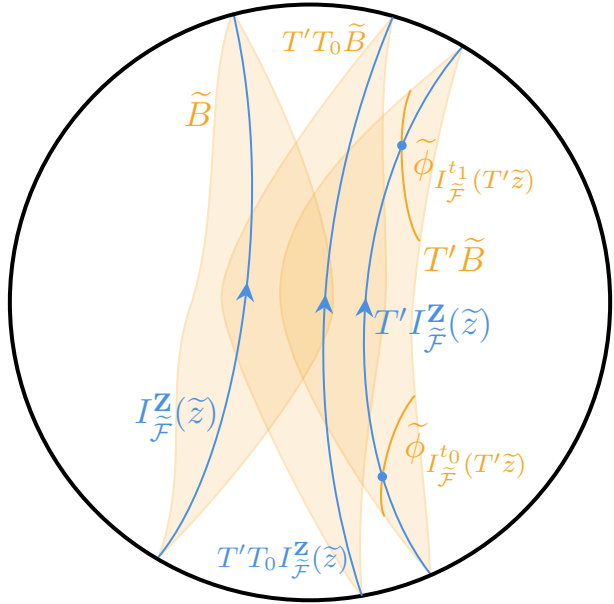


\begin{proof}
Suppose that $T'{\wt z}$ is on the right of $\wt\gamma_{\wt z}$ (the other case being similar); note that (because $d(T'\wt z, \wt\gamma_{\wt z})> C_1 \ge d(\wt z, \wt\gamma_{\wt z})$) this is equivalent to the fact that $T'\gamma_{\wt z}$ is on the right of $\wt\gamma_{\wt z}$. By Lemma~\ref{LemBandDisjoint} one can suppose that $T'\gamma_{\wt z}$ and $\wt\gamma_{\wt z}$ have the same orientation (otherwise the lemma is proved for any choice of $t_0, t_1$). Also, if there exists no $T\in\G$ such that $T\wt\gamma_{\wt z}\subset L(\wt\gamma_{\wt z})$ and $\wt\phi_{\wt z} \subset T\wt B$, then the lemma is proved (for $t_0=t_1$ such that $I^{t_0}_{\wt\F}(\wt y) = \wt\phi_{T'\wt z}$). So we suppose that $\wt\phi_{\wt z} \subset T_0\wt B$.

If $\wt\gamma_{\wt z}$ and $T'T_0\wt\gamma_{\wt z}$ do not have the same orientation (as in Figure~\ref{FigLemIfFarThenNotEquiv0}), then by Lemma~\ref{LemBandDisjoint} we get $T'T_0\wt B\cap \wt B = \emptyset$ and so $\wt\phi_{T'\wt z}\cap \wt B = \emptyset$; this proves the lemma for $t_0=t_1$ such that $I^{t_0}_{\wt\F}(\wt y) = \wt\phi_{T'\wt z}$).

Otherwise $\wt\gamma_{\wt z}$ and $T'T_0\wt\gamma_{\wt z}$ have the same orientation.
An idea of the proof in this case is depicted in Figure~\ref{FigLemIfFarThenNotEquiv}. Recall that $T_0\wt\gamma_{\wt z}$ is on the left of $\wt\gamma_{\wt z}$. From the hypothesis $d(\wt\gamma_{\wt z}, T'\wt z)>C_1$ we deduce that $T'T_0\wt\gamma_{\wt z}$ is between $T'\wt\gamma_{\wt z}$ and $\wt\gamma_{\wt z}$.

By Lemma~\ref{LemAlphaNotCross}, the bands $\wt B$, $T'T_0\wt B$ and $T'\wt B$ are well ordered: we have $T'T_0\wt B \cup R(T'T_0\wt B) \subset R(\wt B) \cup \wt B$ and $T'\wt B \subset R(T'T_0 \wt B) \cup T'T_0\wt B$. 

Pick $\wt y\in B(\wt z,\varep_0)$. By Lemma~\ref{LemNotEquivFiniteTime}, the paths $T'{I^{[t'_0,0]}_{\wt \F}(\wt y)}$ and $T'{I^{[0,t'_1]}_{\wt \F}(\wt y)}$ are not $\wt\F$-equivalent to a subpath of $T'T_0\wt\beta_0$. 
In particular, there exist $t_0\in[t'_0,0]$ and $t_1\in [0,t'_1]$ such that $\wt\phi_{I^{t_0}_{\wt \F}(T'\wt y)}\notin T'T_0\wt B$ and $\wt\phi_{I^{t_1}_{\wt \F}(T'\wt y)}\notin T'T_0\wt B$.
By the previous ordering properties on the bands, this means that $\wt\phi_{I^{t_0}_{\wt \F}(T'\wt y)}\subset R(T'T_0\wt B)\subset R(\wt B)$ and $\wt\phi_{I^{t_1}_{\wt \F}(T'\wt y)}\subset R(T'T_0\wt B)\subset R(\wt B)$. This proves the lemma.
\end{proof}

The following improves Proposition~\ref{PropPasInterTrans}.

\begin{prop}\label{LastPropBndDevIrrat}
The path $\beta_0 = I^\Z_{\wt\F}(\wt z)$ cannot intersect $\wt\F$-transversally any other path.
\end{prop}

\begin{figure}
\begin{center}
\tikzset{every picture/.style={line width=0.75pt}} 

\begin{tikzpicture}[x=0.75pt,y=0.75pt,yscale=-1.2,xscale=1.2]

\draw [color={rgb, 255:red, 83; green, 113; blue, 0 }  ,draw opacity=1 ][line width=0.75]    (238.35,245.18) .. controls (241.79,244.12) and (255.15,243.98) .. (261.95,239.18) .. controls (268.75,234.38) and (299.55,146.79) .. (295.55,75.99) ;
\draw [shift={(253.67,242.61)}, rotate = 168.95] [fill={rgb, 255:red, 83; green, 113; blue, 0 }  ,fill opacity=1 ][line width=0.08]  [draw opacity=0] (8.04,-3.86) -- (0,0) -- (8.04,3.86) -- (5.34,0) -- cycle    ;
\draw [shift={(288.67,156.43)}, rotate = 101.47] [fill={rgb, 255:red, 83; green, 113; blue, 0 }  ,fill opacity=1 ][line width=0.08]  [draw opacity=0] (8.04,-3.86) -- (0,0) -- (8.04,3.86) -- (5.34,0) -- cycle    ;
\draw [color={rgb, 255:red, 83; green, 113; blue, 0 }  ,draw opacity=1 ][line width=0.75]    (145.9,145) .. controls (255.15,113.99) and (306.75,138.39) .. (317.15,131.19) .. controls (327.55,123.99) and (331.02,58.58) .. (381.82,24.18) ;
\draw [shift={(234.71,129.33)}, rotate = 176.27] [fill={rgb, 255:red, 83; green, 113; blue, 0 }  ,fill opacity=1 ][line width=0.08]  [draw opacity=0] (8.04,-3.86) -- (0,0) -- (8.04,3.86) -- (5.34,0) -- cycle    ;
\draw [shift={(342.83,69.22)}, rotate = 116.71] [fill={rgb, 255:red, 83; green, 113; blue, 0 }  ,fill opacity=1 ][line width=0.08]  [draw opacity=0] (8.04,-3.86) -- (0,0) -- (8.04,3.86) -- (5.34,0) -- cycle    ;
\draw [color={rgb, 255:red, 208; green, 2; blue, 27 }  ,draw opacity=1 ]   (150,150) .. controls (267.25,98.58) and (403.25,193.58) .. (470,140) ;
\draw [shift={(315.14,143.28)}, rotate = 190.05] [fill={rgb, 255:red, 208; green, 2; blue, 27 }  ,fill opacity=1 ][line width=0.08]  [draw opacity=0] (7.14,-3.43) -- (0,0) -- (7.14,3.43) -- (4.74,0) -- cycle    ;
\draw [color={rgb, 255:red, 74; green, 144; blue, 226 }  ,draw opacity=1 ]   (250,270) .. controls (293.25,198.08) and (310.26,115.93) .. (302.01,54.85) ;
\draw [shift={(294.25,162.83)}, rotate = 104.23] [fill={rgb, 255:red, 74; green, 144; blue, 226 }  ,fill opacity=1 ][line width=0.08]  [draw opacity=0] (8.04,-3.86) -- (0,0) -- (8.04,3.86) -- (5.34,0) -- cycle    ;
\draw [color={rgb, 255:red, 74; green, 144; blue, 226 }  ,draw opacity=1 ]   (330,230) .. controls (307.25,144.58) and (338.25,44.58) .. (400,20) ;
\draw [shift={(331.5,108.06)}, rotate = 105.47] [fill={rgb, 255:red, 74; green, 144; blue, 226 }  ,fill opacity=1 ][line width=0.08]  [draw opacity=0] (8.04,-3.86) -- (0,0) -- (8.04,3.86) -- (5.34,0) -- cycle    ;
\draw [color={rgb, 255:red, 208; green, 2; blue, 27 }  ,draw opacity=1 ]   (240,250) .. controls (271.75,234.08) and (285.75,270.08) .. (300,260) ;
\draw [shift={(274.01,251.29)}, rotate = 206.77] [fill={rgb, 255:red, 208; green, 2; blue, 27 }  ,fill opacity=1 ][line width=0.08]  [draw opacity=0] (8.04,-3.86) -- (0,0) -- (8.04,3.86) -- (5.34,0) -- cycle    ;
\draw [color={rgb, 255:red, 245; green, 166; blue, 35 }  ,draw opacity=1 ]   (117.92,145.5) .. controls (159.33,128.92) and (159.92,170) .. (136.92,180.5) ;
\draw [color={rgb, 255:red, 245; green, 166; blue, 35 }  ,draw opacity=1 ]   (258.42,220) .. controls (270.92,192) and (297.92,210) .. (283.92,234.5) ;
\draw [color={rgb, 255:red, 245; green, 166; blue, 35 }  ,draw opacity=1 ]   (285.82,56.9) .. controls (285.32,77.4) and (317.85,102.35) .. (324.25,34.35) ;
\draw [color={rgb, 255:red, 245; green, 166; blue, 35 }  ,draw opacity=1 ]   (388.32,9.8) .. controls (368.25,20.25) and (395.82,47.3) .. (412.32,32.3) ;
\draw [color={rgb, 255:red, 245; green, 166; blue, 35 }  ,draw opacity=1 ]   (477.92,113.5) .. controls (451.92,125) and (472.92,163) .. (504.92,141.5) ;
\draw  [draw opacity=0][fill={rgb, 255:red, 208; green, 2; blue, 27 }  ,fill opacity=1 ] (146.93,150) .. controls (146.93,148.08) and (148.48,146.53) .. (150.4,146.53) .. controls (152.32,146.53) and (153.87,148.08) .. (153.87,150) .. controls (153.87,151.92) and (152.32,153.47) .. (150.4,153.47) .. controls (148.48,153.47) and (146.93,151.92) .. (146.93,150) -- cycle ;
\draw  [draw opacity=0][fill={rgb, 255:red, 208; green, 2; blue, 27 }  ,fill opacity=1 ] (466.53,140) .. controls (466.53,138.08) and (468.08,136.53) .. (470,136.53) .. controls (471.92,136.53) and (473.47,138.08) .. (473.47,140) .. controls (473.47,141.92) and (471.92,143.47) .. (470,143.47) .. controls (468.08,143.47) and (466.53,141.92) .. (466.53,140) -- cycle ;
\draw [color={rgb, 255:red, 245; green, 166; blue, 35 }  ,draw opacity=1 ]   (222.75,240.38) .. controls (243.55,235.98) and (245.38,257.86) .. (228.75,264.78) ;
\draw  [draw opacity=0][fill={rgb, 255:red, 65; green, 117; blue, 5 }  ,fill opacity=1 ] (142.43,145) .. controls (142.43,143.08) and (143.98,141.53) .. (145.9,141.53) .. controls (147.82,141.53) and (149.37,143.08) .. (149.37,145) .. controls (149.37,146.92) and (147.82,148.47) .. (145.9,148.47) .. controls (143.98,148.47) and (142.43,146.92) .. (142.43,145) -- cycle ;
\draw [color={rgb, 255:red, 245; green, 166; blue, 35 }  ,draw opacity=1 ]   (312.8,227.15) .. controls (313.94,199.72) and (343.08,196.87) .. (343.06,227.24) ;
\draw  [draw opacity=0][fill={rgb, 255:red, 65; green, 117; blue, 5 }  ,fill opacity=1 ] (234.88,245.18) .. controls (234.88,243.26) and (236.43,241.71) .. (238.35,241.71) .. controls (240.26,241.71) and (241.82,243.26) .. (241.82,245.18) .. controls (241.82,247.1) and (240.26,248.65) .. (238.35,248.65) .. controls (236.43,248.65) and (234.88,247.1) .. (234.88,245.18) -- cycle ;
\draw [color={rgb, 255:red, 245; green, 166; blue, 35 }  ,draw opacity=1 ]   (336.21,39.08) .. controls (320.21,82.58) and (339.6,93.87) .. (372.71,84.58) ;
\draw [color={rgb, 255:red, 74; green, 144; blue, 226 }  ,draw opacity=1 ]   (345.75,66.92) -- (350.67,70.46) ;
\draw [color={rgb, 255:red, 74; green, 144; blue, 226 }  ,draw opacity=1 ]   (382.94,26.21) -- (385.94,30.21) ;
\draw [color={rgb, 255:red, 74; green, 144; blue, 226 }  ,draw opacity=1 ]   (300.88,78.92) -- (306.71,78.65) ;
\draw [color={rgb, 255:red, 74; green, 144; blue, 226 }  ,draw opacity=1 ]   (278,205.92) -- (283.17,208.46) ;
\draw [color={rgb, 255:red, 74; green, 144; blue, 226 }  ,draw opacity=1 ]   (322.33,206.94) -- (328.25,205.29) ;
\draw [color={rgb, 255:red, 74; green, 144; blue, 226 }  ,draw opacity=1 ]   (337.9,83.38) -- (343.02,86.36) ;
\draw  [draw opacity=0][fill={rgb, 255:red, 0; green, 0; blue, 0 }  ,fill opacity=1 ] (320.62,144.81) .. controls (320.62,142.9) and (322.17,141.34) .. (324.09,141.34) .. controls (326,141.34) and (327.55,142.9) .. (327.55,144.81) .. controls (327.55,146.73) and (326,148.28) .. (324.09,148.28) .. controls (322.17,148.28) and (320.62,146.73) .. (320.62,144.81) -- cycle ;
\draw  [draw opacity=0][fill={rgb, 255:red, 0; green, 0; blue, 0 }  ,fill opacity=1 ] (295.12,140.81) .. controls (295.12,138.9) and (296.67,137.34) .. (298.59,137.34) .. controls (300.5,137.34) and (302.05,138.9) .. (302.05,140.81) .. controls (302.05,142.73) and (300.5,144.28) .. (298.59,144.28) .. controls (296.67,144.28) and (295.12,142.73) .. (295.12,140.81) -- cycle ;
\draw    (379.18,189.45) .. controls (346.81,184.38) and (324.33,177.41) .. (304.31,149.62) ;
\draw [shift={(302.78,147.45)}, rotate = 55.43] [fill={rgb, 255:red, 0; green, 0; blue, 0 }  ][line width=0.08]  [draw opacity=0] (7.14,-3.43) -- (0,0) -- (7.14,3.43) -- (4.74,0) -- cycle    ;

\draw (155.67,150.4) node [anchor=north west][inner sep=0.75pt]  [color={rgb, 255:red, 208; green, 2; blue, 27 }  ,opacity=1 ,xscale=1.2,yscale=1.2]  {$I_{\wh{\F}}^{t_{-}}(\wh{y})$};
\draw (147.1,136.79) node [anchor=south] [inner sep=0.75pt]  [color={rgb, 255:red, 65; green, 117; blue, 5 }  ,opacity=1 ,xscale=1.2,yscale=1.2]  {$\wh{x}$};
\draw (293.72,55.5) node [anchor=south] [inner sep=0.75pt]  [color={rgb, 255:red, 74; green, 144; blue, 226 }  ,opacity=1 ,xscale=1.2,yscale=1.2]  {$T'I_{\wh{\F}}^{\Z}(\wh{z})$};
\draw (401.38,16.5) node [anchor=west] [inner sep=0.75pt]  [color={rgb, 255:red, 74; green, 144; blue, 226 }  ,opacity=1 ,xscale=1.2,yscale=1.2]  {$I_{\wh{\F}}^{\Z}(\wh{z})$};
\draw (248.55,236.04) node [anchor=south east] [inner sep=0.75pt]  [color={rgb, 255:red, 65; green, 117; blue, 5 }  ,opacity=1 ,xscale=1.2,yscale=1.2]  {$T'\wh{x}$};
\draw (472,140.07) node [anchor=south west] [inner sep=0.75pt]  [color={rgb, 255:red, 208; green, 2; blue, 27 }  ,opacity=1 ,xscale=1.2,yscale=1.2]  {$I_{\wh{\F}}^{t_{+}}(\wh{y})$};
\draw (326.7,204.41) node [anchor=south west] [inner sep=0.75pt]  [color={rgb, 255:red, 74; green, 144; blue, 226 }  ,opacity=1 ,xscale=1.2,yscale=1.2]  {$s_{--}$};
\draw (340,87) node [anchor=north west][inner sep=0.75pt]  [color={rgb, 255:red, 74; green, 144; blue, 226 }  ,opacity=1 ,xscale=1.2,yscale=1.2]  {$s_{++}$};
\draw (305.44,77.72) node [anchor=north west][inner sep=0.75pt]  [color={rgb, 255:red, 74; green, 144; blue, 226 }  ,opacity=1 ,xscale=1.2,yscale=1.2]  {$s'_{++}$};
\draw (283.45,209.42) node [anchor=south west] [inner sep=0.75pt]  [color={rgb, 255:red, 74; green, 144; blue, 226 }  ,opacity=1 ,xscale=1.2,yscale=1.2]  {$s'_{--}$};
\draw (352.67,70.46) node [anchor=west] [inner sep=0.75pt]  [color={rgb, 255:red, 74; green, 144; blue, 226 }  ,opacity=1 ,xscale=1.2,yscale=1.2]  {\small $s_{--} +n'-1$};
\draw (380.06,39) node [anchor=west] [inner sep=0.75pt]  [color={rgb, 255:red, 74; green, 144; blue, 226 }  ,opacity=1 ,xscale=1.2,yscale=1.2]  {\small $s_{++} +n'-1$};
\draw (330.18,146.22) node [anchor=south west] [inner sep=0.75pt]  [color={rgb, 255:red, 0; green, 0; blue, 0 }  ,opacity=1 ,xscale=1.2,yscale=1.2]  {$I_{\wh{\F}}^{s_{0}}(\wh{z}) =I_{\wh{\F}}^{t_{0}}(\wh{y})$};
\draw (381.18,189.45) node [anchor=west] [inner sep=0.75pt]  [color={rgb, 255:red, 0; green, 0; blue, 0 }  ,opacity=1 ,xscale=1.2,yscale=1.2]  {$I_{\wh{\F}}^{s'_{0}}( T'\wh{z}) =I_{\wh{\F}}^{t'_{0}}(\wh{y})$};
\end{tikzpicture}

\caption{Proof of Proposition~\ref{LastPropBndDevIrrat}.\label{FigBndDevIrrat}}
\end{center}
\end{figure}
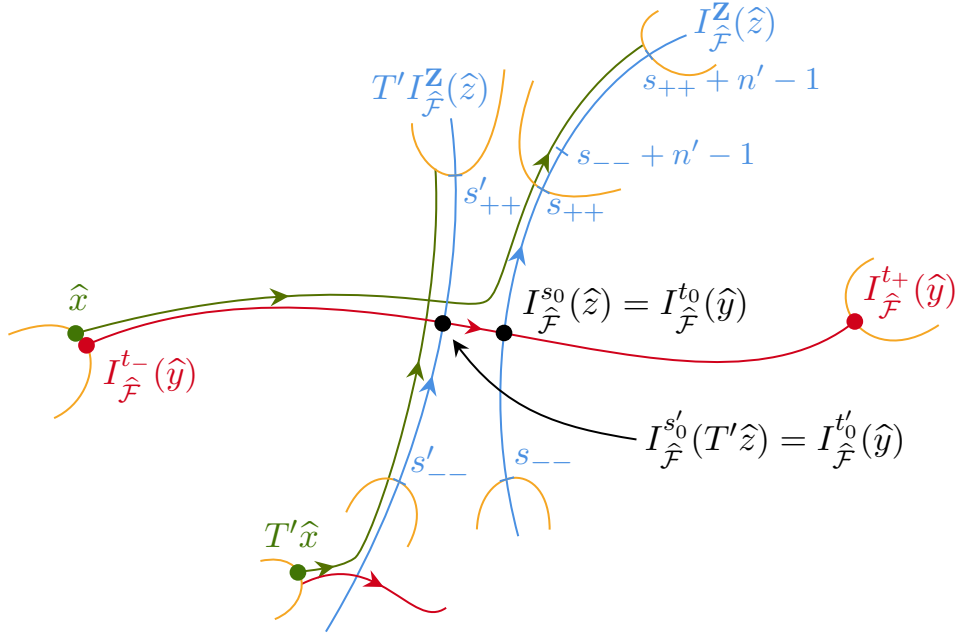

\begin{proof}
We argue by contradiction: suppose that there exist $t_-<t_+$, $t_0\in (t_-,t_+)$, $s_-<s_0<s_+$ such that the paths $I^{[t_-, t_+]}_{\wt \F}(\wt y) $ and $I^{[s_-, s_+]}_{\wt \F}(\wt z)$ intersect $\wt\F$-transversally at $I^{t_0}_{\wt \F}(\wt y) = I^{s_0}_{\wt \F}(\wt z)$. The idea of the proof is illustrated in Figure~\ref{FigBndDevIrrat}.

By \cite[Lemma~2.10]{paper1PAF}, there exists a neighbourhood $\wt W$ of $\wt z$ such that, for every $\wt x'\in W$, the path $I^{[s_-, s_+]}_{\wt\F}(\wt z)$ is $\wt\F$-equivalent to a subpath of $I^{[s_--1, s_+ +1]}_{\wt\F}(\wt x')$.

By Lemma~\ref{LemReturnDeckTrans}, there exist $n\ge s_+-s_- +1$ and $T'\in\G\setminus\{\Id\}$ whose axis is arbitrarily close to $\wt\gamma_{\wt z}$ such that $\wt f^n(\wt z)\in (T^\prime)^{-1}\wt W$. Hence, there exist $s'_-<s'_0<s'_+$ and $t'_0\in(t_-, t_+)$ such that the paths $I^{[t_-, t_+]}_{\wt \F}(\wt y) $ and $I^{[s'_-, s'_+]}_{\wt \F}(T'\wt z)$ intersect $\wt\F$-transversally at $I^{t'_0}_{\wt \F}(\wt y) = I^{s'_0}_{\wt \F}(T'\wt z)$ (for this step one can take $s'_- = s_-+n-1$ and $s_+ = s_++n+1$).

Exchanging $\wt z$ and $T'\wt z$ and changing $T'$ by $T^{\prime-1}$ if necessary\footnote{If this is not the case, one has to choose $n\le s_--s_+-1$ (hence very negative) but Lemma~\ref{LemReturnDeckTrans} also works in this case.}, we can suppose that $\wt\gamma_{T'\wt z}$ is on the left of $\wt\gamma_{\wt z}$.

By Lemma~\ref{LemNotEquivFiniteTime}, by choosing some $s_{--}\le s_-$, $s'_{--}\le s'_-$, $s_{++}\ge s_+$ and $s'_{++}\ge s'_+$, we get:
\begin{itemize}
\item $I^{[s_{--}, s_0]}_{\wt\F}(\wt z)$ is not $\wt\F$-equivalent to a subpath of $I^{[s'_{--}, s'_0]}_{\wt\F}(T'\wt z)$;
\item $I^{[s'_{--}, s'_0]}_{\wt\F}(T'\wt z)$ is not $\wt\F$-equivalent to a subpath of $I^{[s_{--}, s_0]}_{\wt\F}(\wt z)$;
\item $I^{[s_0, s_{++}]}_{\wt\F}(\wt z)$ is not $\wt\F$-equivalent to a subpath of $I^{[s'_0, s'_{++}]}_{\wt\F}(T'\wt z)$;
\item $I^{[s'_0, s'_{++}]}_{\wt\F}(T'\wt z)$ is not $\wt\F$-equivalent to a subpath of $I^{[s_0, s_{++}]}_{\wt\F}(\wt z)$.
\end{itemize}


\bigskip

Remark that we introduced a complication here, as it can happen that $s'_{--}\le s_0$. Let us use the recurrence of $z$ to solve this issue.

Because the point $z$ is recurrent, and more precisely by Lemma~\ref{LemReturnDeckTrans}, there exists $R_0\in\G$ and $n'\ge 0$ such that $n'$ is arbitrarily large, $\wt f^{n'}(\wt z)$ is arbitrarily close to $R_0\wt z$ and the geodesic axis of $R_0$ is arbitrarily close to $\wt\gamma_{\wt z}$.
By \cite[Lemma~2.10]{paper1PAF}, this implies that there exists $s_1\ge s_{++}-s'_{--}+1$ such that $I^{[s'_{--}, s'_{++}]}_{\wt\F}(\wt z)$ is $\wt\F$-equivalent to a subpath of $R_0I^{[s'_{--}+s_1-1, s'_{++} +s_1+1]}_{\wt\F}(\wt z)$. 

By the fundamental proposition of forcing \cite[Proposition~2.11]{paper1PAF}, there exists $\wt x\in\wt S$ such that the transverse path $I^{[t_-, t_0]}_{\wt\F}(\wt y) I^{[s_0, s'_{++} +s_1+1]}_{\wt\F}(\wt z)$ is $\wt\F$-equivalent to a subpath of $I^\Z_{\wt\F}(\wt x)$ (see Figure~\ref{FigBndDevIrrat}).

Let $\wh y$ and $\wh z$ be lifts of $\wt y$ and $\wt z$ to $\wh\dom(\F)$ such that $I^{t_0}_{\wh \F}(\wh y) = I^{s_0}_{\wh \F}(\wh z)$. We also consider ${\wh T',\wh  R_0}$ deck transformations of $\wh\dom (\F)$ that lift $T', R_0\in\G$ such that $I^{t'_0}_{\wh \F}(\wh y) = I^{s'_0}_{\wh \F}(\wh z)$ and that $I^{[s'_{--}, s'_{++}]}_{\wh\F}({\wh T}'\wh z)$ is $\wh\F$-equivalent to a subpath of ${\wh R_0}I^{[s'_{--}+n'-1, s'_{++}+n'+1]}_{\wh\F}(\wh z)$. Denote $\wh B$ the set of leaves of $\wh\F$ met by $I^{\Z}_{\wh \F}(\wh z)$

By the above properties, the paths $I^{[t_-, t_+]}_{\wh \F}(\wh y) $ and $I^{[s_{-}, s_{+}]}_{\wh \F}(\wh z)$ intersect $\wh\F$-transversally at $I^{t_0}_{\wh \F}(\wh y) = I^{s_0}_{\wh \F}(\wh z)$, and the paths $I^{[t_-, t_+]}_{\wh \F}(\wh y) $ and $I^{[s'_{-}, s'_{+}]}_{\wh \F}(T'\wh z)$ intersect $\wh\F$-transversally at $I^{t'_0}_{\wh \F}(\wh y) = I^{s'_0}_{\wh \F}(\wh T'\wh z)$.
The above properties also imply that $L(\wh\phi_{I^{s_-}_{\wh \F}(\wh z)}) \cup R(\wh\phi_{I^{s'_{++}}_{\wh \F}(\wh z)}) \subset R(\wh T'\wh B)$ and $L(\wh\phi_{I^{s'_{--}}_{\wh \F}(\wh T'\wh z)}) \cup R(\wh\phi_{I^{s'_{++}}_{\wh \F}(\wh T'\wh z)}) \subset L(\wh B)$. 
This means that the leaves $\wh\phi_{\wh y}$, $\wh\phi_{I^{s'_{++}}_{\wh \F}(\wh z)}$, $\wh\phi_{I^{s'_{--}}_{\wh \F}(\wh T'\wh z)}$ and $\wh\phi_{I^{s'_{++}}_{\wh \F}(\wh T'\wh z)}$ are indeed ordered as in Figure~\ref{FigBndDevIrrat}. 

Hence, the paths $I^{[t_-, t_0]}_{\wh\F}(\wh y) I^{[s_0, s'_{++}+n'+1]}_{\wh\F}(\wh z)$ and $I^{[s'_{--}, s'_{++}]}_{\wh \F}(\wh T'\wh z)$ intersect $\wh \F$-transversally, so the paths $I^\Z_{\wh\F}(\wh x)$ and $T^{\prime -1}\wh R_0I^\Z_{\wh\F}(\wh x)$ also intersect $\F$-transversally. 

By the fundamental theorem of existence of horseshoes \cite[Theorem~2.12]{paper1PAF}, this implies that there exists a point $\wt z_p\in \wt S$ and $\tau\in\N$ such that $\wt f^\tau(\wt z_p) = T^{\prime -1}R_0\wt z$ (in particular, the projection $z_p$ of $\wt z_p$ on $S$ is $f$-periodic). 

Recall that the geodesic axes of $T'$ and $R_0$ can be chosen arbitrarily close to $\wt\gamma_{\wt z}$ (see Lemma~\ref{LemReturnDeckTrans}). This is a contradiction with Lemma~\ref{LemGeodNotAccuPerOrb}.
\end{proof}

\subsection{Essential points and creation of transverse intersections}\label{SubSecEssen}

\begin{lemma}\label{LemEssentialPoints}
Let $\mu$ be a measure of a minimal non-closed class. For any $\varep>0$ and a $\mu$-typical point $z$, the union 
\[A_\varep = \bigcup_{n\in\N} f^{n}\big(B(z,\varep)\big)\]
is an open surface such that ${i}_*(\pi_1(S_\Lambda))\subset {i}_*(\pi_1(A_\varep))$.
\end{lemma}

\begin{rem}
The proof also implies that the same conclusion holds for a measure $\mu$ of a horseshoe class whose rotation vector is not rational, and whose tracking geodesic crosses any other tracking geodesic of the class.
\end{rem}

\begin{proof}
Suppose by contradiction that there exist $\varep>0$ and a simple essential loop $\alpha\subset S_\Lambda$ that is not freely homotopic (in $S$) to a loop in $A_\varep$. The associated closed geodesic (any simple essential closed loop is freely homotopic to a single simple closed geodesic) $\gamma_\alpha$ is included in $S_\Lambda$.

Let $z$ be a $\mu$-typical point, and $\wt z$ be a lift of $z$ to $\wt S$. 
Let $\wt A_\varep$ be the connected component of the set of points lifting points of $A_\varep$ containing $\wt z$.
Let $\wt\gamma_{\wt z}$ be the tracking geodesic of $\wt z$, $\wt\alpha$ a lift of $\alpha$ and $\wt\gamma_{\wt\alpha}$ the unique geodesic that is at finite distance to $\wt\alpha$ (and this geodesic lifts $\gamma_\alpha$); we choose the lift $\wt\alpha$ such that the geodesics $\wt\gamma_{\wt z}$ and $\wt\gamma_{\wt\alpha}$ cross. We denote $T\in\G$ a primitive deck transformation associated to $\wt\alpha$. 
Let $\wt \Gamma_{\wt A_\varep}$ be the union of the geodesics $\wt \gamma_{\wt\beta}$, where $\wt \beta\subset \wt A_\varep$ is a path lifting an essential loop. We denote $\wt S_{\wt A_\varep} = \conv\big(\wt \Gamma_{\wt A_\varep}\big)$, and $S_{A_\varep}$ the projection of $\wt S_{\wt A_\varep}$ on $S$.

\begin{claim}\label{ClaimEssentialPoints}
The surface $S_{A_\varep}$ contains a closed geodesic $\gamma_0$ crossing $\gamma_z$.
\end{claim}

\begin{proof}
By Lemma~\ref{LemReturnDeckTrans}, there exists $T_1\in\G\setminus\{\Id\}$ and $n_1\in\N$ such that $\wt f^{n_1}(\wt z)\in T_1 B(\wt z,\varep)$, with the axis of $T_1$ being close to $\wt\gamma_{\wt z}$. One can repeat the argument for $f^{n_1}$: there exists $T_2\in\G\setminus\{\Id\}$ and $k_2\in\N$ and $\wt f^{k_2n_1}(\wt z)\in T_2 B(\wt z,\varep)$, with the axis of $T_2$ being close to $\wt\gamma_{\wt z}$, close enough so that $T_2\notin \langle T_1\rangle$. But $\wt f^{n_1}(\wt z)\in T_1 B(\wt z,\varep) \subset T_1\wt A_\varep$, so $\wt f^{k_2n_1}(\wt z)\in  T_1^{k_2}\wt A_\varep \cap T_2 \wt A_\varep$. Hence $T_2^{-1}T_1^{k_2} \in \operatorname{stab}(\wt A_\varep) \setminus \{\Id\}$. By construction, the axis of $T_2^{-1}T_1^{k_2}$ is close to $\wt\gamma_{\wt z}$, its projection on $S$ is the closed geodesic we look for.
\end{proof}

The fact that $\alpha\subset S_\Lambda$ is not freely homotopic to a loop in $A_\varep$ implies that $\wt A_\varep \cap T\wt A_\varep = \emptyset$. So we also have $\wt S_{\wt A_\varep} \cap T\wt S_{\wt A_\varep} = \emptyset$.
By \cite[Lemma 6.7]{alepablo}, the boundary of $\wt S_{\wt A_\varep}$ is made of geodesics that project to simple closed geodesics of $S$; we deduce that there is a closed geodesic $\gamma_1$ of $S_\Lambda$ that is disjoint from $S_{A_\varep}$.

By the arguments of the proof of Lemma~\ref{LemGeodNotAccuPerOrb}, the closed geodesic $\gamma_1$ has to cross the tracking geodesic $\gamma_z$ for $\mu$-a.e.~$z\in S$. Let $\wt\gamma_1$ be a lift of $\gamma_1$ crossing $\wt\gamma_{\wt z}$.
By recurrence of the tracking geodesic $\gamma_z$ in $T^1S$ \cite[Theorem~C]{alepablo}, there exists a sequence $(T_n)_{n\in\Z}\in\G^\Z$ such that $T_n\wt \gamma_1\to_{n\to+\infty} \omega(\wt\gamma_{\wt z})$ and $T_n\wt \gamma_1\to_{n\to-\infty} \alpha(\wt\gamma_{\wt z})$ (for Hausdorff topology in the compactification $\wt S\cup \partial\wt S$). 
In particular, there exists $T_1^+\in \G$ such that $\wt S_{\wt A_\varep}$ lies inside the connected component $\wt E_1$ of the complement of $T_1^+\wt\gamma_1$ not containing $\omega(\wt\gamma_{\wt z})$. 

Consider the geodesic $\gamma_0$ given by Claim~\ref{ClaimEssentialPoints} and a lift $\wt\gamma_0$ to $\wt S$ crossing $\wt\gamma_{\wt z}$. By definition of $S_{A_\varep}$, there exists a closed loop $\alpha_0\subset A_\varep$ and a lift $\wt\alpha_0\subset T_0\wt A_\varep$ of it staying at finite distance to $\wt\gamma_0$. By the same arguments as before, there exists $T_2^+\in \G$ such that $T_2^+\wt\gamma_0$ lies inside the complement of $\wt E_1$. Hence, $T_2^+\wt\gamma_0$ is not contained in $\wt S_{\wt A_\varep}$, so and $T_2^+\wt\alpha_0$ is disjoint from ${\wt A_\varep}$. In particular, $\wt A_\varep$ is included in the connected component $E_0$ of the complement of $T_2^+\wt\alpha_0$ not containing $\omega(\wt\gamma_{\wt z})$.

As $A_\varep$ is an essential positively invariant open set, we have $f(\wt A_\varep)\subset \wt A_\varep$, in particular $\wt f^n(\wt z)\in \wt A_\varep$ for any $n\ge 0$. The fact that $\wt f^n(\wt z)\to \omega(\wt\gamma_{\wt z})$ contradicts the fact that $\wt A_\varep\subset E_0$.
\end{proof}


Let $X_\varep$ be a skeleton of $A_\varep$. By this we mean a closed connected set $X_\varep \subset A_\varep$, with empty interior, such that for any $x\in X_\varep$, we have ${i}_*(\pi_1(X_\varep, x)) = {i}_*(\pi_1(A_\varep, x))$, and that the complement of $X_\varep$ is connected. For example, such a set can be obtained as the union of closed loops generating the fundamental group of the fill of $A_\varep$ (the fill of a sub-surface is the union of the surface with the connected components of its complement that are topological discs). 

Let $\wt X_\varep$ be the lift of $X_\varep$ to $\wt S$ that is included in $\wt A_\varep$.

\begin{lemma}\label{LemAPlusNotEquiv}
There exist $\varep_0>0$ and $N>0, C_0>0$ such that the following is true: for any $\wt y\in \wt X_{\varep_0}\cap R(\wt\gamma_{\wt z})$ (resp. $\wt y\in \wt X_{\varep_0}\cap L(\wt\gamma_{\wt z})$) such that $d(\wt y, \wt\gamma_{\wt z})\ge C_0$, there exists $t_0\le N$ and $t_1\ge -N$ such that $I^{t_0}_{\wt \F}(\wt y), I^{t_1}_{\wt \F}(\wt y)\in R(\wt B)$ (resp. $I^{t_0}_{\wt \F}(\wt y), I^{t_1}_{\wt \F}(\wt y)\in L(\wt B)$).
\end{lemma}

Note that the conclusion of the lemma implies that $I^{t_0}_{\wt \F}(\wt y), I^{t_1}_{\wt \F}(\wt y)\notin \wt B$.

\begin{rem}\label{RemLemAPlusNotEquiv}
The proof of this lemma implies the following stronger conclusion: there exists a neighbourhood $V_\varep$ of $X_\varep$ such that for any skeleton $X'_\varep \subset V_\varep$, the same conclusion holds for $X'_\varep$ instead of $X_\varep$. Some results below (Proposition~\ref{PropExistInterTransIrrat} and Theorem~\ref{ThmBndDevIrrat}) will also hold for $X'_\varep$ instead of $X_\varep$.
\end{rem}

\begin{proof}
Let $\varep_0>0$ be given by Lemma~\ref{LemNotEquivFiniteTime} that work for both $T_0$ and $T_1$. Let also $C_1>0$ given by Lemma~\ref{LemIfFarThenNotEquiv}.

Let $\wt K_{\varep_0}\subset \wt X_{\varep_0}$ be a compact set such that $\wt X_{\varep_0} = \bigcup_{T'\in\G} T'\wt K_{\varep_0}$ ($\wt K_{\varep_0}$ is a ``fundamental domain'' for $\wt X_{\varep_0}$).

For any $\wt x\in \wt K_{\varep_0}$, there exists $n_{x}\in \N$ such that $f^{{-n_{x}}}(x)\in B(z,\varep_0)$. So there exists $T_{\wt x}\in\G$ such that $\wt f^{{-n_{x}}}(\wt x)\in T_{\wt x} B(\wt z,\varep_0)$. As $B(\wt z,\varep_0)$ is open, there exists $\eta_{\wt x}>0$ such that for any $\wt x'\in B(\wt x, \eta_{\wt x})$, we have $\wt f^{{-n_{x}}}(\wt x')\in T_{\wt x} B(\wt z,\varep_0)$.

As $\wt K_{\varep_0}$ is compact, there is a finite set $\wt x_1,\dots,\wt x_k\in \wt K_{\varep_0}$ such that $\wt K_{\varep_0} \subset \bigcup_{1\le j\le k} B(\wt x_j, \eta_{\wt x_j})$. Hence, for any $\wt y'\in \wt K_{\varep_0}$, there exists $j\le k$ such that $\wt f^{{-n_{x_j}}}(\wt y')\in T_{\wt x_j} B(\wt z,\varep_0)$. {Set $N=\max_{j\in\{1,\ldots,k\}}n_{x_j}+1+t_1'-t_0'$}

Let 
\[C_0 = \sup_{1\le j\le k} d\big(\wt f^{n_{x_j}},\Id_{\wt S}\big)+C_1+\varep_0+1,\]
and pick $\wt y\in \wt X_{\varep_0}\cap R(\wt\gamma_{\wt z})$ such that $d(\wt y, \wt\gamma_{\wt z})\ge C_0$.
By the definition of $\wt K_{\varep_0}$, there exists $T'_{\wt y}\in\G$ such that $\wt y\in T'_{\wt y} \wt K_{\varep_0}$. 
Hence, there exists $j = j_{\wt y}\le k$ such that $\wt f^{-n_{\wt x_j}}(T^{\prime\, -1}_{\wt y} \wt y)\in T_{x_j} B(\wt z,\varep_0)$; in other words, setting $S_{\wt y} = T'_{\wt y} T_{\wt x_j}$, we have  $\wt f^{n_{x_j}}(\wt y)\in S_{\wt y} B(\wt z,\varep_0)$.

By the definition of $C_0$, we have 
\begin{align*}
d\big(S_{\wt y}\wt z,\, \wt\gamma_{\wt z}\big) & \ge d\big(\wt f^{n_{x_j}}(\wt y),\, \wt\gamma_{\wt z}\big)-\varep_0\\
& \ge d\big(\wt y,\, \wt\gamma_{\wt z}\big)-\varep_0 - d\big(\wt f^{n_{x_j}},\Id_{\wt S}\big)\\
& \ge C_1+1>C_1.
\end{align*}
It allows us to apply Lemma~\ref{LemIfFarThenNotEquiv} to $T' = S_{\wt y}$: there exists $\overline{t_0}\in[t_0', 1]$ and $t_1'\in [-1, t'_1]$ such that $\wt\phi_{I^{\overline{t_0}}_{\wt \F}({\wt f^{-n_{x_j}}(\wt y)})}\cup \wt\phi_{I^{\overline{t_1}}_{\wt \F}({\wt f^{-n_{x_j}}(\wt y)})} = {\wt\phi_{I^{-n_{x_j}}_{\wt \F}(\wt y)}\cup\wt\phi_{I^{-n_{x_j}}_{\wt \F}(\wt y)} \subset R(\wt B)}$ and the lemma follows from taking $t_0=\overline{t_0}-n_{x_j}\le N,\, t_1=\overline{t_1}-n_{x_j}\ge -N$.
\end{proof}

\begin{prop}\label{PropExistInterTransIrrat}
There exist $C_0>0$ and $\varep_0>0$ such that the following is true.
Suppose that there exists $\wt y_0$ belonging to a connected component $\wt {\mathcal X}_{\wt y_0}^+$ of the complement $\wt X_{\varep_0}$ and $n_0\in\N$ such that $\wt f^{n_0}(\wt y_0)$ belongs to a connected component $\wt {\mathcal X}_{\wt y_0}^+$ of the complement $\wt X_{\varep_0}$, with $\wt {\mathcal X}_{\wt y_0}^- \subset L(V_{C_0}(\wt\gamma_{\wt z}))$ and $\wt {\mathcal X}_{\wt y_0}^+ \subset R(V_{C_0}(\wt\gamma_{\wt z}))$.

Then there exists $\wt y_1\in\wt S$ such that the paths $I^{\Z}_{\wt \F}(\wt y_1) $ and $I^{\Z}_{\wt \F}(\wt z)$ intersect $\wt\F$-transversally.
\end{prop}

\begin{figure}
\begin{center}

\tikzset{every picture/.style={line width=0.75pt}} 

\begin{tikzpicture}[x=0.75pt,y=0.75pt,yscale=-1.2,xscale=1.2]

\draw  [color={rgb, 255:red, 189; green, 16; blue, 224 }  ,draw opacity=0.2 ][fill={rgb, 255:red, 189; green, 16; blue, 224 }  ,fill opacity=0.1 ] (415.72,59.5) .. controls (440.61,81.5) and (444.39,200.39) .. (415.28,238.61) .. controls (380.61,200.39) and (383.5,83.94) .. (415.72,59.5) -- cycle ;
\draw  [draw opacity=0][fill={rgb, 255:red, 176; green, 248; blue, 28 }  ,fill opacity=0.5 ] (325.21,181.59) .. controls (319.67,168.21) and (318.84,141.66) .. (321.51,132.04) .. controls (325.58,133.29) and (326.58,134.21) .. (328.83,132.04) .. controls (331.33,131.71) and (332.25,130.71) .. (333.25,129.13) .. controls (337.17,134.54) and (351.42,128.96) .. (354.42,129.54) .. controls (349.08,141.96) and (362.67,144.88) .. (366.75,153.63) .. controls (354.17,154.88) and (348.33,164.21) .. (346.33,175.04) .. controls (339.33,173.38) and (335.08,178.29) .. (334.33,181.88) .. controls (330.58,178.88) and (328.25,180.38) .. (325.21,181.59) -- cycle ;
\draw  [draw opacity=0][fill={rgb, 255:red, 176; green, 248; blue, 28 }  ,fill opacity=0.5 ] (325.21,181.59) .. controls (319.67,168.21) and (318.84,141.66) .. (321.51,132.04) .. controls (325.58,133.29) and (328.13,138.97) .. (330.87,134.43) .. controls (333.97,137.74) and (340.44,140.51) .. (347.82,137.44) .. controls (351.74,142.85) and (396.13,142.05) .. (413.21,128.51) .. controls (426.44,144.97) and (477.43,152.07) .. (481.51,160.82) .. controls (465.51,163.74) and (410.28,149.37) .. (408.28,160.21) .. controls (338.13,129.74) and (342.52,175.93) .. (341.77,179.52) .. controls (334.44,171.59) and (328.25,180.38) .. (325.21,181.59) -- cycle ;
\draw  [draw opacity=0][fill={rgb, 255:red, 176; green, 248; blue, 28 }  ,fill opacity=0.5 ] (101.72,150.01) .. controls (108.58,147.21) and (111.67,145.88) .. (114.34,136.26) .. controls (126.27,144.54) and (128.36,139.79) .. (137.19,147.21) .. controls (131.92,152.04) and (126.82,155.21) .. (124.88,167.11) .. controls (119.08,159.57) and (116.33,154.75) .. (101.72,150.01) -- cycle ;
\draw  [draw opacity=0][fill={rgb, 255:red, 248; green, 231; blue, 28 }  ,fill opacity=0.5 ] (107.02,164.93) .. controls (108.69,156.02) and (127.27,140.35) .. (131.73,130.12) .. controls (133.77,131.93) and (133.94,130.77) .. (136.02,132.68) .. controls (128.86,148.6) and (121.11,163.35) .. (113.36,173.52) .. controls (113.11,169.85) and (112.77,165.68) .. (107.02,164.93) -- cycle ;
\draw  [draw opacity=0][fill={rgb, 255:red, 248; green, 231; blue, 28 }  ,fill opacity=0.5 ] (219.08,134.29) .. controls (226.41,135.4) and (229.75,135.85) .. (236.97,128.96) .. controls (242.97,142.18) and (247.19,139.18) .. (250.97,150.07) .. controls (243.97,151.52) and (237.97,151.63) .. (230.19,160.85) .. controls (229.08,151.4) and (229.19,145.85) .. (219.08,134.29) -- cycle ;
\draw [color={rgb, 255:red, 155; green, 155; blue, 155 }  ,draw opacity=1 ]   (209.19,130.68) .. controls (221.3,136.9) and (234.97,138.24) .. (242.52,119.79) ;
\draw [color={rgb, 255:red, 155; green, 155; blue, 155 }  ,draw opacity=1 ]   (251.41,154.46) .. controls (251.41,138.46) and (230.97,136.68) .. (238.52,118.24) ;
\draw [color={rgb, 255:red, 155; green, 155; blue, 155 }  ,draw opacity=1 ]   (226.3,168.46) .. controls (229.3,157.79) and (239.86,149.79) .. (255.19,150.02) ;
\draw [color={rgb, 255:red, 155; green, 155; blue, 155 }  ,draw opacity=1 ]   (229.63,167.79) .. controls (231.86,152.02) and (225.19,136.68) .. (211.19,128.68) ;

\draw [color={rgb, 255:red, 189; green, 16; blue, 224 }  ,draw opacity=1 ]   (176.56,60.33) .. controls (172.11,136.33) and (171.22,159.89) .. (176.11,239.44) ;
\draw  [color={rgb, 255:red, 189; green, 16; blue, 224 }  ,draw opacity=0.2 ][fill={rgb, 255:red, 189; green, 16; blue, 224 }  ,fill opacity=0.1 ] (176.56,60.33) .. controls (201.44,82.33) and (205.22,201.22) .. (176.11,239.44) .. controls (141.44,201.22) and (144.33,84.78) .. (176.56,60.33) -- cycle ;
\draw  [line width=1.5]  (80,150) .. controls (80,100.29) and (120.29,60) .. (170,60) .. controls (219.71,60) and (260,100.29) .. (260,150) .. controls (260,199.71) and (219.71,240) .. (170,240) .. controls (120.29,240) and (80,199.71) .. (80,150) -- cycle ;
\draw [color={rgb, 255:red, 155; green, 155; blue, 155 }  ,draw opacity=1 ]   (91.38,151.98) .. controls (104.97,151.11) and (117.38,145.25) .. (114.41,125.54) ;
\draw [color={rgb, 255:red, 155; green, 155; blue, 155 }  ,draw opacity=1 ]   (139.82,150.75) .. controls (131.61,137.01) and (113.15,145.97) .. (110.18,126.26) ;
\draw [color={rgb, 255:red, 155; green, 155; blue, 155 }  ,draw opacity=1 ]   (125.44,175.64) .. controls (122.55,164.95) and (127.51,152.66) .. (140.78,144.99) ;
\draw [color={rgb, 255:red, 155; green, 155; blue, 155 }  ,draw opacity=1 ]   (127.96,173.36) .. controls (121.78,158.67) and (108.19,148.93) .. (92.07,149.24) ;
\draw [color={rgb, 255:red, 155; green, 155; blue, 155 }  ,draw opacity=1 ]   (127,131.21) .. controls (136.43,125.07) and (142.29,145.79) .. (142.98,131.83) ;
\draw [color={rgb, 255:red, 155; green, 155; blue, 155 }  ,draw opacity=1 ]   (132,126.79) .. controls (134.86,134.5) and (100.43,159.36) .. (107.43,171.21) ;
\draw [color={rgb, 255:red, 155; green, 155; blue, 155 }  ,draw opacity=1 ]   (137.86,129.36) .. controls (132.71,139.93) and (117.57,175.07) .. (108.29,177.36) ;
\draw [color={rgb, 255:red, 155; green, 155; blue, 155 }  ,draw opacity=1 ]   (101.87,167.07) .. controls (111.3,160.93) and (115.57,171.03) .. (111.43,179.79) ;
\draw [color={rgb, 255:red, 208; green, 2; blue, 27 }  ,draw opacity=1 ]   (119.86,153.43) .. controls (159.86,123.43) and (193.86,179.07) .. (233.86,149.07) ;
\draw [shift={(233.86,149.07)}, rotate = 323.13] [color={rgb, 255:red, 208; green, 2; blue, 27 }  ,draw opacity=1 ][fill={rgb, 255:red, 208; green, 2; blue, 27 }  ,fill opacity=1 ][line width=0.75]      (0, 0) circle [x radius= 1.34, y radius= 1.34]   ;
\draw [shift={(119.86,153.43)}, rotate = 323.13] [color={rgb, 255:red, 208; green, 2; blue, 27 }  ,draw opacity=1 ][fill={rgb, 255:red, 208; green, 2; blue, 27 }  ,fill opacity=1 ][line width=0.75]      (0, 0) circle [x radius= 1.34, y radius= 1.34]   ;
\draw [color={rgb, 255:red, 208; green, 193; blue, 3 }  ,draw opacity=1 ]   (121.22,195.67) .. controls (118.22,188.03) and (125.05,185.58) .. (115.59,168.2) ;
\draw [shift={(114.14,165.64)}, rotate = 59.59] [fill={rgb, 255:red, 208; green, 193; blue, 3 }  ,fill opacity=1 ][line width=0.08]  [draw opacity=0] (7.14,-3.43) -- (0,0) -- (7.14,3.43) -- (4.74,0) -- cycle    ;
\draw [color={rgb, 255:red, 143; green, 210; blue, 5 }  ,draw opacity=1 ]   (107.29,112.21) .. controls (103.12,119.87) and (99.34,135.98) .. (110.88,144.89) ;
\draw [shift={(113.29,146.5)}, rotate = 210.29] [fill={rgb, 255:red, 143; green, 210; blue, 5 }  ,fill opacity=1 ][line width=0.08]  [draw opacity=0] (7.14,-3.43) -- (0,0) -- (7.14,3.43) -- (4.74,0) -- cycle    ;
\draw  [draw opacity=0][fill={rgb, 255:red, 0; green, 0; blue, 0 }  ,fill opacity=1 ] (110.57,155.18) .. controls (110.57,154.25) and (111.32,153.5) .. (112.25,153.5) .. controls (113.18,153.5) and (113.93,154.25) .. (113.93,155.18) .. controls (113.93,156.11) and (113.18,156.86) .. (112.25,156.86) .. controls (111.32,156.86) and (110.57,156.11) .. (110.57,155.18) -- cycle ;
\draw [color={rgb, 255:red, 189; green, 16; blue, 224 }  ,draw opacity=1 ]   (415.72,59.5) .. controls (411.28,135.5) and (410.39,159.06) .. (415.28,238.61) ;
\draw [color={rgb, 255:red, 155; green, 155; blue, 155 }  ,draw opacity=1 ]   (368.08,157.46) .. controls (365.51,142.36) and (346.9,145.59) .. (356.13,125.28) ;
\draw [color={rgb, 255:red, 155; green, 155; blue, 155 }  ,draw opacity=1 ]   (345.6,179.6) .. controls (347.21,164.97) and (357.05,150.97) .. (371.36,154.21) ;
\draw [color={rgb, 255:red, 155; green, 155; blue, 155 }  ,draw opacity=1 ]   (334.91,186.96) .. controls (332.02,176.27) and (344.59,171.9) .. (350.44,177.28) ;
\draw [color={rgb, 255:red, 155; green, 155; blue, 155 }  ,draw opacity=1 ]   (337.05,186.67) .. controls (335.51,180.97) and (331.21,179.28) .. (325.21,181.59) ;
\draw [color={rgb, 255:red, 155; green, 155; blue, 155 }  ,draw opacity=1 ]   (331.97,126.97) .. controls (337.58,136.81) and (347.36,126.67) .. (361.67,129.9) ;
\draw [color={rgb, 255:red, 155; green, 155; blue, 155 }  ,draw opacity=1 ]   (325.51,128.36) .. controls (326.44,134.51) and (333.67,131.59) .. (335.21,124.97) ;
\draw [color={rgb, 255:red, 155; green, 155; blue, 155 }  ,draw opacity=1 ]   (320.59,131.59) .. controls (325.05,133.28) and (328.74,134.82) .. (330.28,128.21) ;
\draw  [draw opacity=0][fill={rgb, 255:red, 248; green, 231; blue, 28 }  ,fill opacity=0.5 ] (490.84,111.26) .. controls (497.45,124.15) and (500.43,150.54) .. (498.55,160.35) .. controls (494.39,159.43) and (493.32,158.6) .. (491.25,160.94) .. controls (488.79,161.48) and (487.95,162.55) .. (487.09,164.21) .. controls (482.74,159.13) and (468.99,165.85) .. (465.96,165.51) .. controls (470.26,152.7) and (456.49,150.9) .. (451.71,142.51) .. controls (464.15,140.24) and (469.2,130.47) .. (470.32,119.51) .. controls (477.43,120.6) and (481.26,115.35) .. (481.72,111.72) .. controls (485.7,114.41) and (487.91,112.72) .. (490.84,111.26) -- cycle ;
\draw [color={rgb, 255:red, 155; green, 155; blue, 155 }  ,draw opacity=1 ]   (450.07,138.8) .. controls (453.86,153.64) and (472.15,148.91) .. (464.6,169.9) ;
\draw [color={rgb, 255:red, 155; green, 155; blue, 155 }  ,draw opacity=1 ]   (470.67,114.9) .. controls (470.27,129.61) and (461.59,144.36) .. (447.07,142.31) ;
\draw [color={rgb, 255:red, 155; green, 155; blue, 155 }  ,draw opacity=1 ]   (480.73,106.7) .. controls (484.48,117.12) and (472.31,122.5) .. (466.05,117.61) ;
\draw [color={rgb, 255:red, 155; green, 155; blue, 155 }  ,draw opacity=1 ]   (478.62,107.16) .. controls (480.62,112.71) and (485.05,114.05) .. (490.84,111.26) ;
\draw [color={rgb, 255:red, 155; green, 155; blue, 155 }  ,draw opacity=1 ]   (488.53,166.25) .. controls (482.15,156.9) and (473.22,167.8) .. (458.7,165.75) ;
\draw [color={rgb, 255:red, 155; green, 155; blue, 155 }  ,draw opacity=1 ]   (494.86,164.34) .. controls (493.44,158.28) and (486.47,161.78) .. (485.48,168.5) ;
\draw [color={rgb, 255:red, 155; green, 155; blue, 155 }  ,draw opacity=1 ]   (498.55,160.35) .. controls (493.97,159.02) and (491.12,158.16) .. (490.12,164.88) ;
\draw [color={rgb, 255:red, 0; green, 98; blue, 212 }  ,draw opacity=1 ]   (470.07,157.87) .. controls (481.38,142.1) and (489.95,140.38) .. (498.24,139.52) ;
\draw  [line width=1.5]  (319.17,149.17) .. controls (319.17,99.46) and (359.46,59.17) .. (409.17,59.17) .. controls (458.87,59.17) and (499.17,99.46) .. (499.17,149.17) .. controls (499.17,198.87) and (458.87,239.17) .. (409.17,239.17) .. controls (359.46,239.17) and (319.17,198.87) .. (319.17,149.17) -- cycle ;
\draw [color={rgb, 255:red, 155; green, 155; blue, 155 }  ,draw opacity=1 ]   (321.51,132.04) .. controls (325.97,133.73) and (330.13,139.74) .. (331.67,133.13) ;
\draw [color={rgb, 255:red, 155; green, 155; blue, 155 }  ,draw opacity=1 ]   (329.85,132.61) .. controls (330.77,138.76) and (346.88,141.41) .. (348.42,134.79) ;
\draw [color={rgb, 255:red, 155; green, 155; blue, 155 }  ,draw opacity=1 ]   (345.56,134.81) .. controls (351.16,144.64) and (394.67,141.79) .. (416.33,127.29) ;
\draw [color={rgb, 255:red, 155; green, 155; blue, 155 }  ,draw opacity=1 ]   (483.75,164.63) .. controls (481.18,149.53) and (416.75,144.13) .. (412.58,125.46) ;
\draw [color={rgb, 255:red, 155; green, 155; blue, 155 }  ,draw opacity=1 ]   (483.67,159.63) .. controls (469.67,166.13) and (406,144.96) .. (407.67,163.5) ;
\draw [color={rgb, 255:red, 208; green, 2; blue, 27 }  ,draw opacity=1 ]   (355.5,148.79) .. controls (395.5,118.79) and (430.07,187.87) .. (470.07,157.87) ;
\draw [shift={(470.07,157.87)}, rotate = 323.13] [color={rgb, 255:red, 208; green, 2; blue, 27 }  ,draw opacity=1 ][fill={rgb, 255:red, 208; green, 2; blue, 27 }  ,fill opacity=1 ][line width=0.75]      (0, 0) circle [x radius= 1.34, y radius= 1.34]   ;
\draw [shift={(355.5,148.79)}, rotate = 323.13] [color={rgb, 255:red, 208; green, 2; blue, 27 }  ,draw opacity=1 ][fill={rgb, 255:red, 208; green, 2; blue, 27 }  ,fill opacity=1 ][line width=0.75]      (0, 0) circle [x radius= 1.34, y radius= 1.34]   ;
\draw [color={rgb, 255:red, 155; green, 155; blue, 155 }  ,draw opacity=1 ]   (341.27,181.08) .. controls (341.5,150.04) and (363.75,140.96) .. (409.83,160.71) ;
\draw [color={rgb, 255:red, 155; green, 155; blue, 155 }  ,draw opacity=1 ]   (343.17,181.29) .. controls (335.5,171.63) and (331.21,179.28) .. (325.21,181.59) ;

\draw (123.57,151.69) node [anchor=north west][inner sep=0.75pt]  [font=\small,color={rgb, 255:red, 178; green, 3; blue, 25 }  ,opacity=1 ]  {$\wt{y}_{0}$};
\draw (232.21,150.95) node [anchor=north west][inner sep=0.75pt]  [font=\small,color={rgb, 255:red, 178; green, 3; blue, 25 }  ,opacity=1 ]  {$\wt{f}^{n_{0}}(\wt{y}_{0})$};
\draw (234.72,130.15) node [anchor=south east] [inner sep=0.75pt]  [font=\small,color={rgb, 255:red, 170; green, 150; blue, 0 } ,opacity=1 ]  {$\wt\X_{\wt y_0}^{+}$};
\draw (145.11,195.37) node [anchor=north] [inner sep=0.75pt]  [font=\small,color={rgb, 255:red, 170; green, 150; blue, 0 }  ,opacity=1 ]  {$\wt f^{-n_{0}}( \wt\X_{\wt y_0}^{+})$};
\draw (113.36,113.81) node [anchor=south] [inner sep=0.75pt]  [font=\small,color={rgb, 255:red, 80; green, 200; blue, 0 }  ,opacity=1 ]  {$\wt \X_{\wt y_0}^{-}$};
\draw (94.6,148.48) node [anchor=north west][inner sep=0.75pt]  [font=\small,color={rgb, 255:red, 0; green, 0; blue, 0 }  ,opacity=1 ]  {$\wt{y}_{1}$};
\draw (353.2,149.64) node [anchor=north west][inner sep=0.75pt]  [font=\small,color={rgb, 255:red, 178; green, 3; blue, 25 }  ,opacity=1 ]  {$\wt{y}_{0}$};
\draw (462.43,160.5) node [anchor=north west][inner sep=0.75pt]  [font=\small,color={rgb, 255:red, 178; green, 3; blue, 25 }  ,opacity=1 ]  {$\wt{f}^{n_{0}}(\wt{y}_{0})$};
\draw (490.55,142.63) node [anchor=south east] [inner sep=0.75pt]  [font=\small,color={rgb, 255:red, 0; green, 98; blue, 212 }  ,opacity=1 ]  {$\wt{\delta }$};
\draw (501.52,133.72) node [anchor=west] [inner sep=0.75pt]  [font=\small,color={rgb, 255:red, 0; green, 98; blue, 212 }  ,opacity=1 ]  {$\chi $};
\draw (169.82,94.87) node [anchor=south east] [inner sep=0.75pt]  [font=\small,color={rgb, 255:red, 177; green, 5; blue, 212 }  ,opacity=1 ]  {$V_{C_0}( \wt\gamma_{\wt z} )$};
\draw (411.56,89.87) node [anchor=south east] [inner sep=0.75pt]  [font=\small,color={rgb, 255:red, 177; green, 5; blue, 212 }  ,opacity=1 ]  {$V_{C_0}( \wt\gamma_{\wt z} )$};
\draw (348.44,123.48) node [anchor=south] [inner sep=0.75pt]  [font=\small,color={rgb, 255:red, 80; green, 200; blue, 0 }  ,opacity=1 ]  {$\wt\X_{\wt y_0}^{-}$};
\draw (415.78,163.88) node [anchor=north] [inner sep=0.75pt]  [font=\small,color={rgb, 255:red, 80; green, 200; blue, 0 },opacity=1 ]  {$\wt f^{n_0}( \wt\X_{\wt y_0}^{-})$};
\draw (474.44,112.49) node [anchor=south east] [inner sep=0.75pt]  [font=\small,color={rgb, 255:red, 170; green, 150; blue, 0 } ,opacity=1 ]  {$\wt \X_{\wt y_0}^{+}$};

\end{tikzpicture}

\caption{Proof of Proposition~\ref{PropExistInterTransIrrat}: the case where $X_{\varep_0}$ is a skeleton of the surface (left) and the case it is not (right).}\label{FigPropExistInterTransIrrat}
\end{center}
\end{figure}
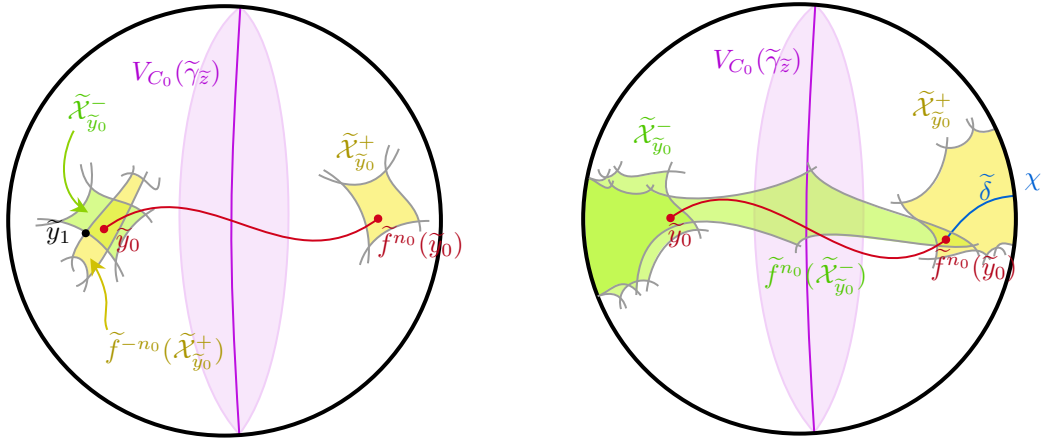

\begin{proof}
We do the proof for $n_0>0$, the case $n_0<0$ being identical. See Figure~\ref{FigPropExistInterTransIrrat}.
Let $C_0$ and $\varep_0$ given by Lemma~\ref{LemAPlusNotEquiv}.
Suppose that there exists $\wt y_0$ belonging to a connected component $\wt {\mathcal X}_{\wt y_0}^-$ of the complement $\wt X_{\varep_0}$ and $n_0>0$ such that $\wt f^{n_0}(\wt y_0)$ belongs to a connected component $\wt {\mathcal X}_{\wt y_0}^+$ of the complement $\wt X_{\varep_0}$, with $\wt {\mathcal X}_{\wt y_0}^- \subset L(V_{C_0}(\wt\gamma_{\wt z}))$ and $\wt {\mathcal X}_{\wt y_0}^+ \subset R(V_{C_0}(\wt\gamma_{\wt z}))$.

Let us prove that there exists $\wt y_1\in \wt X_{\varep_0} \cap L(V_{C_0}(\wt\gamma_{\wt z}))$, such that $\wt f^{n_0}(\wt y_1)\in \wt X_{\varep_0}\cap R(V_{C_0}(\wt\gamma_{\wt z}))$. 

Indeed, we have two cases. The first is when $X_\varep$ is a skeleton of $S$ (\emph{i.e.}~${i}_*(\pi_1(X_\varep)) = \pi_1(S)$), equivalently when $\wt \X_{y_0}^-$ and $\wt \X_{y_0}^+$ are fundamentals domains of $S$, and hence so is $\wt f^{-n_0}(\wt \X_{y_0}^+)$. The two fundamental domains $\wt \X_{y_0}^-$ and $\wt f^{-n_0}(\wt \X_{y_0}^+)$ intersect (at $\wt y_0$), hence their boundaries intersect at one point $\wt y_1 \in \wt X_{\varep_0}^- \cap \wt f^{-n_0}(\wt X_{\varep_0}^+)$, with the property that $\wt y_1\in L(V_{C_0}(\wt\gamma_{\wt z}))$ and $\wt f^{n_0}(\wt y_1)\in R(V_{C_0}(\wt\gamma_{\wt z}))$.

The second case is when $X_\varep$ is not a skeleton of $S$, and hence $\wt \X_{y_0}^-$ and $\wt \X_{y_0}^+$ are unbounded, in other words their boundaries contain points of $\partial\wt S$. As $\wt \X_{y_0}^+$ is a topological disc, it is path connected. Its projection on $S$ contains an essential loop, hence there is a path included in $\wt \X_{y_0}^+$ with well defined $\omega$-limit belonging to ${\partial \wt S}$. We deduce that there is a path $\wt \delta : [0,1]\to \wt \X_{y_0}^+ \cup {\partial \wt S}$ linking $\wt f^{n_0}(\wt y_0)$ to a point $\chi\in {\partial \wt S}$. Note that by our construction, $\chi$ is an endpoint of a deck transformation, hence it cannot be an endpoint of $\wt\gamma_{\wt z}$. We deduce that $\chi\in R(\wt\gamma_{\wt z})$, while $\overline{\wt \X_{y_0}^-} \cap R(\wt\gamma_{\wt z}) = \emptyset$. As $\wt f|_{\partial \wt S} = \Id_{\partial \wt S}$, we deduce that 
\[\wt f^{n_0}\big(\overline{\wt \X_{y_0}^-}\big) \cap {\partial\wt S} \cap R(\wt\gamma_{\wt z}) = \emptyset.\]
As a consequence, $\wt\delta$ links $\wt f^{n_0}(\wt y_0) \in \wt f^{n_0}\big(\overline{\wt \X_{y_0}^-}\big)$ to $\chi \notin \wt f^{n_0}\big(\overline{\wt \X_{y_0}^-}\big)$. We deduce that $\wt\delta$ meets $\partial \wt f^{n_0}\big({\wt \X_{y_0}^-}\big)$, in particular $\partial \wt f^{n_0}\big({\wt \X_{y_0}^-}\big)$ meets $\wt \X_{y_0}^+$. As it also meets the complement of $\wt \X_{y_0}^+$ (consider a point in $\wt f^{n_0}\big(\overline{\wt \X_{y_0}^-}\big)\cap \partial \wt S$), we deduce that 
\[\partial \wt f^{n_0}\big({\wt \X_{y_0}^-}\big) \cap \partial {\wt \X_{y_0}^-} \neq\emptyset.\]
This proves that there exists $\wt y_1\in \wt X_{\varep_0}^-$ with $\wt y_1\in L(V_{C_0}(\wt\gamma_{\wt z}))$, such that $\wt f^{n_0}(\wt y_1)\in \wt X_{\varep_0}^+$ and $\wt f^{n_0}(\wt y_1)\in R(V_{C_0}(\wt\gamma_{\wt z}))$.
\bigskip

Let us apply Lemma~\ref{LemAPlusNotEquiv} to both $\wt y_1$ and $\wt f^{n_0}(\wt y_1)$: 
there exist $t_-\le N$ and $t_+\ge -N$ such that
\[I^{t_-}_{\wt \F}(\wt y_1)\in L(\wt B)\qquad \text{and}\qquad I^{t_+-n_0}_{\wt \F}\big(\wt f^{n_0}(\wt y_1)\big) = I^{t_+ }_{\wt \F}(\wt y)\in R(\wt B).\]
Note that one can suppose that {$n_0> 2N+2$} (and hence $t_-<t_+$), taking {$C_0>(2N+2)d(\wt f, \Id_{\wt S})$} if necessary.
Moreover, as $I^{[t_-, t_+]}_{\wt \F}(\wt y_1)$ crosses $\wt B$ and hence $I^\Z_{\wt \F}(\wt z)$, there exist $t_0\in (t_-,t_+)$ and $s_0\in\R$ such that 
\[I^{t_0}_{\wt \F}(\wt y_1) = I^{s_0}_{\wt \F}(\wt z).\] 
Finally, if the path $I^\Z_{\wt \F}(\wt z)$ accumulated in $I^{[t_-, t_+]}_{\wt \F}(\wt y_1)$, \cite[Proposition~2.5]{paper1PAF} would imply the existence of $T\in\G\setminus \{\Id\}$ such that $\wt B$ is $T$-invariant. By \cite[Proposition~2.18]{paper1PAF}, this would imply that there is a sequence $(t'_k)\in\N^\N$, with $t_k\to +\infty$, such that the points $I^{t'_k}_{\wt \F}(\wt z)$ stay at finite distance to the geodesic axis of $T$. This is impossible as the tracking geodesic of $z$ is not closed and hence cannot be equal to the geodesic axis of $T$. Hence, there exist $s_-<s_0<s_+$ such that 
\[\wt\phi_{I^{s_-}_{\wt \F}(\wt z)} \cap I^{[t_-, t_+]}_{\wt \F}(\wt y_1) = \wt\phi_{I^{s_+}_{\wt \F}(\wt z)} \cap I^{[t_-, t_+]}_{\wt \F}(\wt y_1) = \emptyset.\]
This implies (see \cite[Lemma~2.17]{paper1PAF}) that the paths $I^{[t_-, t_+]}_{\wt \F}(\wt y_1) $ and $I^{[s_-, s_+]}_{\wt \F}(\wt z)$ intersect $\wt\F$-transversally at $I^{t_0}_{\wt \F}(\wt y_1) = I^{s_0}_{\wt \F}(\wt z)$.
\end{proof}

\subsection{Proofs of Theorem~\ref{ThmBndDevIrrat} and Corollary~\ref{CoroBndedDevIrrat}}

\begin{proof}[Proof of Theorem~\ref{ThmBndDevIrrat}]
We argue by contradiction: suppose that there exist $\wt y_0$ belonging to a connected component $\wt {\mathcal X}_{\wt y_0}^-$ of the complement $\wt X_{\varep_0}^-$ and $n_0>0$ such that $\wt f^{n_0}(\wt y_0)$ belongs to a connected component $\wt {\mathcal X}_{\wt y_0}^+$ of the complement $\wt X_{\varep_0}^+$, with $\wt {\mathcal X}_{\wt y_0}^- \subset L(V_{C_0}(\wt\gamma_{\wt z}))$ and $\wt {\mathcal X}_{\wt y_0}^+ \subset R(V_{C_0}(\wt\gamma_{\wt z}))$.

Let us apply Proposition~\ref{PropExistInterTransIrrat}: there exists $\wt y_1\in\wt S$ such that the paths $I^{\Z}_{\wt \F}(\wt y_1) $ and $I^{\Z}_{\wt \F}(\wt z)$ intersect $\wt\F$-transversally. This contradicts Proposition~\ref{LastPropBndDevIrrat}.
\end{proof}

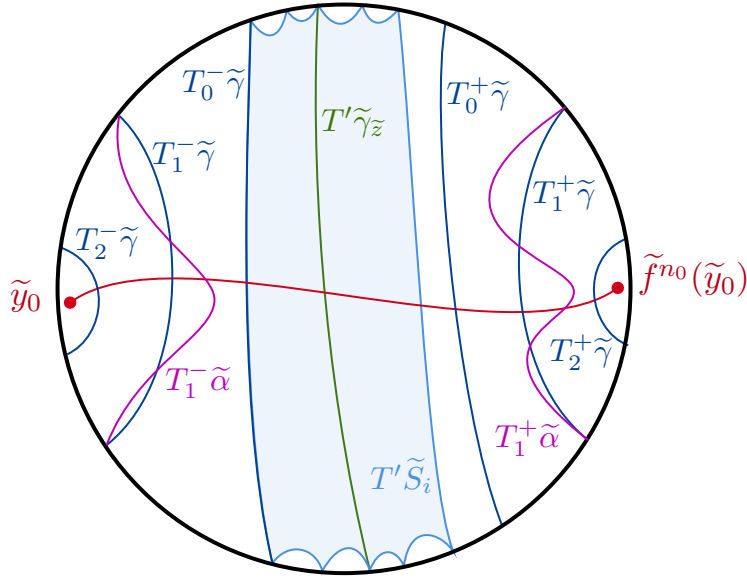
\begin{figure}
\begin{center}

\tikzset{every picture/.style={line width=0.75pt}} 

\begin{tikzpicture}[x=0.75pt,y=0.75pt,yscale=-1.3,xscale=1.3]

\draw [color={rgb, 255:red, 65; green, 117; blue, 5 }  ,draw opacity=1 ]   (320,50) .. controls (317.06,108.27) and (321.06,191.13) .. (339.91,268.85) ;
\draw  [color={rgb, 255:red, 74; green, 144; blue, 226 }  ,draw opacity=1 ][fill={rgb, 255:red, 74; green, 144; blue, 226 }  ,fill opacity=0.1 ] (308,52.29) .. controls (312.17,58.29) and (316.83,57.79) .. (320,50) .. controls (325.5,59.96) and (331.67,59.79) .. (337,50.29) .. controls (340.17,59.46) and (345.15,58.58) .. (350.82,51.91) .. controls (356.53,103.34) and (361.1,233.91) .. (371.96,261.63) .. controls (367.42,254.81) and (355.5,249.63) .. (353,267.13) .. controls (349,259.13) and (343,260.79) .. (339.91,268.85) .. controls (334.68,255.33) and (326.83,259.79) .. (321.67,269.13) .. controls (318.9,259.8) and (308.67,257.63) .. (302.5,266.13) .. controls (288.11,202.09) and (293.33,87.29) .. (294.17,56.29) .. controls (302.67,68.12) and (306.83,58.63) .. (308,52.29) -- cycle ;
\draw [color={rgb, 255:red, 0; green, 66; blue, 146 }  ,draw opacity=1 ]   (294.17,56.29) .. controls (293.68,91.34) and (287.39,200.77) .. (302.5,266.13) ;
\draw [color={rgb, 255:red, 208; green, 2; blue, 27 }  ,draw opacity=1 ]   (224.82,165.34) .. controls (264.82,135.34) and (395.1,189.63) .. (435.1,159.63) ;
\draw [shift={(435.1,159.63)}, rotate = 323.13] [color={rgb, 255:red, 208; green, 2; blue, 27 }  ,draw opacity=1 ][fill={rgb, 255:red, 208; green, 2; blue, 27 }  ,fill opacity=1 ][line width=0.75]      (0, 0) circle [x radius= 2.01, y radius= 2.01]   ;
\draw [shift={(224.82,165.34)}, rotate = 323.13] [color={rgb, 255:red, 208; green, 2; blue, 27 }  ,draw opacity=1 ][fill={rgb, 255:red, 208; green, 2; blue, 27 }  ,fill opacity=1 ][line width=0.75]      (0, 0) circle [x radius= 2.01, y radius= 2.01]   ;
\draw [color={rgb, 255:red, 0; green, 66; blue, 146 }  ,draw opacity=1 ]   (243.68,93.06) .. controls (269.39,117.06) and (273.96,192.2) .. (238.82,220.49) ;
\draw [color={rgb, 255:red, 0; green, 66; blue, 146 }  ,draw opacity=1 ]   (221.68,144.2) .. controls (241.96,152.2) and (238.53,178.49) .. (223.68,185.34) ;
\draw [color={rgb, 255:red, 0; green, 66; blue, 146 }  ,draw opacity=1 ]   (437.96,140.77) .. controls (420.25,148.49) and (423.39,175.63) .. (439.1,181.91) ;
\draw [color={rgb, 255:red, 0; green, 66; blue, 146 }  ,draw opacity=1 ]   (415.1,89.63) .. controls (391.1,117.06) and (388.53,190.2) .. (423.68,218.49) ;
\draw [color={rgb, 255:red, 0; green, 66; blue, 146 }  ,draw opacity=1 ]   (369.02,56.93) .. controls (362.82,91.34) and (372.53,204.77) .. (390.82,251.91) ;
\draw  [line width=1.5]  (220,160) .. controls (220,99.25) and (269.25,50) .. (330,50) .. controls (390.75,50) and (440,99.25) .. (440,160) .. controls (440,220.75) and (390.75,270) .. (330,270) .. controls (269.25,270) and (220,220.75) .. (220,160) -- cycle ;
\draw [color={rgb, 255:red, 188; green, 0; blue, 181 }  ,draw opacity=1 ]   (243.68,93.06) .. controls (237.1,130.77) and (280.25,150.2) .. (280.25,164.2) .. controls (280.25,178.2) and (250.82,189.06) .. (238.82,220.49) ;
\draw [color={rgb, 255:red, 188; green, 0; blue, 181 }  ,draw opacity=1 ]   (415.1,89.63) .. controls (347.39,136.49) and (418.25,148.2) .. (418.25,161.06) .. controls (418.25,173.91) and (374.53,179.63) .. (423.68,218.49) ;

\draw (319.9,95.51) node [anchor=west] [inner sep=0.75pt]  [font=\small,color={rgb, 255:red, 65; green, 117; blue, 5 }  ,opacity=1 ,xscale=1.2,yscale=1.2]  {$T'\wt{\gamma }_{\wt{z}}$};
\draw (352.2,232.22) node  [font=\small,color={rgb, 255:red, 74; green, 144; blue, 226 }  ,opacity=1 ,xscale=1.2,yscale=1.2]  {$T'\wt{S}_{i}$};
\draw (293.49,80) node [anchor=east] [inner sep=0.75pt]  [font=\small,color={rgb, 255:red, 0; green, 66; blue, 146 }  ,opacity=1 ,xscale=1.2,yscale=1.2]  {$T_{0}^{-}\wt{\gamma }$};
\draw (255,108) node [anchor=west] [inner sep=0.75pt]  [font=\small,color={rgb, 255:red, 0; green, 66; blue, 146 }  ,opacity=1 ,xscale=1.2,yscale=1.2]  {$T_{1}^{-}\wt{\gamma }$};
\draw (225.45,149.75) node [anchor=south west] [inner sep=0.75pt]  [font=\small,color={rgb, 255:red, 0; green, 66; blue, 146 }  ,opacity=1 ,xscale=1.2,yscale=1.2]  {$T_{2}^{-}\wt{\gamma }$};
\draw (368,84.99) node [anchor=west] [inner sep=0.75pt]  [font=\small,color={rgb, 255:red, 0; green, 66; blue, 146 }  ,opacity=1 ,xscale=1.2,yscale=1.2]  {$T_{0}^{+}\wt{\gamma }$};
\draw (401,123) node [anchor=west] [inner sep=0.75pt]  [font=\small,color={rgb, 255:red, 0; green, 66; blue, 146 }  ,opacity=1 ,xscale=1.2,yscale=1.2]  {$T_{1}^{+}\wt{\gamma }$};
\draw (434.53,176) node [anchor=north east] [inner sep=0.75pt]  [font=\small,color={rgb, 255:red, 0; green, 66; blue, 146 }  ,opacity=1 ,xscale=1.2,yscale=1.2]  {$T_{2}^{+}\wt{\gamma }$};
\draw (215.62,162.87) node [anchor=east] [inner sep=0.75pt]  [color={rgb, 255:red, 182; green, 2; blue, 24 }  ,opacity=1 ,xscale=1.2,yscale=1.2]  {$\wt{y}_{0}$};
\draw (441.38,154.87) node [anchor=west] [inner sep=0.75pt]  [color={rgb, 255:red, 182; green, 2; blue, 24 }  ,opacity=1 ,xscale=1.2,yscale=1.2]  {$\wt{f}^{n_{0}}(\wt{y}_{0})$};
\draw (414.28,209.35) node [anchor=north east] [inner sep=0.75pt]  [font=\small,color={rgb, 255:red, 188; green, 0; blue, 181 }  ,opacity=1 ,xscale=1.2,yscale=1.2]  {$T_{1}^{+}\wt{\alpha }$};
\draw (260.02,187.07) node [anchor=north west][inner sep=0.75pt]  [font=\small,color={rgb, 255:red, 188; green, 0; blue, 181 }  ,opacity=1 ,xscale=1.2,yscale=1.2]  {$T_{1}^{-}\wt{\alpha }$};

\end{tikzpicture}
\caption{Proof of Corollary~\ref{CoroBndedDevIrrat}.}\label{FigCoroBndedDevIrrat}
\end{center}
\end{figure}

\begin{proof}[Proof of Corollary~\ref{CoroBndedDevIrrat}]
Let $i\in I_{\mathrm{m}}$, $\Lambda_i$ the associated lamination and $S_i$ the associated surface. Let $\gamma$ be a closed geodesic that is a boundary component of this surface $S_i$. 

Pick $z\in S$ a $\mu$-typical point; to this point Theorem~\ref{ThmBndDevIrrat} associates a constant $C_0>0$ and a set $X\subset S$ with ${i}_*(\pi_1(S_i)) \subset {i}_*(\pi_1(X))$. Hence there exists a closed loop $\alpha\subset X$ that is freely homotopic to $\gamma$ (see Figure~\ref{FigCoroBndedDevIrrat}). 

Let $\wt S_i$, $\wt\gamma$ and $\wt z$ be lifts of respectively $S_i$, $\gamma$ and $z$ to $\wt S$ such that $\wt\gamma\cup\wt\gamma_{\wt z} \subset \overline{\wt S_i}$. 
As $\alpha$ is homotopic to $\gamma$, there exists a lift $\wt\alpha$ of $\alpha$ to $\wt S$ such that $d(\wt\alpha,\wt\gamma) = d_0<+\infty$.

As $\gamma$ is simple and closed, there exists $M>0$ such that if $\wt\beta : [0,1]\to\wt S$ crosses $M$ lifts $T_1\wt\gamma,\dots,T_M\wt\gamma$ of $\gamma$, then there exists $1\le i,j,k\le M$ such that $T_j\wt\gamma$ is between $T_1\wt\gamma$ and $T_k\wt\gamma$, and that $d(T_1\wt\gamma, T_j\wt\gamma) > C_0+d_0$ and $d(T_j\wt\gamma, T_k\wt\gamma) > C_0+d_0$. 
\bigskip

Suppose that there exists an orbit $\wt y_0,\dots,\wt f^{n_0}(\wt y_0)$ crossing at least $N=2M+1$ different lifts of $\gamma$. As $X$ has empty interior, perturbing a bit $y_0$ if necessary, one can suppose that $y_0\notin X$ and $f^n(y_0)\notin X$.

Note that as $\gamma$ is the boundary component of $S_i$, either a left neighbourhood of $\wt\gamma$ or a right neighbourhood of $\wt\gamma$ is included in $S_i$. This implies that if $\wt\beta : [0,1]\to \wt S$ is a path linking $\wt y_0$ to $\wt f^{n_0}(\wt y_0)$, then there exists $T'\in\G$ and $t'\in[0,1]$ such that $\wt\beta(t')\in T'\wt\gamma_{\wt z}$, and that both $\wt\beta|_{[0,t']}$ and $\wt\beta|_{[t',1]}$ cross $M$ lifts of $\gamma$.

By the above property, we deduce that $\wt\beta|_{[0,t']}$ crosses three lifts $T_0^-\wt\gamma,T_1^-\wt\gamma$ and $T_2^-\wt\gamma$ of $\gamma$ such that $d(T_0^-\wt\gamma, T_1^-\wt\gamma) > C_0+d_0$ and $d(T_1^-\wt\gamma, T_2^-\wt\gamma) > C_0+d_0$. Note that $\gamma_z$ is disjoint from $\gamma$, so one can suppose that $T_0^-\wt\gamma$ is between $T_1^-\wt\gamma$ and $T'\wt\gamma_{\wt z}$, and that $T_1^-\wt\gamma$ is between $T_0^-\wt\gamma$ and $T_2^-\wt\gamma$. 
This implies that 
\[d(T'\wt\gamma_{\wt z}, T_1^+\wt\alpha) \ge d(T'\wt\gamma_{\wt z}, T_1^+\wt\gamma) - d(T_1^+\wt\gamma, T_1^+\wt\alpha) > C_0+d_0-d_0 = C_0\]
and that $T_1^-\wt\alpha$ separates $\wt\beta(0) = \wt y_0$ from $T'\wt\gamma_{\wt z}$. 

Similarly, there exists $T_1^+\in\G$ such that $d(T'\wt\gamma_{\wt z}, T_1^+\wt\alpha) > C_0$ and that $T_1^+\wt\alpha$ separates $\wt\beta(1) = \wt f^{n_0}(\wt y_0)$ from $T'\wt\gamma_{\wt z}$. Moreover, $T'\wt\gamma_{\wt z}$ separates $T_1^-\wt\alpha$ and $T_1^+\wt\alpha$.

This implies that the connected component of the complement $\wt\X^-$ of $\wt X$ containing $\wt y_0$ (that must be disjoint from $T_1^-\wt\alpha$) is at a distance of at least $C_0$ from $T'\wt\gamma_{\wt z}$, and that the connected component of the complement $\wt\X^+$ of $\wt X$ containing $\wt f^{n_0}(\wt y_0)$ (that must be disjoint from $T_1^+\wt\alpha$) is at a distance of at least $C_0$ from $T'\wt\gamma_{\wt z}$. We moreover have that $T'\wt\gamma_{\wt z}$ separates $\wt\X^-$ from $\wt\X^+$. Exchanging the roles of $y_0$ and $f^{n_0}(y_0)$ and changing $n_0$ with $-n_0$ if necessary, one can suppose that $\wt\X^-\subset L(\wt\gamma_{T'\wt z})$ and that $\wt\X^+\subset R(\wt\gamma_{T'\wt z})$. This allows us to apply Theorem~\ref{ThmBndDevIrrat}, which leads to a contradiction.
\end{proof}

\subsection{Proof of Theorem~\ref{ThmBndDevIrrat2}}\label{SubSecThmBndDevIrrat2}

Let $C_0$ and $\varep_0$ given by Proposition~\ref{PropExistInterTransIrrat}.   Choose $\varep\le \varep_0$.

We suppose that the set $X_{\varep}$ defined before Lemma~\ref{LemAPlusNotEquiv} is a skeleton of the set $A_{\varep}$ (given by Lemma~\ref{LemEssentialPoints}). Let $\wt A_\varep$ be the lift of $A_\varep$ to $\wt S$ containing $\wt z$, and $\wt X_\varep$ the lift of $X_\varep$ to $\wt S$ included in $\wt A_\varep$.

If the skeleton $X_\varep$ of $A_\varep$ is a skeleton of $S$, then $S\setminus X_\varep$ is made of a disc whose lifts to $\wt S$ are bounded, so Theorem~\ref{ThmBndDevIrrat2} is proven. Hence we suppose from now that $X_\varep$ is not a skeleton of $S$, in other words that the complement of $X_\varep$ is made of a set whose lifts to $\wt S$ are unbounded. 

\begin{lemma}\label{LemYInAv}
There exists $C_1\ge 0$ such that the following is true. 
If $\wt y_0\in\wt S$ and $n_0\in\N$ are such that $y_0\in L(V_{C_1}(\wt\gamma_{\wt z}))$ and $\wt f^{n_0}(y_0) \in R(V_{C_1}(\wt\gamma_{\wt z}))$, then $\wt y_0\in \fil(\wt A_{\varep})$.
\end{lemma}

\begin{proof}
As in the proof of Lemma~\ref{LemEssentialPoints}, let $\wt \Gamma_{\wt A_\varep}$ be the union of the geodesics $\wt \gamma_{\wt\beta}$, where $\wt \beta\subset \wt A_\varep$ is a path lifting an essential closed loop; let also $\wt S_{\wt A_\varep} = \conv\big(\wt \Gamma_{\wt A_\varep}\big)$.

We claim that there is a set $\wt Y\subset \wt X_\varep$ separating $\wt S$ and staying at finite distance to $\wt\gamma_{\wt z}$. Indeed, each loop defining $X_\varep$ is at finite distance (in $\wt S$) to a geodesic loop. Let us call $X'$ the union of these geodesic loops; $X'$ forms a skeleton of the geodesic surface $\pr_S(\wt S_{\wt A_\varep})$, with the property that the angle between two of these geodesics at an intersection point is uniformly bounded from below. Let $\wt X'$ be the lift of $X'$ included in $\wt S_{\wt A_\varep}$. We are reduced to show that there is a subset of $\wt X'$ separating $\wt S$ and staying at finite distance to $\wt\gamma_{\wt z}$.
Each point $\wt x\in \wt\gamma_{\wt z}$ belongs to the boundary of a (bounded) disc $\wt D_{\wt x}$ that is a connected component of the complement of $\wt\gamma_{\wt z} \cup \wt X'$. Let $\gamma_0\subset X'$ be a closed geodesic that meets the minimal geodesic lamination $\overline{\gamma_z}$. The return time of $\gamma_z$ on $\gamma_0$ is uniformly bounded (see also Lemma~\ref{LemSyndeticInter}), and the angle of intersection between $\gamma_z$ and $\gamma_0$ is continuous on $\gamma_z \cap \gamma_0$. We deduce that the diameter of the sets
\[\left(\bigcup_{\wt x\in \wt\gamma_{\wt z}(I)} \wt D_{\wt x}\right)_{\substack{I\textrm{ interval}\\ \gamma_{z}(I) \cap \gamma_0 = \emptyset}}\]
is uniformly bounded. It then suffices to consider the set
\[\bigcup_{\wt x\in \wt\gamma_{\wt z}(\R)} \big(\wt D_{\wt x} \cap \wt X'\big);\]
it is included in $\wt X'$, separates $\wt S$ and stays at finite distance to $\wt\gamma_{\wt z}$. This proves the existence of $\wt Y$.
\bigskip

We denote $C_1\ge C_0$ such that $\wt Y \subset V_{C_1}(\wt\gamma_{\wt z})$. 
Suppose by contradiction that $\wt y_0\notin \fil(\wt A_{\varep})$ while $\wt y_0\in L(V_{C_1}(\wt\gamma_{\wt z}))$ and $\wt f^{n_0}(\wt y_0) \in R(V_{C_1}(\wt\gamma_{\wt z}))$.

As $\fil(A_{\varep})$ is an essential filled $f$-invariant open set, the set $\fil(\wt A_{\varep})$ is $\wt f$-invariant, as well as any of its boundary components. So any connected component of the complement of $\fil(\wt A_{\varep})$ is $\wt f$-invariant.
Hence the points $\wt y_0$ and $\wt f^{n_0}(\wt y_0)$ belong to the same connected component of the complement of $\wt X'$, in particular to the same connected component of the complement of ${\wt Y}$. 
This contradicts the hypothesis that $\wt y_0\in L(V_{C_1}(\wt\gamma_{\wt z}))$ and $\wt f^{n_0}(\wt y_0) \in R(V_{C_1}(\wt\gamma_{\wt z}))$.
\end{proof}

The {projection of the set} ${\Fix(\wt f)}$ being inessential, there exist a finite number of pairwise disjoint bounded discs $D_1,\dots, D_k$ containing {this projection} with the property that for $i\neq j$ one has $f(D_i)\cap D_j = \emptyset$, {and that if $\wt D_i$ is a lift of $D_i$, then $\wt f(\wt D_i)$ is disjoint from $R\wt D_i$ for any $R\in\G\setminus\{\Id\}$.} 

Let $(U_j)_{1\le j\le p}$ be a finite set of trivializing charts for the foliation $\F$ on the complement of the discs $D_i$, such that for any $j$, for any lift $\wh U_j$ of $U_j$ to $\wh\dom(I)$ and any $\wh y\in \wh U_j$, the trajectory $I^{[-1,1]}_{\wh \F}(\wh y)$ meets all the leaves of $\wh\F$ meeting $\wh U_j$. 

We also require the following property, that is automatically satisfied if the charts $U_j$ are small enough: if $U_j$ intersects $X_\varep$, then $U_j\subset V_\varep$, where $V_\varep$ is the neighbourhood of $X_\varep$ given by Remark~\ref{RemLemAPlusNotEquiv}.

\begin{lemma}\label{LemMeetPhiTwice}
There exists $C_2>0$ such that for any $\wt y_0\in\wt S$ and any $n_0\in\N$ such that $d\big(\wt y_0, \wt f^{n_0}(\wt y_0)\big)\ge C_2$, there exists $1\le j\le p$ and $T\in\G\setminus\{\Id\}$ such that the trajectory $\wt y_0, \dots ,\wt f^{n_0}(\wt y_0)$ meets one lift $\wt U_j$ of $U_j$ and then $T\wt U_{j}$.
\end{lemma}

\begin{proof}
We choose lifts $(\wt U_j)_{1\le j\le p}$ of $(U_j)_{1\le j\le p}$ to $\wt S$.
Let $d_0 = \max_{1\le i\le k}\diam(D_i) + \max_{1\le j\le p}\diam(U_j) + d(\wt f,\Id)$. Let $C_2 = (2p+2)d_0$. Consider $2p+2$ geodesics $\wt\gamma_1,\dots,\wt\gamma_{2p+2}$ of $\wt S$, each of them separating $\wt y_0$ from $\wt f^{n_0}(\wt y_0)$ and pairwise at a distance $\ge d_0$. 
These geodesics define strips $\mathrm{St}_j = \big(R(\wt\gamma_j)\cap L(\wt\gamma_{j+1})\big)_{1\le j\le 2p+1}$ that we can suppose nonempty up to permuting the geodesics $\gamma_\ell$. 
Each strip $\mathrm{St}_\ell$ meets the orbit $\wt y_0, \dots, \wt f^{n_0}(\wt y_0)$ at at least one point $\wt y'_\ell$ whose projection on $S$ satisfies $y'_\ell \notin \bigcup_{1\le i\le k} D_i$. Hence each point $\wt y'_\ell$ belongs to one chart $T_\ell \wt U_{j_\ell}$ for some $T_\ell\in \G$. 
By the pigeonhole principle there exist $\ell_0 \neq \ell_1$ that are even and such that $j_{\ell_0} = j_{\ell_1}$. 
By hypothesis on the size of the strips one has $T_{\ell_0} \neq T_{\ell_1}$. It then suffices to choose for $\wt \phi$ any leaf meeting $U_{j_{\ell_0}}$ and $T = T_{\ell_1}T_{\ell_0}^{-1}$. 
\end{proof}

\begin{lemma}\label{LemExistC4}
For any $C_0>0$, there exists $C_3\ge C_0$ such that if $\wt y\in L(V_{{C_3}}(\wt\gamma_{\wt z}))$ and if $\wt y$ and $\wt f(\wt y)$ belong to the closure of different connected components $\wt\X_1$ and $\wt \X_2$ of the complement of $\wt X_\varep$, then one of $\wt\X_1$ and $\wt \X_2$ is disjoint from $V_{C_0}(\wt\gamma_{\wt z})$.
\end{lemma}

\begin{proof}
Denote $\wt\Pi_i$ the union of the lifts of boundary components of $S_i$; it is a union of lifts of closed geodesics.
In $\wt S$, the connected components of the complement of $\wt X_\varep$ are in bijection with geodesics of $\wt \Pi_i$. Each lift of such geodesic is the boundary component of two connected components of the complement of $\wt \Pi_i$, one that is a lift of $S_i$ and one that is at finite Hausdorff distance to a (unique) connected component of the complement of $\wt X_\varep$. This Hausdorff distance is uniformly bounded by $d_1>0$.

We parametrise geodesics by arclength. Using the crown components of $\Lambda_i$ (see \cite[Lemma 4.4]{casson}), we get that there exists $p_0>0$ such that for any geodesic $\wt\gamma$ of $\wt\Pi_i$ and any geodesic $\wt\gamma'$ of $\wt\Lambda_i$, one has $\operatorname{length}(\pr_{\wt\gamma}(\wt\gamma')) \le p_0$. Note that it implies that for any geodesic $\wt\gamma$ of $\wt\Pi_i$ and any geodesic $\wt\gamma'$ of $\wt\Pi_i$ different from $\wt\gamma$, one has $\operatorname{length}(\pr_{\wt\gamma}(\wt\gamma')) \le p_0$.
Set 
\[C_3 = 2p_0 + 2C_0 + 3d_1 + d(\wt f,\Id).\]

Suppose that $\wt y\in L(V_{{C_3}}(\wt\gamma_{\wt z}))$ is such that $\wt y$ and $\wt f(\wt y)$ belong to the closure of different connected components $\wt\X_1$ and $\wt \X_2$ of the complement of $\wt X_\varep$. If $\wt\X_1$ is disjoint from $V_{C_0}(\wt\gamma_{\wt z})$, the lemma is proved. If not, let $\wt\gamma_1$ be the geodesic of $\wt\Pi_i$ associated to $\wt\X_1$ and pick $t_0$ such that $d(\wt\gamma_1, \wt\gamma_{\wt z}) = d(\wt\gamma_1(t_0), \wt\gamma_{\wt z}) \le C_0+d_1$. As $\wt f(\wt y)\in\wt \X_2$, the point $\wt y$ is at a distance of at most $d(\wt f, \Id)$ to $\partial \wt\X_1$, hence at a distance of at most $d(\wt f, \Id) + d_1$ to $\wt\gamma_1$. 
Hence $d(\wt y, \pr_{\wt\gamma_1}(\wt y)) \le d(\wt f, \Id) + d_1$, so
\begin{align*}
C_3 & \le d(\wt y,\wt\gamma_{\wt z}) \\
&\le d(\wt y, \pr_{\wt\gamma_1}(\wt y)) + d(\pr_{\wt\gamma_1}(\wt y), \wt\gamma_1(t_0)) + d(\wt\gamma_1(t_0), \wt\gamma_{\wt z})\\
&\le d(\wt f, \Id) + d_1 + d(\pr_{\wt\gamma_1}(\wt y), \wt\gamma_1(t_0)) + C_0+d_1,
\end{align*}
hence
$d(\pr_{\wt\gamma_1}(\wt y), \wt\gamma_1(t_0)) \ge 2 p_0 + C_0 + d_1$. 
Using $\wt\gamma_2$ the geodesic of $\wt\Pi_i$ associated to $\wt\X_2$, this implies that 
\[|\pr_{\wt\gamma_1} (\wt \X_2) - t_0 | \ge  p_0 + C_0,\]
and so
\[\inf\{|x-y|\mid x\in \pr_{\wt\gamma_1}(\wt \X_2),\, y \in \pr_{\wt\gamma_1}(\wt\gamma_{\wt z})\} \ge C_0.\]
As the projection $\pr_{\wt\gamma_1}$ is 1-Lipschitz, we deduce the lemma.
\end{proof}

\begin{lemma}\label{LemDrawCrosses}
Suppose that there is an essential (in $S$) transverse loop $\alpha$ and a lift $\wh\alpha$ of $\alpha$ to $\wh\dom(I)$ whose associated band $\wh B'$ is drawn by the transverse trajectory $I^\Z_{\wh \F}(\wh z)$. Then the geodesic axis of the projection $\wt \alpha$ of $\wh\alpha$ on $\wt S$ crosses $\wt \gamma_{\wt z}$.
\end{lemma}

\begin{proof}
By \cite[Proposition~2.5]{paper1PAF} the transverse trajectory $I^\Z_{\wh \F}(\wh z)$ cannot accumulate in $\wh \alpha$. 
By \cite[Proposition~2.18]{paper1PAF} the transverse trajectory $I^\Z_{\wh \F}(\wh z)$ cannot be equivalent to $\wh \alpha$ at $+\infty$ or at $-\infty$. 
Hence the transverse trajectory $I^\Z_{\wh \F}(\wh z)$ goes out of $\wh B'$ both in positive and negative times. 

Suppose that the transverse trajectory $I^\Z_{\wh \F}(\wh z)$ draws the band $\wh B'$. By what we have said, either it draws and visits $\wh B$ or it draws and crosses $\wh B$.
As it has no transverse self-intersection (Proposition~\ref{PropPasInterTrans}), \cite[Proposition~2.15]{paper1PAF} implies the first case is impossible. 

Hence $I^\Z_{\wh \F}(\wh z)$ draws and crosses the band $\wh B'$, for example from left to right. Suppose by contradiction that the geodesic axis of the projection $\wt \alpha$ of $\wh\alpha$ on $\wt S$ {does not cross} $\wt \gamma_{\wt z}$. Then for crossing number reasons the trajectory $I^\Z_{\wh\F}(\wh z)$ also has to cross the band $\wt B'$ from right to left. \cite[Proposition~2.16]{paper1PAF} states that it is impossible, a contradiction.
\end{proof}

\begin{lemma}\label{LemIfDevThenBothSides}
Suppose that there exists $n\in\Z$ and $C>0$ such that $\wt f^n(\wt z)\in L(V_C(\wt\gamma_{\wt z}))$. Then there exist $T\in\G$ and $m_1, m_2\in\Z$ such that $\wt f^{m_1}(\wt z)\in L(V_C(T\wt\gamma_{\wt z}))$ and $\wt f^{m_2}(\wt z)\in R(V_C(T\wt\gamma_{\wt z}))$.
\end{lemma}

\begin{proof}
The idea of this proof is depicted in Figure~\ref{FigLemIfDevThenBothSides}.

By hypothesis, we have that
\[d\big(\wt f^n(\wt z), \wt\gamma_{\wt z}\big) \ge C+4\delta\]
for some $\delta>0$.
Following \cite[Section~4]{casson}, there are two possibilities. 

Either $\wt\gamma_{\wt z}$ is accumulated by geodesics of $\wt\Lambda_i$ that separate $\wt\gamma_{\wt z}$ from $\wt f^n(\wt z)$. In this case, there exists a geodesic $\wt \gamma'\subset \wt\Lambda_i$ such that 
\[d\big(\wt f^n(\wt z), \wt\gamma'\big) \ge C+3\delta,\]
but with the additional property that for $m$ large enough we also have 
\[d\big(\wt f^m(\wt z), \wt\gamma'\big) \ge C+2\delta,\]
with $\wt\gamma'$ separating $\wt f^n(\wt z)$ from $\wt f^m(\wt z)$. As $\gamma_z$ is dense in $\Lambda_i$, we deduce that there is $T\in\G$ such that $V_C(T\wt\gamma_{\wt z})$ separates $\wt f^n(\wt z)$ from $\wt f^m(\wt z)$.

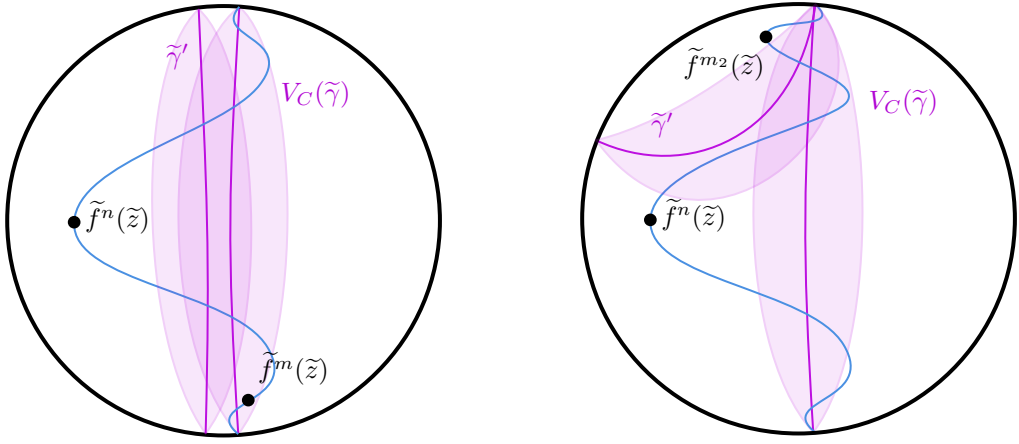
\begin{figure}
\begin{center}

\tikzset{every picture/.style={line width=0.75pt}} 

\begin{tikzpicture}[x=0.75pt,y=0.75pt,yscale=-1.2,xscale=1.2]


\draw [color={rgb, 255:red, 189; green, 16; blue, 224 }  ,draw opacity=1 ]   (436.24,79.35) .. controls (428.09,135.6) and (382.53,154.05) .. (344.97,136.27) ;
\draw  [color={rgb, 255:red, 189; green, 16; blue, 224 }  ,draw opacity=0.2 ][fill={rgb, 255:red, 189; green, 16; blue, 224 }  ,fill opacity=0.1 ] (196.56,80.33) .. controls (221.44,102.33) and (225.22,221.22) .. (196.11,259.44) .. controls (161.44,221.22) and (164.33,104.78) .. (196.56,80.33) -- cycle ;
\draw  [color={rgb, 255:red, 189; green, 16; blue, 224 }  ,draw opacity=0.2 ][fill={rgb, 255:red, 189; green, 16; blue, 224 }  ,fill opacity=0.1 ] (435.72,79.5) .. controls (460.61,101.5) and (464.39,220.39) .. (435.28,258.61) .. controls (400.61,220.39) and (403.5,103.94) .. (435.72,79.5) -- cycle ;
\draw [color={rgb, 255:red, 189; green, 16; blue, 224 }  ,draw opacity=1 ]   (196.56,80.33) .. controls (192.11,156.33) and (191.22,179.89) .. (196.11,259.44) ;
\draw  [line width=1.5]  (100,170) .. controls (100,120.29) and (140.29,80) .. (190,80) .. controls (239.71,80) and (280,120.29) .. (280,170) .. controls (280,219.71) and (239.71,260) .. (190,260) .. controls (140.29,260) and (100,219.71) .. (100,170) -- cycle ;
\draw [color={rgb, 255:red, 189; green, 16; blue, 224 }  ,draw opacity=1 ]   (435.72,79.5) .. controls (431.28,155.5) and (430.39,179.06) .. (435.28,258.61) ;
\draw [color={rgb, 255:red, 189; green, 16; blue, 224 }  ,draw opacity=1 ]   (179.66,81.38) .. controls (182.99,158.94) and (184.54,180.71) .. (182.54,259.16) ;
\draw [color={rgb, 255:red, 74; green, 144; blue, 226 }  ,draw opacity=1 ]   (436.24,79.35) .. controls (448.75,91.16) and (415.42,83.6) .. (415.64,92.94) .. controls (415.86,102.27) and (456.81,110.26) .. (449.2,120.27) .. controls (441.58,130.29) and (368.73,144.03) .. (367.53,169.63) .. controls (366.33,195.23) and (433.78,198.27) .. (447.93,222.03) .. controls (462.09,245.79) and (419.53,244.83) .. (435.8,258.46) ;
\draw  [draw opacity=0][fill={rgb, 255:red, 0; green, 0; blue, 0 }  ,fill opacity=1 ] (364.89,169.63) .. controls (364.89,168.17) and (366.07,166.99) .. (367.53,166.99) .. controls (368.99,166.99) and (370.17,168.17) .. (370.17,169.63) .. controls (370.17,171.09) and (368.99,172.27) .. (367.53,172.27) .. controls (366.07,172.27) and (364.89,171.09) .. (364.89,169.63) -- cycle ;
\draw  [draw opacity=0][fill={rgb, 255:red, 0; green, 0; blue, 0 }  ,fill opacity=1 ] (413,92.94) .. controls (413,91.48) and (414.18,90.3) .. (415.64,90.3) .. controls (417.1,90.3) and (418.28,91.48) .. (418.28,92.94) .. controls (418.28,94.4) and (417.1,95.58) .. (415.64,95.58) .. controls (414.18,95.58) and (413,94.4) .. (413,92.94) -- cycle ;
\draw [color={rgb, 255:red, 74; green, 144; blue, 226 }  ,draw opacity=1 ]   (196.56,80.33) .. controls (185.05,93.41) and (220.25,93.01) .. (205.05,113.01) .. controls (189.85,133.01) and (129.05,145.01) .. (127.85,170.61) .. controls (126.65,196.21) and (194.09,199.25) .. (208.25,223.01) .. controls (222.4,246.77) and (179.85,245.81) .. (196.11,259.44) ;
\draw  [draw opacity=0][fill={rgb, 255:red, 0; green, 0; blue, 0 }  ,fill opacity=1 ] (125.21,170.61) .. controls (125.21,169.16) and (126.39,167.97) .. (127.85,167.97) .. controls (129.3,167.97) and (130.49,169.16) .. (130.49,170.61) .. controls (130.49,172.07) and (129.3,173.25) .. (127.85,173.25) .. controls (126.39,173.25) and (125.21,172.07) .. (125.21,170.61) -- cycle ;
\draw  [line width=1.5]  (339.17,169.17) .. controls (339.17,119.46) and (379.46,79.17) .. (429.17,79.17) .. controls (478.87,79.17) and (519.17,119.46) .. (519.17,169.17) .. controls (519.17,218.87) and (478.87,259.17) .. (429.17,259.17) .. controls (379.46,259.17) and (339.17,218.87) .. (339.17,169.17) -- cycle ;
\draw  [color={rgb, 255:red, 189; green, 16; blue, 224 }  ,draw opacity=0.2 ][fill={rgb, 255:red, 189; green, 16; blue, 224 }  ,fill opacity=0.1 ] (436.24,79.35) .. controls (473.56,118.93) and (393.56,203.16) .. (344.97,136.27) .. controls (378.45,127.83) and (411.34,102.49) .. (436.24,79.35) -- cycle ;
\draw  [color={rgb, 255:red, 189; green, 16; blue, 224 }  ,draw opacity=0.2 ][fill={rgb, 255:red, 189; green, 16; blue, 224 }  ,fill opacity=0.1 ] (179.66,81.38) .. controls (204.54,103.38) and (211.66,220.93) .. (182.54,259.16) .. controls (153.98,219.62) and (152.78,105.62) .. (179.66,81.38) -- cycle ;
\draw  [draw opacity=0][fill={rgb, 255:red, 0; green, 0; blue, 0 }  ,fill opacity=1 ] (197.63,245.19) .. controls (197.63,243.73) and (198.81,242.55) .. (200.27,242.55) .. controls (201.73,242.55) and (202.91,243.73) .. (202.91,245.19) .. controls (202.91,246.65) and (201.73,247.83) .. (200.27,247.83) .. controls (198.81,247.83) and (197.63,246.65) .. (197.63,245.19) -- cycle ;

\draw (212.58,117.21) node [anchor=west] [inner sep=0.75pt]  [font=\small,color={rgb, 255:red, 177; green, 5; blue, 212 }  ,opacity=1 ]  {$V_{C}(\wt \gamma )$};
\draw (457.12,120.21) node [anchor=west] [inner sep=0.75pt]  [font=\small,color={rgb, 255:red, 177; green, 5; blue, 212 }  ,opacity=1 ]  {$V_{C}(\wt \gamma )$};
\draw (177.03,101.21) node [anchor=east] [inner sep=0.75pt]  [font=\small,color={rgb, 255:red, 177; green, 5; blue, 212 }  ,opacity=1 ]  {$\wt\gamma '$};
\draw (131.19,168.59) node [anchor=west] [inner sep=0.75pt]  [font=\small]  {$\wt f^{n}(\wt z)$};
\draw (370.87,167.6) node [anchor=west] [inner sep=0.75pt]  [font=\small]  {$\wt f^{n}(\wt z)$};
\draw (372.65,136.41) node [anchor=south] [inner sep=0.75pt]  [font=\small,color={rgb, 255:red, 177; green, 5; blue, 212 }  ,opacity=1 ]  {$\wt\gamma '$};
\draw (416.28,96.34) node [anchor=north east] [inner sep=0.75pt]  [font=\small]  {$\wt f^{m_{2}}(\wt z)$};
\draw (203.19,240) node [anchor=south west] [inner sep=0.75pt]  [font=\small]  {$\wt f^{m}(\wt z)$};

\end{tikzpicture}

\caption{Proof of Lemma~\ref{LemIfDevThenBothSides}: first case (left) and second case (right).}\label{FigLemIfDevThenBothSides}
\end{center}
\end{figure}

Or $\wt\gamma_{\wt z}$ is not accumulated by geodesics of $\wt\Lambda_i$ that separate $\wt\gamma_{\wt z}$ from $\wt f^n(\wt z)$. By \cite[Section~4]{casson}, there exists a geodesic $\wt\gamma'\subset\wt\Lambda_i$ having an endpoint in common with $\wt\gamma_{\wt z}$ and located on the same side of $\wt\gamma_{\wt z}$ as $\wt f^n(\wt z)$. We treat the case where this common endpoint is $\omega(\wt\gamma_{\wt z})$, the other case being similar.

By Lusin theorem, there exists a set $A$ of positive $\mu$-measure such that the restriction to $A$ of the map $x\mapsto\wt\gamma_{\wt x}$ is continuous. 
The point $z$ being typical, there exists $k\in\Z$ such that $f^k(z)\in A$. So $f^n(z)\in f^{n-k}(A)$.
The point $z$ being recurrent, and by the continuity property of the map $\wt z\mapsto \wt\gamma_{\wt z}$ in restriction to $f^{n-k}(A)$, we deduce that there exists $m_2\in\N$ large such that
\[d\big(\wt f^{m_2}(\wt z), \wt\gamma_{\wt z}\big) \ge C+3\delta\]
and $\wt f^{m_2}(\wt z)$ is close to $\omega(\wt\gamma_{\wt z})$. This implies that for such $m_2$ large enough, we have 
\[d\big(\wt f^{m_2}(\wt z), \wt\gamma'\big) \ge C+2\delta,\]
with $\wt\gamma'$ separating $\wt f^{m}(\wt z)$ from $\wt\gamma_{\wt z}$. 
Moreover, for $m$ large enough we have 
\[d\big(\wt f^{-m}(\wt z), \wt\gamma'\big) \ge C+\delta,\]
with $\wt f^{-m}(\wt z)$ and $\wt\gamma_{\wt z}$ being on the same side of $\wt\gamma'$.
As $\gamma_z$ is dense in $\Lambda_i$, we deduce that there is $T\in\G$ such that, setting $m_1 = -m$, we have $V_C(T\wt\gamma_{\wt z})$ separates $\wt f^{m_1}(\wt z)$ from $\wt f^{m_2}(\wt z)$.
\end{proof}



Let $C_0$ and $\varep_0$ given by Proposition~\ref{PropExistInterTransIrrat}, $\varep<\varep_0$, $C_1$ given by Lemma~\ref{LemYInAv}, $C_2$ given by Lemma~\ref{LemMeetPhiTwice} and $C_3$ given by Lemma~\ref{LemExistC4} applied to $C_0$.
Let us choose an isotopy $I$ between $\Id_S$ and $f$ that is given by Theorem~\ref{ThExistIstop}, that lifts to an isotopy $\wt I$ between $\Id_{\wt S}$ and $\wt f$. {We assume that $\Fix(I)$ is inessential, as it is contained in the projection of $\Fix(\wt f)$, the set of contractible fixed points of $f$. Let $M$ be given by \cite[Lemma~2.9]{paper1PAF}, and we assume that every transverse trajectory of a point in $\dom(I)$ is chosen such that its lift to $\wt{S}$ has diameter at most $M$.}

Let $C_4 = \max(C_1, C_3+C_2+2d(\wt I, \Id_{\wt S}))$ and {$C_5=C_4+2M$}.

\begin{lemma}\label{CoroBoundedTyp}
For $\mu$-a.e.~$z\in S$, if $\wt f^n(\wt z)\in L(V_{C_4}(\wt\gamma_{\wt z}))$ then there exists $t_1\in{[0, n]}$ such that $I^{t_1}_{\wt\F}(\wt z) \in L(\wt B)$. 
\end{lemma}

\begin{lemma}\label{LemThmBndDevIrrat2}
Let $\wt y_0\in \fil(\wt A_\varep)$ and $n_1\in\N$ be such that $\wt y_0\in L(V_{{C_5}}(\wt\gamma_{\wt z}))$ and $\wt f^{n_1}(\wt y_0) \notin L(V_{{C_5-C_2-d(\wt I, \Id_S)}}(\wt\gamma_{\wt z}))$. 
Then there exists $t_1\in [0, n_1]$ such that $I^{t_1}_{\wt\F}(\wt y_0) \in L(\wt B)$. 
\end{lemma}

Let us first show how these lemmas imply Theorem~\ref{ThmBndDevIrrat2}. In fact, we will use the last property of Theorem~\ref{ThmBndDevIrrat2} in the proof of Lemma~\ref{LemThmBndDevIrrat2}: in logical order we prove Lemma~\ref{CoroBoundedTyp}, then the last conclusion of Theorem~\ref{ThmBndDevIrrat2} {using $\overline{C}=C_4$}, then Lemma~\ref{LemThmBndDevIrrat2} and finally the first conclusion of Theorem~\ref{ThmBndDevIrrat2} {using $\overline{C}=C_5$}. We adopt this order of presentation because Lemmas~\ref{CoroBoundedTyp} and \ref{LemThmBndDevIrrat2} share the same beginning of proof, as well as both conclusions of Theorem~\ref{ThmBndDevIrrat2}.

\begin{proof}[Proof of Theorem~\ref{ThmBndDevIrrat2}]
Thanks to Lemma~\ref{LemIfDevThenBothSides}, the second conclusion of Theorem~\ref{ThmBndDevIrrat2} follows from the first one applied to the point $z$. Hence, we will only prove this first conclusion, using {$\overline{C}=C_4$ and} Lemma~\ref{CoroBoundedTyp} for the case of ${\wt z}$ and {using $\overline{C}=C_5$ and} Lemma~\ref{LemThmBndDevIrrat2} for the general case.

Let $\wt y_0\in\wt S$ be such that $\wt y_0\in L(V_{{\overline{C}}}(\wt\gamma_{\wt z}))$ and $\wt f^{n_0}(\wt y_0)\in R(V_{{\overline{C}}}(\wt\gamma_{\wt z}))$.

Let $n_1\in [0,n_0)$ be the first $n\ge 0$ such that $\wt f^{n+1}(\wt y_0)\notin L(V_{{\overline{C}-C_2-d(\wt I, \Id)}}(\wt\gamma_{\wt z}))$. Then $d(\wt y_0, \wt f^{n_1}(\wt y_0))\ge C_2$. 
Similarly, let $n_2\in (n_1,n_0)$ be the last $n\le n_0$ such that $\wt f^{n-1}(\wt y_0)\notin R(V_{{\overline{C}-C_2-d(\wt I, \Id)}}(\wt\gamma_{\wt z}))$. Then $d(f^{n_2}(\wt y_0), \wt f^{n_0}(\wt y_0))\ge C_2$.

We have four cases:
\begin{enumerate}
\item\label{Ca1} There exist $\wt \X_1$ and $\wt \X_2$ two closures of connected components of the complement of $\wt X_\varep$, and $0\le n'_1\le n_1$ and $n_2\le n'_2\le n_0$ such that $\wt f^{n'_1}(\wt y_0)\in \wt\X_1$ and $\wt f^{n'_2}(\wt y_0)\in \wt\X_2$, and such that $\wt \X_1 \subset L(V_{C_0}(\wt\gamma_{\wt z}))$ and $\wt \X_2\subset R(V_{C_0}(\wt\gamma_{\wt z}))$.
\item\label{Ca2} There exists $\wt \X_1$ a closure of a connected components of the complement of $\wt X_\varep$, and $0\le n'_1\le n_1$ such that $\wt f^{n'_1}(\wt y_0)\in \wt\X_1$ and $\wt \X_1 \subset L(V_{C_0}(\wt\gamma_{\wt z}))$. Moreover, for any closure $\wt \X_2$ of a connected component of the complement of $\wt X_\varep$ meeting the segment of orbit $\wt f^{n_2}(\wt y_0),\dots, \wt f^{n_0}(\wt y_0)$, one has $\wt \X_2 \cap V_{C_0}(\wt \gamma_{\wt z})\neq\emptyset$.
\item\label{Ca3} There exists $\wt \X_2$ a closure of a connected components of the complement of $\wt X_\varep$, and $n_2\le n'_2\le n_0$ such that $\wt f^{n'_2}(\wt y_0)\in \wt\X_2$ and $\wt \X_2 \subset R(V_{C_0}(\wt\gamma_{\wt z}))$. Moreover, for any closure $\wt \X_1$ of a connected component of the complement of $\wt X_\varep$ meeting the segment of orbit $\wt y_0,\dots, \wt f^{n_1}(\wt y_0)$, one has $\wt \X_1 \cap V_{C_0}(\wt \gamma_{\wt z})\neq\emptyset$.
\item\label{Ca4} For any closures $\wt \X_1, \wt \X_2$ of connected components of the complement of $\wt X_\varep$, the first one meeting the segment of orbit $\wt y_0,\dots, \wt f^{n_1}(\wt y_0)$ and the second one meeting the segment of orbit $\wt f^{n_2}(\wt y_0),\dots, \wt f^{n_0}(\wt y_0)$, one has $\wt \X_1 \cap V_{C_0}(\wt \gamma_{\wt z})\neq\emptyset$ and $\wt \X_2 \cap V_{C_0}(\wt \gamma_{\wt z})\neq\emptyset$.
\end{enumerate}

\paragraph{Case~\ref{Ca1}.} Under these hypotheses, Theorem~\ref{ThmBndDevIrrat2} directly follows from  Theorem~\ref{ThmBndDevIrrat}.

\paragraph{Case~\ref{Ca4}.} 
By Lemma~\ref{LemYInAv}, we have $\wt y_0\in \fil(\wt A_\varep)$.

Applying Lemma~\ref{CoroBoundedTyp} (for $y_0 = z$) or Lemma~\ref{LemThmBndDevIrrat2} (general case), we get that there exists $t_1\in [0, n_1]$ such that $I^{t_1}_{\wt\F}(\wt y_0) \in L(\wt B)$. 
Similarly $I^{t_2}_{\wt \F}(\wt y_0)\in R(\wt B)$. 
Hence there exist $t_1<t<t_2$ and $t'\in\R$ such that $I^{t}_{\wt \F}(\wt y_0) = I^{s}_{\wt \F}(\wt z)$, and that choosing lifts $\wh y_0$ and $\wh z$ of $y_0$ and $z$ to $\wh\dom(I)$ such that $I^{t}_{\wh \F}(\wh y_0) = I^{s}_{\wh \F}(\wh z)$, we have $I^{t_1}_{\wh \F}(\wh y_0)\in L(\wh B)$ and $I^{t_2}_{\wh \F}(\wh y_0)\in R(\wh B)$ (where $\wh B$ is the set of leaves of $\wh \F$ met by $I^{\Z}_{\wh \F}(\wh z)$). So $\wh y_0\in L(\wh B)$ and $\wh f^{n_0}(\wh y_0)\in R(\wh B)$

Moreover, by \cite[Proposition~2.5]{paper1PAF}, the transverse trajectory $I^\Z_{\wh \F}(\wh z)$ cannot accumulate in any transverse trajectory; this implies (similarly to \cite[Lemma~2.17]{paper1PAF}) that $I^{\Z}_{\wh \F}(\wh z)$ and $I^{[0, n_0]}_{\wh \F}(\wh y_0)$ intersect $\F$-transversally. This contradicts Proposition~\ref{LastPropBndDevIrrat}.

\begin{figure}
\begin{center}

\tikzset{every picture/.style={line width=0.75pt}} 

\begin{tikzpicture}[x=0.75pt,y=0.75pt,yscale=-1.2,xscale=1.2]

\draw  [color={rgb, 255:red, 189; green, 16; blue, 224 }  ,draw opacity=0.2 ][fill={rgb, 255:red, 189; green, 16; blue, 224 }  ,fill opacity=0.1 ] (224.72,45.5) .. controls (249.61,67.5) and (253.39,186.39) .. (224.28,224.61) .. controls (189.61,186.39) and (192.5,69.94) .. (224.72,45.5) -- cycle ;
\draw  [draw opacity=0][fill={rgb, 255:red, 176; green, 248; blue, 28 }  ,fill opacity=0.5 ] (134.21,167.59) .. controls (128.67,154.21) and (127.84,127.66) .. (130.51,118.04) .. controls (134.58,119.29) and (135.58,120.21) .. (137.83,118.04) .. controls (140.33,117.71) and (141.25,116.71) .. (142.25,115.13) .. controls (146.17,120.54) and (160.42,114.96) .. (163.42,115.54) .. controls (158.08,127.96) and (171.67,130.88) .. (175.75,139.63) .. controls (163.17,140.88) and (157.33,150.21) .. (155.33,161.04) .. controls (148.33,159.38) and (144.08,164.29) .. (143.33,167.88) .. controls (139.58,164.88) and (137.25,166.38) .. (134.21,167.59) -- cycle ;
\draw  [draw opacity=0][fill={rgb, 255:red, 176; green, 248; blue, 28 }  ,fill opacity=0.5 ] (134.21,167.59) .. controls (128.67,154.21) and (127.84,127.66) .. (130.51,118.04) .. controls (134.58,119.29) and (137.13,124.97) .. (139.87,120.43) .. controls (142.97,123.74) and (149.44,126.51) .. (156.82,123.44) .. controls (160.74,128.85) and (205.13,128.05) .. (222.21,114.51) .. controls (235.44,130.97) and (286.43,138.07) .. (290.51,146.82) .. controls (274.51,149.74) and (219.28,135.37) .. (217.28,146.21) .. controls (147.13,115.74) and (151.52,161.93) .. (150.77,165.52) .. controls (143.44,157.59) and (137.25,166.38) .. (134.21,167.59) -- cycle ;
\draw [color={rgb, 255:red, 189; green, 16; blue, 224 }  ,draw opacity=1 ]   (224.72,45.5) .. controls (220.28,121.5) and (219.39,145.06) .. (224.28,224.61) ;
\draw [color={rgb, 255:red, 155; green, 155; blue, 155 }  ,draw opacity=1 ]   (177.08,143.46) .. controls (174.51,128.36) and (155.9,131.59) .. (165.13,111.28) ;
\draw [color={rgb, 255:red, 155; green, 155; blue, 155 }  ,draw opacity=1 ]   (154.6,165.6) .. controls (156.21,150.97) and (166.05,136.97) .. (180.36,140.21) ;
\draw [color={rgb, 255:red, 155; green, 155; blue, 155 }  ,draw opacity=1 ]   (143.91,172.96) .. controls (141.02,162.27) and (153.59,157.9) .. (159.44,163.28) ;
\draw [color={rgb, 255:red, 155; green, 155; blue, 155 }  ,draw opacity=1 ]   (146.05,172.67) .. controls (144.51,166.97) and (140.21,165.28) .. (134.21,167.59) ;
\draw [color={rgb, 255:red, 155; green, 155; blue, 155 }  ,draw opacity=1 ]   (140.97,112.97) .. controls (146.58,122.81) and (156.36,112.67) .. (170.67,115.9) ;
\draw [color={rgb, 255:red, 155; green, 155; blue, 155 }  ,draw opacity=1 ]   (134.51,114.36) .. controls (135.44,120.51) and (142.67,117.59) .. (144.21,110.97) ;
\draw [color={rgb, 255:red, 155; green, 155; blue, 155 }  ,draw opacity=1 ]   (129.59,117.59) .. controls (134.05,119.28) and (137.74,120.82) .. (139.28,114.21) ;
\draw  [draw opacity=0][fill={rgb, 255:red, 248; green, 231; blue, 28 }  ,fill opacity=0.5 ] (299.84,97.26) .. controls (306.45,110.15) and (309.43,136.54) .. (307.55,146.35) .. controls (303.39,145.43) and (302.32,144.6) .. (300.25,146.94) .. controls (297.79,147.48) and (296.95,148.55) .. (296.09,150.21) .. controls (291.74,145.13) and (256.44,155.34) .. (253.4,155) .. controls (257.71,142.19) and (245.13,130.06) .. (240.35,121.67) .. controls (254.18,120) and (278.2,116.47) .. (279.32,105.51) .. controls (286.43,106.6) and (290.26,101.35) .. (290.72,97.72) .. controls (294.7,100.41) and (296.91,98.72) .. (299.84,97.26) -- cycle ;
\draw [color={rgb, 255:red, 155; green, 155; blue, 155 }  ,draw opacity=1 ]   (237.68,115.5) .. controls (241.47,130.34) and (260.29,136.84) .. (252.74,157.83) ;
\draw [color={rgb, 255:red, 155; green, 155; blue, 155 }  ,draw opacity=1 ]   (279.67,100.9) .. controls (281.07,117.67) and (251.74,122.17) .. (236.74,120.5) ;
\draw [color={rgb, 255:red, 155; green, 155; blue, 155 }  ,draw opacity=1 ]   (289.73,92.7) .. controls (293.48,103.12) and (281.31,108.5) .. (275.05,103.61) ;
\draw [color={rgb, 255:red, 155; green, 155; blue, 155 }  ,draw opacity=1 ]   (287.62,93.16) .. controls (289.62,98.71) and (294.05,100.05) .. (299.84,97.26) ;
\draw [color={rgb, 255:red, 155; green, 155; blue, 155 }  ,draw opacity=1 ]   (297.53,152.25) .. controls (291.15,142.9) and (262.43,157.06) .. (247.9,155) ;
\draw [color={rgb, 255:red, 155; green, 155; blue, 155 }  ,draw opacity=1 ]   (303.86,150.34) .. controls (302.44,144.28) and (295.47,147.78) .. (294.48,154.5) ;
\draw [color={rgb, 255:red, 155; green, 155; blue, 155 }  ,draw opacity=1 ]   (307.55,146.35) .. controls (302.97,145.02) and (300.12,144.16) .. (299.12,150.88) ;
\draw  [line width=1.5]  (128.17,135.17) .. controls (128.17,85.46) and (168.46,45.17) .. (218.17,45.17) .. controls (267.87,45.17) and (308.17,85.46) .. (308.17,135.17) .. controls (308.17,184.87) and (267.87,225.17) .. (218.17,225.17) .. controls (168.46,225.17) and (128.17,184.87) .. (128.17,135.17) -- cycle ;
\draw [color={rgb, 255:red, 155; green, 155; blue, 155 }  ,draw opacity=1 ]   (130.51,118.04) .. controls (134.97,119.73) and (139.13,125.74) .. (140.67,119.13) ;
\draw [color={rgb, 255:red, 155; green, 155; blue, 155 }  ,draw opacity=1 ]   (138.85,118.61) .. controls (139.77,124.76) and (155.88,127.41) .. (157.42,120.79) ;
\draw [color={rgb, 255:red, 155; green, 155; blue, 155 }  ,draw opacity=1 ]   (154.56,120.81) .. controls (160.16,130.64) and (203.67,127.79) .. (225.33,113.29) ;
\draw [color={rgb, 255:red, 155; green, 155; blue, 155 }  ,draw opacity=1 ]   (292.75,150.63) .. controls (290.18,135.53) and (225.75,130.13) .. (221.58,111.46) ;
\draw [color={rgb, 255:red, 155; green, 155; blue, 155 }  ,draw opacity=1 ]   (292.67,145.63) .. controls (278.67,152.13) and (215,130.96) .. (216.67,149.5) ;
\draw [color={rgb, 255:red, 208; green, 2; blue, 27 }  ,draw opacity=1 ]   (164.5,134.79) .. controls (204.5,104.79) and (237.18,145.83) .. (279.07,143.87) ;
\draw [shift={(279.07,143.87)}, rotate = 357.31] [color={rgb, 255:red, 208; green, 2; blue, 27 }  ,draw opacity=1 ][fill={rgb, 255:red, 208; green, 2; blue, 27 }  ,fill opacity=1 ][line width=0.75]      (0, 0) circle [x radius= 1.34, y radius= 1.34]   ;
\draw [shift={(164.5,134.79)}, rotate = 323.13] [color={rgb, 255:red, 208; green, 2; blue, 27 }  ,draw opacity=1 ][fill={rgb, 255:red, 208; green, 2; blue, 27 }  ,fill opacity=1 ][line width=0.75]      (0, 0) circle [x radius= 1.34, y radius= 1.34]   ;
\draw [color={rgb, 255:red, 155; green, 155; blue, 155 }  ,draw opacity=1 ]   (150.27,167.08) .. controls (150.5,136.04) and (172.75,126.96) .. (218.83,146.71) ;
\draw [color={rgb, 255:red, 155; green, 155; blue, 155 }  ,draw opacity=1 ]   (152.17,167.29) .. controls (144.5,157.63) and (140.21,165.28) .. (134.21,167.59) ;

\draw (162.2,135.64) node [anchor=north west][inner sep=0.75pt]  [font=\small,color={rgb, 255:red, 178; green, 3; blue, 25 }  ,opacity=1 ]  {$\wt{y}_{0}$};
\draw (257,146) node [anchor=north west][inner sep=0.75pt]  [font=\small,color={rgb, 255:red, 178; green, 3; blue, 25 }  ,opacity=1 ]  {$\wt{f}^{n_{0}}(\wt{y}_{0})$};
\draw (220.56,75.87) node [anchor=south east] [inner sep=0.75pt]  [font=\small,color={rgb, 255:red, 177; green, 5; blue, 212 }  ,opacity=1 ]  {$V_{C_0}(\wt \gamma_{\wt z} )$};
\draw (157.44,109.48) node [anchor=south] [inner sep=0.75pt]  [font=\small,color={rgb, 255:red, 80; green, 200; blue, 0 } ,opacity=1 ]  {$\wt{\mathcal X}_1$};
\draw (221.94,146.88) node [anchor=north] [inner sep=0.75pt]  [font=\small,color={rgb, 255:red, 80; green, 200; blue, 0 } ,opacity=1 ]  {$\wt f^{n}(\wt{\mathcal X}_1)$};
\draw (283.44,98.49) node [anchor=south east] [inner sep=0.75pt]  [font=\small,color={rgb, 255:red, 170; green, 150; blue, 0 } ,opacity=1 ]  {$\wt{\mathcal X}_2$};
\end{tikzpicture}

\caption{Proof of Case~\ref{Ca2} of Theorem~\ref{ThmBndDevIrrat2}. Note that $\wt{\mathcal X}_2$ intersects $V_{C_0}(\wt\gamma_{\wt z})$.\label{FigProofCase2}}

\end{center}
\end{figure}
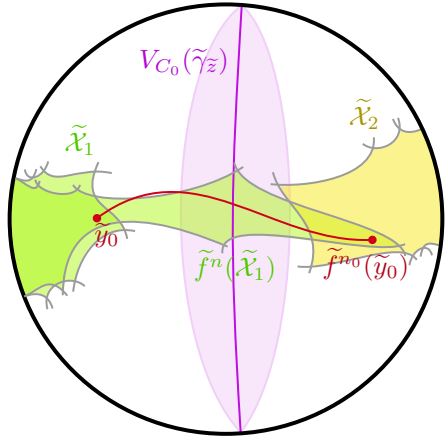

\paragraph{Case~\ref{Ca2}.}
Cases \ref{Ca2}.~and \ref{Ca3}.~are symmetrical, so we only treat Case~\ref{Ca2}.  
Reasoning as in the proof of Proposition~\ref{PropExistInterTransIrrat}, we get that there exists $\wt y'\in X_\varep \cap L(V_{C_0}(\wt\gamma_{\wt z}))$ such that $\wt f^{n_0-n'_1}(\wt y') \in R(V_{C_0}(\wt\gamma_{\wt z}))$
(see Figure~\ref{FigProofCase2}).

It then suffices to use Lemma~\ref{LemAPlusNotEquiv} and Lemma~\ref{CoroBoundedTyp} (for $y_0 = z$) or Lemma~\ref{LemThmBndDevIrrat2} (general case), as in the previous cases.
\end{proof}

\begin{proof}[Proofs of {Lemma~\ref{CoroBoundedTyp} and}  Lemma~\ref{LemThmBndDevIrrat2}]
If there exists a skeleton $X'_\varep\subset V_\varep$ (given by Remark~\ref{RemLemAPlusNotEquiv})
and $\wt \X'_1$ a closure of a connected components of the complement of $\wt X'_\varep$, and $0\le n'_1\le n_1$ such that $\wt f^{n'_1}(\wt y_0)\in \wt\X_1$ and $\wt \X_1 \subset L(V_{C_0}(\wt\gamma_{\wt z}))$, then the lemma is a consequence of Lemma~\ref{LemAPlusNotEquiv}.
\bigskip

So we suppose we are not in this case: for any skeleton $X'_\varep\subset V_\varep$ and any closure $\wt \X'_1$ of a connected component of the complement of $\wt X'_\varep$ meeting the segment of orbit $\wt y_0,\dots, \wt f^{n_1}(\wt y_0)$, one has $\wt \X'_1 \cap V_{C_0}(\wt \gamma_{\wt z})\neq\emptyset$.

Fix one skeleton $X_\varep\subset V_\varep$. By Lemma~\ref{LemExistC4}, the segment of orbit $\wt y_0,\dots, \wt f^{n_1}(\wt y_0)$ stays in a single closure $\wt \X_1$ of a connected component of the complement of $\wt X_\varep$.
Applying Lemma~\ref{LemMeetPhiTwice}, one gets that there exists $0\le n'_1 < n''_1\le n_1$, $T_1\in\G\setminus\{\Id\}$, $1\le j_1\le p$ and a lift $\wt U_{j_1}$ of $U_{j_1}$ to $\wt S$ such that $\wt f^{n'_1}(\wt y_0)\in \wt U_{j_1}$ and $\wt f^{n''_1}(\wt y_0)\in T_1\wt U_{j_1}$.

\begin{claim}
The axis of $T_1$ does not cross the geodesic $\gamma_z$.
\end{claim}

\begin{proof}
We have two cases.

Either the chart $U_{j_1}$ does not intersect $X_\varep$. 
The segment of orbit $\wt y_0,\dots, \wt f^{n_1}(\wt y_0)$ stays in the set $\wt \X_1 \cap \fil(\wt A_\varep)$, so the trajectory under the isotopy $ I^{[n'_1, n''_1]}(y_0)$ is homotopic (relative to endpoints) to a curve that does not intersect $X_\varep$. As $U_{j_1} \cap X_\varep = \emptyset$, the loop obtained by closing this curve with a curve inside $U_{j_1}$ is disjoint from $X_\varep$. This implies that the axis of $T_1$ does not cross the geodesic $\gamma_z$. 

Or the chart $U_{j_1}$ intersects $X_\varep$. By Remark~\ref{RemLemAPlusNotEquiv}, we have that $U_{j_1}\subset V_\varep$. Recall that $ f^{n'_1}(y_0), f^{n''_1}(y_0)\in U_\varep$; by a local perturbation one can modify $X_\varep$ to get another skeleton $X'_\varep\subset V_\varep$ such that $ \wt f^{n'_1}(\wt y_0)$ and $\wt f^{n''_1}(\wt y_0)$ belong to different connected components of the complement of $X'_\varep$. By Remark~\ref{RemLemAPlusNotEquiv} one can work with $X'_\varep$ instead of $X_\varep$. This contradicts the hypothesis made above.
\end{proof}


Let us call $\alpha_1$ the transverse loop obtained from closing the trajectory $I^{[n'_1, n''_1]}_\F(y_0)$ inside $U_{j_1}$, and $\wt\alpha_1$ a lift of $\alpha_1$ to $\wt S$, $\wt B_1$ the union of leaves of $\wt\F$ crossing $\wt\alpha_1$ such that $I^{[n'_1, n''_1]}_{\wt \F}(\wt y_0)$ draws $\wt B_1$. 

The fact that the geodesic axis of $T_1$ does not cross $\gamma_z$ allows us to apply Lemma~\ref{LemDrawCrosses}, which implies that the trajectory $I^\Z_{\wt \F}(\wt z)$ does not draw the band $\wt B_1$. In particular, applying the proof we have made up to now to the orbit of $\wt z$ in place of the orbit of $\wt y_0$, we deduce that {Lemma~\ref{CoroBoundedTyp}} holds. 
This, the second conclusion of Theorem~\ref{ThmBndDevIrrat2}, and \cite[Lemma~2.9]{paper1PAF} show that $I^\Z_{\wt\F}(\wt z)$ is contained in $V_{C_4+M}$. In particular, $L(\wt B)$ does not intersect $R(V_{C_4+M})$ and $R(\wt B)$ does not intersect $L(V_{C_4+M}(\wt\gamma_{\wt z}))$ (recall that $\wt B$ is the union of the leaves met by $I^\Z_{\wt\F}(\wt z)$).

As the trajectory $I^\Z_{\wt \F}(\wt z)$ does not draw the band $\wt B_1$, we deduce that there exists $0\le t_1\le n_1$ such that $I^{t_1}_{\wt \F}(\wt y_0)\notin\wt B$. 
But as  $\{\wt y_0, \dots, \wt f^{n_1}(\wt y_0)\} \subset L(V_{C_4+2M}(\wt\gamma_{\wt z}))$, again by \cite[Lemma~2.9]{paper1PAF} we deduce that $I^{[0,n_1]}_{\wt{\F}}(\wt y_0)$ is contained in $L(V_{C_4+M}(\wt\gamma_{\wt z}))$, which implies that  $I^{t_1}_{\wt \F}(\wt y_0)\in L(\wt B)$ as it is neither in $\wt B$ nor in $R(\wt B)$. This proves Lemma~\ref{LemThmBndDevIrrat2}.
\end{proof}



\section{Bounded deviations in the direction of the lamination}\label{SecBoundDirec}

\subsection{Building good approximations}\label{SubsecApprox}

Let $z$ be a typical point for an ergodic measure $\mu$ that belongs to a minimal non-closed class (Definition~\ref{DefClassMeas}). Let $\wt z$ be a lift of $z$ to $\wt S$. Denote $\wt\beta_0 = I^\Z_{\wt \F}(\wt z)$.

For $\wt A\subset\wt S$ and $R>0$, denote $V_R(\wt A)$ the set of points of $\wt S$ at a distance $\le R$ of one point of $\wt A$.
%
%

Let $\wt W$ be a small open ball around $\wt z$, that is a trivializing chart for $\wt\F$. As $z$ is recurrent and the trajectory of $\wt z$ is proper, there exists a deck transformation $T_0 \neq\Id$ and $n>0$ such that $\wt f^n(\wt z) \in T_0\wt W$. This allows us to define an approximation $\alpha_0$ of the trajectory $\beta_0$ of $z$ (see \cite[Subsection~2.3]{paper1PAF}), such that the path $\wt\alpha_0$ is $T_0$-invariant. 
Denote $\wt B_0$ the set of leaves met by $\wt\alpha_0$.

\begin{lemma}\label{LemAlphaiSimple}
Let $\alpha_0$ be an approximation of $\beta_0$. Then there is another approximation $\alpha'_0$ associated to a deck transformation $T'_0$ such that the transverse path $\wt\alpha'_0$ is simple and that $\langle T'_0\rangle = \langle T_0\rangle$. 
\end{lemma}

\begin{proof}
By hypothesis, $\wt\alpha_1$ is $T_1$-invariant for some $T_1\in\G$. One can suppose that for any $t\in\R$ we have $\wt\alpha_1(t+1) = T_1\wt\alpha_1(t)$. 
As $\alpha_1$ is an approximation of $I^\Z_\F(z)$, by shifting the parametrization of $\wt\alpha_1$ if necessary, we can suppose that there exists $s_i<s'_i$ such that $\wt\alpha_1|_{[0,1]}$ is $\wt\F$-equivalent to $I^{[s_i, s'_i]}_{\wt\F}(\wt z)$.

Let us denote $\wc\alpha_1$ the projection of $\wt\alpha_1$ to the annulus $\wt S/T_1$ and $\wc \F$ the projection of $\wt \F$ to $\wt S/T_1$.
Let us first show that a leaf cannot meet $\wc\alpha_1$ an infinite number of times. 
Suppose by contradiction that $\wc\alpha_1$ meets a leaf $\wc\phi$ an infinite number of times. Consider $t_0$ such that $\wc\alpha_1(t_0)$ is a accumulation point of intersections of $\wc\alpha_1$ with $\wc\phi$. By Poincaré-Bendixson theorem, this implies that $\wc\alpha_1(t_0)$ is in the boundary of a disc in the two ends compactification of the annulus $\wt S/T_1$ (this compactification is a sphere), such that any transverse trajectory meeting the boundary of this disc is wandering. This is a contradiction.

Hence, by restricting $\alpha_1$ to another approximation $\alpha'_1$ of $\beta_0$ (with a deck transformation $T'_0$ satisfying $\langle T'_0\rangle = \langle T_0\rangle$), we can suppose that $\wc\alpha_1|_{(0,1)}$ does not meet the leaf $\wc\phi_{\wc\alpha_1(0)}$.
\bigskip

Suppose by contradiction that $\wt\alpha_1$ is not simple: there exists $t\neq t'$ such that $\wt\alpha_1(t) = \wt\alpha_1(t')$. Note that it implies that $t-t'\notin\Z$. By translating $t$ and $t'$ if necessary, one can suppose that $t\in[0,1)$, so writing $k=\lfloor t'\rfloor$ and $t'' = t'-k$, we have that $\wt\alpha_1(t) = T_1^k\wt\alpha_1(t'')$ (with $t''\in[0,1)$). 
As $\wt\alpha_1|_{[0,1]}$ is $\wt\F$-equivalent to $I^{[s_i, s'_i]}_{\wt\F}(\wt z)$ and $I^{\Z}_{\wt\F}(\wt z)$ does not meet any leaf of $\wt\F$ twice (Lemma~\ref{LemAlphaSimple}), we deduce that $k\neq 0$. By \cite[Lemma~2.1]{paper1PAF}, we can suppose that $k = \pm 1$. 
By exchanging the roles of $t$ and $t''$ if necessary, one can also suppose that $t<t''$. 
By changing $T_1$ by $T_1^{-1}$ if necessary, we can suppose that $k=1$.
Because $t-t'\notin\Z$, we have $t''-t<1$

In the annulus $\wt S/T_1$ we have $\wc\alpha_1(t) = \wc\alpha_1(t'')$. 
As $\wt\alpha_1|_{[0,1]}$ is simple, and as $\wt\alpha_1$ is uniformly continuous, there exists $\delta>0$ such that if $\wc\alpha_1(s) = \wc\alpha_1(s'')$, then $|s-s''|\ge\delta$. Consider an interval $[s, s'']\subset [t,t'']\subset[0,1]$, minimal for inclusion among those with the property that $\wc\alpha_1(s) = \wc\alpha_1(s'')$; by construction $\wc\alpha_1|_{[s, s'')}$ is simple. Because $t''-t\notin\Z$ and $t,t''\in[0,1]$, we have $s''-s<1$.

Denote $\wc B$ the set of leaves of $\wc \F$ (the projection of $\wt \F$ to $\wt S/T_1$) crossed by $\wc\alpha_1|_{[s, s'')}$. Because $\wc\alpha_1|_{[s, s'')}$ is simple, this set is an annulus.
We have seen in the beginning of the proof that $\wc\alpha_1|_{(0,1)}$ is disjoint from $\wc\phi_{\wc\alpha_1(0)}$. This implies that $\wc\phi_{\wc\alpha_1(0)}\not\subset\wc B$. 
Denote $s_1 = \max\{t\in [0,s]\mid \wc\alpha_1(t)\notin \wc B\}$ and $s''_1 = \min\{t\in [s'',1]\mid \wc\alpha_1(t)\notin \wc B\}$ (such numbers exist because $\wc\alpha_1(0)\notin\wc B$ and $\wc\alpha_1(1)\notin\wc B$).
We have two cases:
\begin{itemize}
\item Either $\wc\alpha_1(s_1)$ and $\wc\alpha_1(s''_1)$ lie in the same side of $\wc B$. This implies that $\wc\alpha_1|_{[0,1]}$ draws and visits $\wc B$, and so by \cite[Proposition~2.15]{paper1PAF} that $\wc\alpha_1|_{[0,1]}$ has a self $\wc\F$-transverse intersection, contradicting Lemma~\ref{PropPasInterTrans}.
\item Or $\wc\alpha_1(s_1)$ and $\wc\alpha_1(s''_1)$ lie in different sides of $\wc B$. This implies that $\wc\alpha_1|_{[s_1,s''_1]}$ draws and crosses $\wc B$. In particular, one of $\wc\alpha_1(s_1)$ and $\wc\alpha_1(s''_1)$ lie on the same side of $\wc B$ as $\wc\alpha_1(0) = \wc\alpha_1(1)$. Let us treat the case where it is $\wc\alpha_1(s_1)$, the other one being identical. This means that $\wc\alpha_1(1)$ and $\wc\alpha_1(s''_1)$ lie on different sides of $\wc B$, hence $\wc\alpha_1|_{[s''_1, 1]}$ crosses $\wc B$ in the other direction than $\wc\alpha_1|_{[s_1,s''_1]}$. By \cite[Proposition~2.16]{paper1PAF}, this  implies that $\wc\alpha_1|_{[0,1]}$ has a self $\wc\F$-transverse intersection, contradicting Lemma~\ref{PropPasInterTrans}.
\end{itemize}
\end{proof}
%
%
%
%
%

\begin{lemma}\label{LemTrackCrossAxis}
The tracking geodesic $\wt\gamma_{\wt z}$ of $\wt z$ crosses the axis of the deck transformation $T_0$.
\end{lemma}

\begin{proof}
By \cite[Proposition~2.18]{paper1PAF}, the trajectories $\wt\alpha_0$ and $\wt\beta_0$ cannot be equivalent (neither at $+\infty$, nor at $-\infty$): if this was true this proposition would imply that there is an infinite number of iterates of $\wt z$ that stay at finite distance to the axis of $T_0$, which is impossible as the tracking geodesic $\wt\gamma_{\wt z}$ of $\wt z$ has no common endpoint with a closed geodesic.

By \cite[Proposition~2.4]{paper1PAF}, the trajectory $\wt\beta_0$  cannot accumulate in $\wt\alpha_0$ (neither at $+\infty$, nor at $-\infty$), otherwise $\beta_0$ would accumulate in itself as $\alpha_0$ is made of pieces of $\beta_0$.

Hence, the trajectory $\wt\beta_0$ has to go out of $\wt B_0$ both in negative and positive times. By Proposition~\ref{PropPasInterTrans}, the trajectory $\wt\beta_0$ has no self-transverse intersection; \cite[Proposition~2.15]{paper1PAF} prevents $\wt\beta_0$ to visit and draw $\wt B_0$.

Hence, the trajectory $\wt\beta_0$ has to cross the band $\wt B_0$, either from left to right or from right to left. Suppose, without losing generality, that it is from left to right.

By \cite[Proposition~2.15]{paper1PAF}, the trajectory $\wt\beta_0$ cannot moreover cross the band $\wt B_0$ from right to left (because, again, by Proposition~\ref{PropPasInterTrans}, the trajectory $\wt\beta_0$ has no self-transverse intersection). Combining these facts, 
we deduce that after the drawing component of $\wt\beta_0$ in $\wt B_0$, one has no other drawing component and we only have visiting components in $\wt B_0$. This implies that there is $t_n\to{n\to +\infty} +\infty$ such that $\beta_0(t_n)\in R(\wt\alpha_0)$. Similarly, there is $t_n\to{n\to -\infty} -\infty$ such that $\beta_0(t_n)\in L(\wt\alpha_0)$. 

In particular, the tracking geodesic $\wt\gamma_{\wt z}$ of $\wt z$ either has a common endpoint with the axis of $T_0$ or crosses the axis of the deck transformation $T_0$. The first case is impossible, as $\gamma_z$ is a geodesic of a minimal lamination without closed leaves, and the axis of $T_0$ is a closed geodesic.
\end{proof}

For $\wt x\in \wt S$, define $d_{\text{alg}}(\wt x, \wt\gamma)$ as the algebraic distance of $\wt x$ to $\wt\gamma$ (its absolute value is equal to the distance between $\wt x$ and $\wt \gamma$, and its sign is negative on the left of $\wt\gamma$ and  positive on the right of $\wt\gamma$).

\begin{lemma}\label{LemHyperbolicGeom}
Let $T\in\G$ be a deck transformation and $\wt\gamma$ a geodesic. Denote $\wt\gamma_T$ the geodesic axis of $T$. Suppose that $\wt\gamma\cap\wt\gamma_T\neq\emptyset$ and that $\wt\gamma\neq\wt\gamma_T$. Then the following properties are equivalent:
\begin{enumerate}[label=(\roman*)]
\item $\wt\gamma_T$ crosses $\wt\gamma$ from left to right;
\item for any $\wt x\in\wt S$, we have $d_{\text{alg}}(T\wt x,\wt\gamma)> d_{\text{alg}}(\wt x,\wt\gamma)$;
\item there exists $\wt x\in\wt S$ such that $d_{\text{alg}}(T\wt x,\wt\gamma)\ge d_{\text{alg}}(\wt x,\wt\gamma)$.
\end{enumerate}
\end{lemma}

\begin{figure}
\begin{center}

\tikzset{every picture/.style={line width=0.75pt}} 

\begin{tikzpicture}[x=0.75pt,y=0.75pt,yscale=-1,xscale=1]

\draw  [draw opacity=0][fill={rgb, 255:red, 27; green, 159; blue, 0 }  ,fill opacity=0.08 ] (280,110) -- (360,110) -- (320,210) -- cycle ;
\draw  [draw opacity=0][fill={rgb, 255:red, 27; green, 159; blue, 0 }  ,fill opacity=0.08 ] (289.73,210.31) .. controls (284.97,201.73) and (282.25,191.81) .. (282.25,181.24) .. controls (282.25,148.64) and (308.12,122.21) .. (340.03,122.21) .. controls (371.94,122.21) and (397.81,148.64) .. (397.81,181.24) .. controls (397.81,191.73) and (395.13,201.57) .. (390.44,210.1) -- (340.03,181.24) -- cycle ; \draw  [color={rgb, 255:red, 27; green, 159; blue, 0 }  ,draw opacity=1 ] (289.73,210.31) .. controls (284.97,201.73) and (282.25,191.81) .. (282.25,181.24) .. controls (282.25,148.64) and (308.12,122.21) .. (340.03,122.21) .. controls (371.94,122.21) and (397.81,148.64) .. (397.81,181.24) .. controls (397.81,191.73) and (395.13,201.57) .. (390.44,210.1) ;  
\draw  [draw opacity=0][fill={rgb, 255:red, 255; green, 255; blue, 255 }  ,fill opacity=1 ] (289.78,210.71) .. controls (299.97,193.63) and (318.5,182.21) .. (339.66,182.21) .. controls (360.98,182.21) and (379.63,193.81) .. (389.77,211.11) -- (339.66,241.11) -- cycle ; \draw  [color={rgb, 255:red, 27; green, 159; blue, 0 }  ,draw opacity=1 ] (289.78,210.71) .. controls (299.97,193.63) and (318.5,182.21) .. (339.66,182.21) .. controls (360.98,182.21) and (379.63,193.81) .. (389.77,211.11) ;  
\draw [color={rgb, 255:red, 27; green, 159; blue, 0 }  ,draw opacity=1 ]   (280,110) -- (320,210) ;
\draw [color={rgb, 255:red, 27; green, 159; blue, 0 }  ,draw opacity=1 ]   (360,110) -- (320,210) ;
\draw [color={rgb, 255:red, 0; green, 105; blue, 219 }  ,draw opacity=1 ]   (320,110) -- (320,210) ;
\draw  [draw opacity=0] (290.01,209.09) .. controls (290.49,181.9) and (312.69,160) .. (340,160) .. controls (367.61,160) and (390,182.39) .. (390,210) .. controls (390,210.07) and (390,210.14) .. (390,210.21) -- (340,210) -- cycle ; \draw  [color={rgb, 255:red, 0; green, 105; blue, 219 }  ,draw opacity=1 ] (290.01,209.09) .. controls (290.49,181.9) and (312.69,160) .. (340,160) .. controls (367.61,160) and (390,182.39) .. (390,210) .. controls (390,210.07) and (390,210.14) .. (390,210.21) ;  
\draw [line width=1.5]    (260,210) -- (420,210) ;

\draw (320,106.6) node [anchor=south] [inner sep=0.75pt]  [color={rgb, 255:red, 0; green, 102; blue, 221 }  ,opacity=1 ,xscale=1.2,yscale=1.2]  {$\wt{\gamma }$};
\draw (380,167.27) node [anchor=south] [inner sep=0.75pt]  [color={rgb, 255:red, 0; green, 102; blue, 221 }  ,opacity=1 ,xscale=1.2,yscale=1.2]  {$\wt{\gamma }_{T}$};

\end{tikzpicture}

\vspace{-20pt}
\caption{Proof of Lemma~\ref{LemHyperbolicGeom}.}\label{FigLemHyperbolicGeom}
\end{center}
\end{figure}
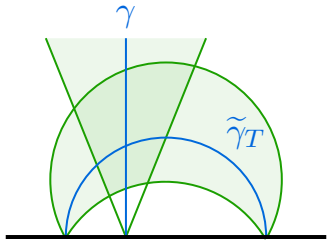

\begin{proof}
Note that it suffices to prove that if $\wt\gamma_T$ crosses $\wt\gamma$ from left to right, then for any $\wt x\in\wt S$, we have $d_{\text{alg}}(T\wt x,\wt\gamma) > d_{\text{alg}}(\wt x,\wt\gamma)$: this will imply, by symmetry, that if $\wt\gamma_T$ crosses $\wt\gamma$ from right to left, then for any $\wt x\in\wt S$, we have $d_{\text{alg}}(T\wt x,\wt\gamma)\le d_{\text{alg}}(\wt x,\wt\gamma)$.

Let us use the Poincaré half-plane model $\Hy_+$ (see Figure~\ref{FigLemHyperbolicGeom}). By applying an isometry if necessary, one can suppose that $\wt\gamma = i\R_+^*$, and that $\wt\gamma_T$ is a half-circle with center on the real axis, starting at $y_-\in\R_-^*$ and ending at $y_+\in\R_+^*$. Note that the set of points that are equidistant to $\wt\gamma$ (for the algebraic distance $d_{\text{alg}}$), called \emph{hypercycles} or \emph{hypercircles}, are of the type $z\R_+^*$ with $z\in\Hy_+$; in particular the algebraic distance to $\wt\gamma$ grows iff the argument decreases.
Similarly, the set of points that are equidistant to $\wt\gamma_T$ are intersections of circles passing through both $y_-$ and $y_+$ with $\Hy_+$. 

It then suffices to note that browsing the circle passing through both $y_-$ and $y_+$ in the indirect direction (which does the orbit of a point under $T$) the argument of the point decreases, hence the algebraic distance to $\wt\gamma$ increases.
\end{proof}


The following lemma is based on Atkinson's theorem \cite{zbMATH03533870}.

\begin{lemma}\label{LemRtoLandLtoR}
There exist two approximations $\wt\alpha_{0}$ and $\wt\alpha_{1}$ of the trajectory $I^\Z_{\wt \F}(\wt z)$, the first one associated to a deck transformation $T_0$ whose axis $\wt\gamma_0$ crosses $\wt\gamma_{\wt z}$ from left to right and the second one associated to a deck transformation $T_1$ whose axis $\wt\gamma_1$ crosses $\wt\gamma_{\wt z}$ from right to left. Moreover, one can suppose that $\wt\gamma_0$ and $\wt\gamma_1$ are arbitrarily close to $\wt\gamma_{\wt z}$. 
\end{lemma}

\begin{proof}
Let $\wt W$ be a small open chart for the foliation $\wt \F$ whose projection $W$ on $S$ has positive $\mu$-measure. Fix $\wt x_0\in \wt W$. 
For any $\wt x \in\wt S$ whose projection $x\in S$ belongs to $W$, define $T_{\wt x}$ as the deck transformation such that $\wt x\in T_{\wt x}\wt W$. Note that 
\begin{equation}\label{EqLemRtoLandLtoR}
d\big(\wt x,\, T_{\wt x} \wt x_0\big) \le \diam(\wt W).
\end{equation}

For $x\in W$, let $\tau_0(x)>0$ be the first returm time of $x$ in $W$: $f^{\tau_0(x)}(x)\in W$; it is well defined $\mu$-almost everywhere. For $x\in W$ whose first return time to $W$ is well defined, and $\wt x$ a lift of $x$ to $S$, define $\wt g(x) = \wt f^{\tau_0(x)}(\wt x)$. 

For $z\in S$ that is $\mu$-typical, and $\wt z$ a lift of $z$ to $\wt S$, define
\[\varphi(z) = d_{\text{alg}}\big(T_{\wt g(\wt z)}(\wt x_0), \wt\gamma_{\wt z}\big) - d_{\text{alg}}\big(T_{\wt z}\wt x_0, \wt\gamma_{\wt z}\big).\]

Note that the value of $\varphi(z)$ is independent of the choice of the lift $\wt z$, as $\wt g$, $\wt x\mapsto T_{\wt x}$ and $\wt x\mapsto \wt\gamma_{\wt x}$ are $\G$-equivariant. Moreover, the map $\varphi$ is Borel, as $\wt x\mapsto T_{\wt x}$ and $\wt z\mapsto \wt\gamma_{\wt z}$ are (see \cite[Theorem B]{alepablo}).
\bigskip

As we will apply Atkinson's theorem, we first have to prove that 
\begin{equation}\label{EqInt0Atk}
\int_S \varphi\dd\mu = 0.
\end{equation}

Remark that  
\[|\varphi(\wt z)| \le  \big|d_{\text{alg}}\big(\wt g(\wt z), \wt\gamma_{\wt z}\big) - d_{\text{alg}}\big(\wt z, \wt\gamma_{\wt z}\big)\Big| + 2\diam(\wt W) \le d\big(\wt g(\wt z), \wt z\big) + 2\diam(\wt W),\]
and that $d\big(\wt g(\wt z), \wt z\big) \le \tau_0(\wt z) d\big(\wt f, \Id_{\wt S}\big)$ is $\mu$-integrable because of Kac's recurrence theorem. 

Let $\mu_0$ be the measure induced by $\mu$ via the first return of $f$ on $W$. By ergodicity of $\mu_0$, and by Birkhoff's ergodic theorem, for $\mu_0$-a.e.~$z\in W$, using the fact that $\wt\gamma_{\wt f(\wt z)} = \wt\gamma_{\wt z}$ and hence that $\wt\gamma_{\wt g(\wt z)} = \wt\gamma_{\wt z}$, 
\begin{align*}
\int_S \varphi\dd\mu_0
& = \lim_{n\to+\infty}\frac 1n\sum_{k=0}^{n-1} \varphi(g^k(z)) \\
& = \lim_{n\to+\infty} \frac 1n\Big(d_{\text{alg}}\big(T_{\wt g^n(\wt z)} \wt x_0 , \wt\gamma_{\wt z}\big) - d_{\text{alg}}\big(T_{\wt z}\wt x_0, \wt\gamma_{\wt z}\big)\Big)\\
& = \lim_{n\to+\infty} \frac 1n d_{\text{alg}}\big(\wt g^n(\wt z), \wt\gamma_{\wt z}\big)
\end{align*}
(the last equality comes from \eqref{EqLemRtoLandLtoR}).
By the definition of tracking geodesic, the last quantity is equal to 0, proving that $\int_S \varphi\dd\mu_0 = 0$ and hence \eqref{EqInt0Atk}. 
\bigskip

Let us prove that $n\mapsto d_{\text{alg}}\big(T_{\wt g^n(\wt z)} \wt x_0, \wt\gamma_{\wt z}\big)$ is not $\mu_0$-a.e.~constant. Suppose it is the case. Then $d\big( \wt x_0,T_{\wt g^n(\wt z)}^{-1} \wt\gamma_{\wt z}\big)$ is constant, in other words the geodesics $T_{\wt g^n(\wt z)}^{-1} \wt\gamma_{\wt z}$ are all tangent to a single circle of center $\wt x_0$. 
But as $\wt g^n(\wt z)$ tends to $\omega(\wt\gamma_{\wt z})$ when $n$ goes to $+\infty$, there is a sequence $(n_k)$ going to $+\infty$ such that the $T_{\wt g^{n_k}(\wt z)}^{-1}$ are pairwise different and hence, because $\wt\gamma_{\wt z}$ is simple, that the $T_{\wt g^{n_k}(\wt z)}^{-1} \wt\gamma_{\wt z}$ are pairwise disjoint. 
This contradicts the fact that $\gamma_{\wt z}$ is simple : there is only a finite number of pairwise disjoint geodesics that are all tangent to a single circle. 

Hence, (changing $z$ by an iterate of it if necessary) there is $n_0>0$ such that $\sum_{k=0}^{n_0-1} \varphi(g^k(z))\neq 0$. 
\bigskip

By Atkinson's theorem \cite{zbMATH03533870}, for any $\varep>0$ and $\mu_0$-a.e.~$z\in\wt W$, there is an infinite number of $n\in\N$ such that 
\[\left|\sum_{k=0}^{n-1} \varphi(g^k(z)) \right| = 
\left|d_{\text{alg}}\big(T_{\wt g^n(\wt z)} \wt x_0, \wt\gamma_{\wt z}\big) - d_{\text{alg}}\big(T_{\wt z}\wt x_0, \wt\gamma_{\wt z}\big)\right|\le\varep.\]

Applying this result to $\varep = \big|\sum_{k=0}^{n_0-1} \varphi(g^k(z))\big|/2$, we deduce that there exists $n_1>n_0$ such that 
\[\left|\sum_{k=0}^{n_1-1} \varphi(g^k(z))\right| \le \frac12 \left|\sum_{k=0}^{n_0-1} \varphi(g^k(z))\right|,\]
in particular 
\[\operatorname{sign}\left(\sum_{k=0}^{n_0-1} \varphi(g^k(z))\right) = -\operatorname{sign}\left(\sum_{k=n_0}^{n_1-1} \varphi(g^k(z))\right),\]
or equivalently
\begin{multline*}
\operatorname{sign}\Big(d_{\text{alg}}\big(T_{\wt g^{n_0}(\wt z)} \wt x_0, \wt\gamma_{\wt z}\big) - d_{\text{alg}}\big(T_{\wt z}\wt x_0, \wt\gamma_{\wt z}\big)\Big)\\
 = -\operatorname{sign}\Big(d_{\text{alg}}\big(T_{\wt g^{n_1}(\wt z)} \wt x_0, \wt\gamma_{\wt z}\big) - d_{\text{alg}}\big(T_{\wt g^{n_0}(\wt z)}\wt x_0, \wt\gamma_{\wt z}\big)\Big).
\end{multline*}

As the times $n_0$ and $n_1$ are return times to $W$, we have seen in Lemma~\ref{LemTrackCrossAxis} that it forces the geodesic axes of $T_{\wt g^{n_0}(\wt z)}$ and $T_{\wt g^{n_0}(\wt z)}^{-1} T_{\wt g^{n_1}(\wt z)}$ to cross $\wt\gamma_{\wt z}$. This allows us to apply Lemma~\ref{LemHyperbolicGeom} that implies that one of the geodesic axes of $T_{\wt g^{n_0}(\wt z)}$ and $T_{\wt g^{n_0}(\wt z)}^{-1} T_{\wt g^{n_1}(\wt z)}$ crosses $\wt\gamma_{\wt z}$ from left to right, and the other one crosses $\wt\gamma_{\wt z}$ from right to left. Note that the deck transformation $T_{\wt g^{n_0}(\wt z)}^{-1} T_{\wt g^{n_1}(\wt z)}$ is associated to the return number $n_1-n_0$ of $g^{n_0}(z)$ in $W$.

To these deck transformations  are associated two approximations $\wt\alpha_{0}$ and $\wt\alpha_{1}$ of the trajectory $I^\Z_{\wt \F}(\wt z)$, the first one being $T_{\wt g^{n_0}(\wt z)}$-invariant and the second one $T_{\wt g^{n_0}(\wt z)}^{-1} T_{\wt g^{n_1}(\wt z)}$-invariant. Without loss of generality, one can suppose that the first one crosses $\wt\gamma_{\wt z}$ from left to right and the second one crosses $\wt\gamma_{\wt z}$ from right to left. 
This proves the lemma.
\end{proof}

\subsection{Proof of Theorem~\ref{TheoBndedDirLam}}\label{SubsecTheoBndedDirLam}

Let $\wt\alpha_0$, $\wt\alpha_1$, $\wt\gamma_0$, $\wt\gamma_1$, $T_0$ and $T_1$ be given by Lemma~\ref{LemRtoLandLtoR}. Note that the distances between $\wt\alpha_0$ and $\wt\gamma_0$, and between $\wt\alpha_1$ and $\wt\gamma_1$, are uniformly bounded.

Recall that $\wt\gamma_0$ and $\wt\gamma_1$ are given by Lemma~\ref{LemRtoLandLtoR}.

\begin{lemma}\label{LemSyndeticInter}
There exists $C_0>0$ such that for any $t\in\R$, there are two deck transformations $R_0$ and $R_1$ such that $\wt\gamma_{\wt z}([t, t+C_0])\cap R_0\wt\gamma_0\neq \emptyset$ and $\wt\gamma_{\wt z}([t, t+C_0])\cap R_1\wt\gamma_1\neq \emptyset$. Moreover, the angle between these pieces of geodesics at the intersection points is bigger (in absolute value) than $C_0^{-1}$.
\end{lemma}

\begin{proof}
Denote $\gamma_z$, $\gamma_0$ and $\gamma_1$ the respective projections of $\wt\gamma_{\wt z}$, $\wt\gamma_0$ and $\wt\gamma_1$ on $S$.

Recall that $\wt\gamma_0$ and $\wt\gamma_1$ can be chosen arbitrarily close to $\wt\gamma_{\wt z}$. By Lemma~\ref{LemGeodNotAccuPerOrb}, this implies that $\gamma_0$ and $\gamma_1$ are included in the surface $S_\mu$. 
In particular, they are global sections for the lamination $\Lambda_\mu$: for $i=0,1$, for any geodesic $\gamma$ of $\Lambda_\mu$ and any $t\in\R$, there exist $t_i > t$ such that $\gamma(t_i)\in \gamma_i$. This applies in particular for $\gamma(t)\in \gamma_i$, that implies that the first return time for the geodesic flow restricted to $\Lambda_\mu$ to the section $\gamma_i$ is finite, hence (by continuity) locally bounded and hence (by compactness of $\gamma_i\cap\Lambda_\mu$) uniformly bounded. Note also that the angle between $\gamma_i$ and a geodesic of $\Lambda_i$ is nonzero at an intersection point, hence locally bounded away from 0 and hence uniformly bounded away from 0. This proves the lemma.
\end{proof}

\begin{lemma}\label{LemInterAlphaiEmpty}
There is $C_1>0$ an increasing sequence $(t^k)_k\in\R^\Z$ and, for $i=0,1$, sequences $(R_i^k)_k\in\G^\Z$ of deck transformations such that for any $k\in\Z$, we have (recall that we identify a geodesic with its $\R$ parametrization)
\begin{itemize}
\item $2C_1\le t^{k+1}-t^k\le 4C_1$, and
\item $\pr_{\wt\gamma_{\wt z}}\big(R_0^k\wt\alpha_0\cup R_1^k\wt\alpha_1\big) \subset [t^k-C_1, t^k+C_1]$.
\end{itemize}
Moreover, one can take $R_0^0 = R_1^0 = \Id$. 
\end{lemma}

Note that the first condition implies that the sequence $(t^k)_k$ meets any interval of length $4C_1$. Combining both conditions, one gets that for any $k\in\Z$, 
\[\big(R_0^k\wt\alpha_0\cup R_1^k\wt\alpha_1\big) \cap \big(R_0^{k+1}\wt\alpha_0 \cup R_1^{k+1}\wt\alpha_1 \big) = \emptyset.\]

\begin{proof}
By Lemma~\ref{LemSyndeticInter}, there exists two sequences $(R_0^k)_{k\in\Z}$ and $(R_1^k)_{k\in\Z}$ of deck transformations, and two increasing sequences $(\tilde t_0^k)_{k\in\Z}$ and $(\tilde t_1^k)_{k\in\Z}$ of times such that for $i=0,1$ and $k\in\Z$, we have $\wt\gamma_{\wt z}(\tilde t_i^k)\in R_i^k\wt\gamma_i$, and such that for $i=0,1$, the sequence $(\tilde t_i^k)_{k\in\Z}$ meets any interval of length $C_0$. Moreover, the angle between the geodesics $\wt\gamma_{\wt z}$ and $R_i^k\wt\gamma_i$ at the intersection points is bigger than $C_0^{-1}$. We will extract increasing sequences $(t_i^k)_k$ from $(\tilde t_i^k)_{k\in\Z}$ to get the conclusions of the lemma.

First, by shifting the indices of the sequences $(\tilde t_1^k)_k$ and $(R_1^k)_k$ if necessary, we can suppose that $|\tilde t_0^0 - \tilde t_1^0|\le C_0$.
Recall that we have $\wt\gamma_{\wt z}(\tilde t_i^k)\in R_i^k\wt\gamma_i$. By the fact that the angle between geodesics is bigger than $C_0^{-1}$ and the fact that the distance between $\wt\gamma_i$ and $\wt\alpha_i$ is bounded, we deduce that there exists $C'_1>C_0$ such that for any $k$ and for $i=0, 1$, we have$\pr_{\wt\gamma_{\wt z}}(R_i^k\wt\alpha_i) \subset [\tilde t_i^k-C'_1, \tilde t_i^k+C'_1]$. As a consequence, if $|\tilde t_i^k - \tilde t_{i'}^{k'}|\ge 2C'_1$, then $R_i^k\wt\alpha_i \cap R_{i'}^{k'}\wt\alpha_{i'} = \emptyset$.
Extracting recursively sequences $(t_i^k)_k$ from $(\tilde t_i^k)_k$, we get that for any $k$, we have:
\begin{itemize}
\item $|t_0^k - t_1^k|\le C_0$;
\item $\max(t_0^{k+1}, t_1^{k+1}) - \min(t_0^k, t_1^k)\le 2C'_1+2C_0$;
\item $\min(t_0^{k+1}, t_1^{k+1}) - \max(t_0^k, t_1^k)\ge 2C'_1$.
\end{itemize} 
This ensures the conclusions of the lemma, with $C_1 = C'_1+C_0$ and $t^k = t_0^k$.
\end{proof}

Note that in the previous proof, one can choose $R_0^0 = \Id_{\wt S}$. We will use this fact in the sequel.

For $i=0,1$, recall that $\wt B_i$ is the set of leaves met by $\wt\alpha_i$; by Lemma~\ref{LemAlphaiSimple} it is a topological plane.
Let $C_1$ be given by Lemma~\ref{LemInterAlphaiEmpty}. 

\begin{lemma}\label{LemCrossDependOnProj}
There exists $C_2>0$ such that, for any $\wt y\in\wt S$ and $n\in\N$:
\begin{itemize}
\item if $\pr_{\wt\gamma_{\wt z}}\big(\wt f^n(\wt y)\big) - \pr_{\wt\gamma_{\wt z}}(\wt y) \ge C_2$, then there exists $k\in\Z$ such that $I^{[0,n]}_{\wt \F}(\wt y)$ crosses $R_0^k\wt B_0$ from right to left;
\item if $\pr_{\wt\gamma_{\wt z}}\big(\wt f^n(\wt y)\big) - \pr_{\wt\gamma_{\wt z}}(\wt y) \le -C_2$, then there exists $k\in\Z$ such that $I^{[0,n]}_{\wt \F}(\wt y)$ crosses $R_0^k\wt B_0$ from left to right.
\end{itemize}
\end{lemma}

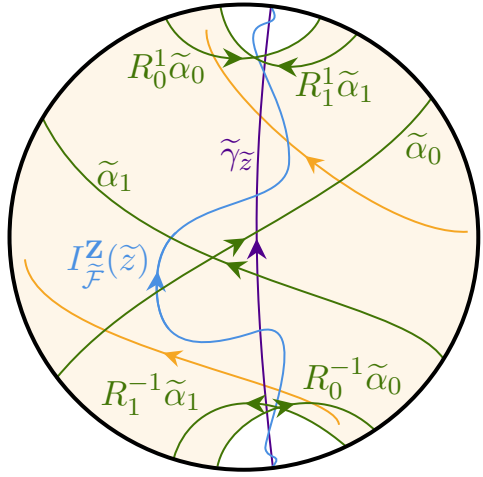
\begin{figure}
\begin{center}
\tikzset{every picture/.style={line width=0.75pt}} 

\begin{tikzpicture}[x=0.75pt,y=0.75pt,yscale=-1.3,xscale=1.3]
\draw  [draw opacity=0][fill={rgb, 255:red, 245; green, 166; blue, 35 }  ,fill opacity=0.1 ] (304.85,62.98) .. controls (306.26,72.93) and (311.8,78.93) .. (317.79,83.44) .. controls (326.88,82.47) and (338.73,78.47) .. (343.5,67.58) .. controls (421.94,97.03) and (423.94,199.03) .. (353.5,234.23) .. controls (340.73,223.7) and (347.65,228.62) .. (326.41,218.6) .. controls (315.19,223.7) and (308.42,228.77) .. (303.85,241.73) .. controls (205.37,232.18) and (191.08,83.03) .. (304.85,62.98) -- cycle ;
\draw [color={rgb, 255:red, 245; green, 166; blue, 35 }  ,draw opacity=1 ]   (229.94,161.6) .. controls (233.37,195.6) and (349.08,209.89) .. (350.51,225.6) ;
\draw [shift={(283.01,197.76)}, rotate = 18.87] [fill={rgb, 255:red, 245; green, 166; blue, 35 }  ,fill opacity=1 ][line width=0.08]  [draw opacity=0] (7.14,-3.43) -- (0,0) -- (7.14,3.43) -- (4.74,0) -- cycle    ;
\draw [color={rgb, 255:red, 245; green, 166; blue, 35 }  ,draw opacity=1 ]   (299.65,72.75) .. controls (297.94,94.18) and (365.37,152.75) .. (399.93,150.9) ;
\draw [shift={(337.45,121.57)}, rotate = 37.99] [fill={rgb, 255:red, 245; green, 166; blue, 35 }  ,fill opacity=1 ][line width=0.08]  [draw opacity=0] (7.14,-3.43) -- (0,0) -- (7.14,3.43) -- (4.74,0) -- cycle    ;
\draw [color={rgb, 255:red, 80; green, 0; blue, 139 }  ,draw opacity=1 ]   (324.4,63.8) .. controls (316.65,127.82) and (317.22,173.25) .. (325.08,241.71) ;
\draw [shift={(318.91,153.6)}, rotate = 89.76] [fill={rgb, 255:red, 80; green, 0; blue, 139 }  ,fill opacity=1 ][line width=0.08]  [draw opacity=0] (8.04,-3.86) -- (0,0) -- (8.04,3.86) -- (5.34,0) -- cycle    ;
\draw [color={rgb, 255:red, 74; green, 144; blue, 226 }  ,draw opacity=1 ]   (324.4,63.8) .. controls (330.04,73.14) and (316.71,66.92) .. (317.6,76.92) .. controls (318.49,86.92) and (332.02,104.2) .. (331.35,123.98) .. controls (330.69,143.76) and (283.74,135.79) .. (280.85,166.23) .. controls (277.97,196.67) and (296.07,197.11) .. (318.73,189.55) .. controls (341.4,182) and (321.38,230.69) .. (322.27,234.91) .. controls (323.16,239.14) and (329.16,236.91) .. (325.08,241.71) ;
\draw [color={rgb, 255:red, 74; green, 144; blue, 226 }  ,draw opacity=1 ]   (285.19,188.46) .. controls (278.02,180.29) and (278.94,158.79) .. (287.92,151.23) ;
\draw [shift={(280.87,166.11)}, rotate = 94.35] [fill={rgb, 255:red, 74; green, 144; blue, 226 }  ,fill opacity=1 ][line width=0.08]  [draw opacity=0] (8.04,-3.86) -- (0,0) -- (8.04,3.86) -- (5.34,0) -- cycle    ;
\draw [color={rgb, 255:red, 65; green, 117; blue, 5 }  ,draw opacity=1 ]   (241.85,207.75) .. controls (279.25,161.75) and (355.25,142.5) .. (386,99.25) ;
\draw [shift={(314.77,153.43)}, rotate = 149.84] [fill={rgb, 255:red, 65; green, 117; blue, 5 }  ,fill opacity=1 ][line width=0.08]  [draw opacity=0] (8.04,-3.86) -- (0,0) -- (8.04,3.86) -- (5.34,0) -- cycle    ;
\draw [color={rgb, 255:red, 65; green, 117; blue, 5 }  ,draw opacity=1 ]   (390.75,200.25) .. controls (373.25,174.25) and (267,170) .. (236.75,107) ;
\draw [shift={(306.79,162.9)}, rotate = 21.87] [fill={rgb, 255:red, 65; green, 117; blue, 5 }  ,fill opacity=1 ][line width=0.08]  [draw opacity=0] (8.04,-3.86) -- (0,0) -- (8.04,3.86) -- (5.34,0) -- cycle    ;
\draw [color={rgb, 255:red, 65; green, 117; blue, 5 }  ,draw opacity=1 ]   (357.35,74.48) .. controls (339.35,95.23) and (309.1,89.73) .. (304.85,62.98) ;
\draw [shift={(326.57,86.66)}, rotate = 4.02] [fill={rgb, 255:red, 65; green, 117; blue, 5 }  ,fill opacity=1 ][line width=0.08]  [draw opacity=0] (8.04,-3.86) -- (0,0) -- (8.04,3.86) -- (5.34,0) -- cycle    ;
\draw [color={rgb, 255:red, 65; green, 117; blue, 5 }  ,draw opacity=1 ]   (284.1,68.48) .. controls (296.1,92.98) and (336.6,84.98) .. (343.5,67.58) ;
\draw [shift={(313.84,83.87)}, rotate = 181.4] [fill={rgb, 255:red, 65; green, 117; blue, 5 }  ,fill opacity=1 ][line width=0.08]  [draw opacity=0] (8.04,-3.86) -- (0,0) -- (8.04,3.86) -- (5.34,0) -- cycle    ;
\draw [color={rgb, 255:red, 65; green, 117; blue, 5 }  ,draw opacity=1 ]   (303.85,241.73) .. controls (308.1,221.21) and (341.85,205.98) .. (364.6,227.98) ;
\draw [shift={(333.49,217.54)}, rotate = 167.84] [fill={rgb, 255:red, 65; green, 117; blue, 5 }  ,fill opacity=1 ][line width=0.08]  [draw opacity=0] (8.04,-3.86) -- (0,0) -- (8.04,3.86) -- (5.34,0) -- cycle    ;
\draw [color={rgb, 255:red, 65; green, 117; blue, 5 }  ,draw opacity=1 ]   (353.5,234.23) .. controls (321.1,205.98) and (293.1,216.98) .. (287.1,238.48) ;
\draw [shift={(314.41,217.3)}, rotate = 2.32] [fill={rgb, 255:red, 65; green, 117; blue, 5 }  ,fill opacity=1 ][line width=0.08]  [draw opacity=0] (8.04,-3.86) -- (0,0) -- (8.04,3.86) -- (5.34,0) -- cycle    ;
\draw  [line width=1.5]  (224,153) .. controls (224,103.29) and (264.29,63) .. (314,63) .. controls (363.71,63) and (404,103.29) .. (404,153) .. controls (404,202.71) and (363.71,243) .. (314,243) .. controls (264.29,243) and (224,202.71) .. (224,153) -- cycle ;

\draw (319.29,121.04) node [anchor=east] [inner sep=0.75pt]  [color={rgb, 255:red, 80; green, 0; blue, 139 }  ,opacity=1 ,xscale=1.2,yscale=1.2]  {$\wt{\gamma }_{\wt{z}}$};
\draw (279.54,163.45) node [anchor=east] [inner sep=0.75pt]  [color={rgb, 255:red, 74; green, 144; blue, 226 }  ,opacity=1 ,xscale=1.2,yscale=1.2]  {$I_{\wt{\F}}^{\Z}(\wt{z})$};
\draw (374.6,111.4) node [anchor=north west][inner sep=0.75pt]  [color={rgb, 255:red, 65; green, 117; blue, 5 }  ,opacity=1 ,xscale=1.2,yscale=1.2]  {$\wt{\alpha }_{0}$};
\draw (255.96,136.72) node [anchor=south west] [inner sep=0.75pt]  [color={rgb, 255:red, 65; green, 117; blue, 5 }  ,opacity=1 ,xscale=1.2,yscale=1.2]  {$\wt{\alpha }_{1}$};
\draw (300.94,78.9) node [anchor=north east] [inner sep=0.75pt]  [color={rgb, 255:red, 65; green, 117; blue, 5 }  ,opacity=1 ,xscale=1.2,yscale=1.2]  {$R_{0}^{1}\wt{\alpha }_{0}$};
\draw (332.1,85.91) node [anchor=north west][inner sep=0.75pt]  [color={rgb, 255:red, 65; green, 117; blue, 5 }  ,opacity=1 ,xscale=1.2,yscale=1.2]  {$R^{1}_{1}\wt{\alpha }_{1}$};
\draw (299.12,223.16) node [anchor=south east] [inner sep=0.75pt]  [color={rgb, 255:red, 65; green, 117; blue, 5 }  ,opacity=1 ,xscale=1.2,yscale=1.2]  {$R^{-1}_{1}\wt{\alpha }_{1}$};
\draw (335,217) node [anchor=south west] [inner sep=0.75pt]  [color={rgb, 255:red, 65; green, 117; blue, 5 }  ,opacity=1 ,xscale=1.2,yscale=1.2]  {$R^{-1}_{0}\wt{\alpha }_{0}$};

\end{tikzpicture}

\caption{Idea of the proof of Lemma~\ref{LemCrossDependOnProj}: leaves (in orange) crossing $\wt\alpha_0$ have to stay in the orange area, hence $\wt B_0$ is included in this area.}\label{FigCrossDependOnProj}
\end{center}
\end{figure}

\begin{proof}
Set $C_2 = 16C_1$. Let us prove the first point, the second being identical. 
Suppose that $\pr_{\wt\gamma_{\wt z}}\big(\wt f^n(\wt y)\big) - \pr_{\wt\gamma_{\wt z}}(\wt y) \ge C_2$. Then Lemma~\ref{LemInterAlphaiEmpty} ensures that there exists $k\in\Z$ such that 
\[\pr_{\wt\gamma_{\wt z}}(\wt y) + 2 C_1 \le t^{k-1} \le \pr_{\wt\gamma_{\wt z}}(\wt y) + 6 C_1 .\]
Applying Lemma~\ref{LemInterAlphaiEmpty} once again, we deduce that 
\[t^{k+1} \le t^{k-1} + 8 C_1 \le \pr_{\wt\gamma_{\wt z}}(\wt y) + 14 C_1 \le \pr_{\wt\gamma_{\wt z}}\big(\wt f^n(\wt y)\big)-2C_1.\]
These inequalities, combined with Lemma~\ref{LemInterAlphaiEmpty}, imply that 
\begin{align*}
\pr_{\wt\gamma_{\wt z}}(\wt y) & < \min \pr_{\wt\gamma_{\wt z}}\Big(R_0^{k-1} \wt\alpha_0 \cup R_1^{k-1} \wt\alpha_1\Big);\\
\pr_{\wt\gamma_{\wt z}}\big(\wt f^n(\wt y)\big) & > \max \pr_{\wt\gamma_{\wt z}}\Big(R_0^{k+1} \wt\alpha_0 \cup R_1^{k+1} \wt\alpha_1\Big), 
\end{align*}
and so (in the sequel, we will make repeated use of Lemma~\ref{LemAlphaiSimple})
\[\wt y \in R(R_0^{k-1}\wt\alpha_0)\cap L(R_1^{k-1} \wt\alpha_1)\qquad \mathrm{and}\qquad  \wt f^n(\wt y)\in L(R_0^{k+1}\wt\alpha_0)\cap R(R_1^{k+1} \wt\alpha_1).\]

On the other hand, by Lemma~\ref{LemInterAlphaiEmpty},
\[R_0^{k} \wt\alpha_0 \subset L(R_0^{k-1}\wt\alpha_0)\cap R(R_1^{k-1} \wt\alpha_1) \cap R(R_0^{k+1}\wt\alpha_0)\cap L(R_1^{k+1} \wt\alpha_1).\]
Note that any leaf of $\wt \F$ intersecting this last set is disjoint from $R(R_0^{k-1}\wt\alpha_0)\cap L(R_1^{k-1} \wt\alpha_1)$ and from $L(R_0^{k+1}\wt\alpha_0)\cap R(R_1^{k+1} \wt\alpha_1)$ (see Figure~\ref{FigCrossDependOnProj}). Hence, $\wt y\notin R_0^k \wt B_0$ and $\wt f^n(\wt y)\notin R_0^k \wt B_0$. As $\wt y \in R(R_0^{k-1}\wt\alpha_0)\subset R(R_0^{k}\wt\alpha_0)$ and $\wt f^n(\wt y)\in L(R_0^{k+1}\wt\alpha_0 \subset L(R_0^{k}\wt\alpha_0)$, we deduce that $\wt y \in R(R_0^{k}\wt B_0)$ and $\wt f^n(\wt y)\in L(R_0^{k}\wt B_0)$.
\end{proof}

\begin{proof}[Proof of Theorem~\ref{TheoBndedDirLam}]
Set $C = 2 C_2 + d(\wt f, \Id_{\wt S})$.

By construction of $\wt B_0$ (it is the band associated to the approximation $\wt\alpha_0$ of $I^\Z_{\wt\F}(\wt z)$), this band is drawn by the transverse trajectory $I^\Z_{\wt\F}(\wt z)$. Moreover, applying Lemma~\ref{LemCrossDependOnProj} to $\wt y = \wt f^{-m}(\wt z)$ and $n=2m$ for $m$ large enough, we deduce that $I^\Z_{\wt\F}(\wt z)$ crosses $\wt B$ from right to left (in fact, we need the fact that $R_0^0 = \Id_{\wt S}$).

Suppose that the theorem is false: there exist $\wt y_0\in \wt S$ and $n_0\in\N$ such that $\pr_{\wt\gamma_{\wt z}}\big(\wt f^{n_0}(\wt y_0)\big) - \pr_{\wt\gamma_{\wt z}}(\wt y_0) \le -C$.
This allows us to apply Lemma~\ref{LemCrossDependOnProj} to $\wt y = \wt y_0$ and $n=n_0$, which implies that there exist $k\in\Z$ such that $I^{[0,n_0]}_{\wt \F}(\wt y_0)$ crosses $R_0^k\wt B_0$ from left to right.
%

By \cite[Proposition~2.16]{paper1PAF}, denoting $T_0\in\G$ the deck transformation associated to the loop $\wt\alpha_0$, we get that there are some $j,j'\in\Z$ such that $(R_0^k)^{-1}I^{[0,m]}_{\wt \F}(\wt y_0)$ and $T_0^j I^{\Z}_{\wt \F}(\wt z)$ intersect $\F$-transversally, and that $(R_0^{k'})^{-1}I^{[m,n_0]}_{\wt \F}(\wt y_0)$ and $T_0^{j'} I^{\Z}_{\wt \F}(\wt z)$ intersect $\F$-transversally. This is a contradiction with Proposition~\ref{LastPropBndDevIrrat}.
\end{proof}

\subsection{Proof of Corollary F}
\begin{proof}
Note that the second hypothesis of the corollary is stronger than the first one, so it suffices to prove both conclusions under the first hypothesis.

Let $\mu\in\Me(f)$ be such that $\rho(\mu) = \rho$. As $\rho\neq 0$, we have $\mu\in\M(f)$. Let $\cl_i$ be the class of $\mu$ for the relation $\sim$. Because $\rho\notin\R H_1(S,\Q)$, the class $\cl_i$ is not a closed class. It cannot be a chaotic class either: in this case, \cite[Theorem D]{alepablo} implies that $\rho$ is accumulated by elements of $\rote(S)\cap H_1(S,\Q)$, which contradicts our hypothesis.

Let us denote by $\Lambda_i$ the minimal orientable (by Theorem~\ref{TheoLaminMinim}) lamination associated to the class $\cl_i$ and given by Theorem~\ref{Thifsimpletheoremintro}. 
By orientability, there exists an embedded segment $J\subset S$ that is positively transverse to $\Lambda_i$. We can moreover suppose that this segment $J$ is included in the surface $S_i$ associated to $\cl_i$ (defined after Theorem~\ref{Thifsimpletheoremintro}). 
Following one leaf $\lambda$ of this minimal lamination, one can find two consecutive intersection points of $\lambda$ with $J$; by concatenating with the enclosed sub-segment of $J$ this gives us a simple closed loop $\beta \subset S_i$ that crosses positively the lamination $\Lambda_i$. Let $\gamma$ be the closed geodesic freely homotopic to $\beta$, this geodesic is positively transverse to $\Lambda_i$ and, by construction, included in $S_i$. 

By \cite[Proposition 4.6]{alepablo}, we have that $\rho(\mu) = [\wt\gamma_{\wt z}]$ for $\mu$-a.e.\ $z\in S$, where $[\wt\gamma_{\wt z}] = \lim_{t\to+\infty}\frac1t [\wt\gamma_{\wt z}|_{[0, t\vartheta_\mu]}]$. Hence, because $\Lambda_i$ is transverse to $\gamma$, we have $\rho(\mu)\wedge[\gamma]>0$. This implies that $\rho\wedge [\gamma]>0$. 

Let $\mu'\in \M(f)$ such that $\rho(\mu')\wedge [\gamma]\neq 0$. 
As before, we have $\rho(\mu') = [\wt\gamma_{\wt z}]$ for $\mu'$-a.e.\ $z\in S$, so $\gamma_z$ intersects $\gamma\subset S_i$. We deduce that $\mu'\in\cl_i$, and in particular $\gamma_z\in\Lambda_i$. We deduce that $[\wt\gamma_{\wt z}]\wedge [\gamma]>0$, hence $\rho(\mu')\wedge [\gamma]>0$. 

What we have done implies that $\rote(f) \subset \{\cdot\wedge [\gamma]\ge 0\}$ and so $\rot(f) \subset \conv(\rote(f) )\subset \{\cdot\wedge [\gamma]\ge 0\}$.

If we suppose that $\rho(\mu')\in \R\rho$, then either $\rho(\mu')\wedge [\gamma] = 0$ and in this case $\rho(\mu') = 0$, or $\rho(\mu')\wedge [\gamma]\neq 0$ and thus $\rho(\mu')\in \R_+\rho$.
\end{proof}

\small
\bibliographystyle{alpha}
\bibliography{Biblio}

\begin{thebibliography}{GLCPT25}

\bibitem[AGZ95]{anosovweil1}
Samuil Aranson, Vyacheslav~Z. Grines, and Evgeny~V. Zhuzhoma.
\newblock On the geometry and topology of flows and foliations on surfaces and
  the {Anosov} problem.
\newblock {\em Sb. Math.}, 186(8):1107--1146, 1995.

\bibitem[AGZ01]{anosovweil2}
Samuil Aranson, Vyacheslav~Z. Grines, and Evgeny~V. Zhuzhoma.
\newblock On {Anosov}-{Weil} problem.
\newblock {\em Topology}, 40(3):475--502, 2001.

\bibitem[Atk76]{zbMATH03533870}
Giles Atkinson.
\newblock Recurrence of co-cycles and random walks.
\newblock {\em J. Lond. Math. Soc., II. Ser.}, 13:486--488, 1976.

\bibitem[AZL22]{Addas-Zanata_2022}
Salvador Addas-Zanata and Xiao-Chuan Liu.
\newblock On stable and unstable behaviour of certain rotation segments.
\newblock {\em Nonlinearity}, 35(11):5813, oct 2022.

\bibitem[BCLR20]{bguin2016fixed}
Fran\c{c}ois B\'{e}guin, Sylvain Crovisier, and Fr\'{e}d\'{e}ric Le~Roux.
\newblock Fixed point sets of isotopies on surfaces.
\newblock {\em J. Eur. Math. Soc. (JEMS)}, 22(6):1971--2046, 2020.

\bibitem[CB88]{casson}
Andrew~J. Casson and Steven~A. Bleiler.
\newblock {\em Automorphisms of surfaces after {N}ielsen and {T}hurston},
  volume~9 of {\em London Mathematical Society Student Texts}.
\newblock Cambridge University Press, Cambridge, 1988.

\bibitem[GLCPT25]{GLCPT}
Pierre-Antoine Guih{\'e}neuf, Patrice Le~Calvez, Alejandro Passeggi, and
  Fabio~Armando Tal.
\newblock Area preserving surface homeomorphisms without horseshoe, 2025.
\newblock in preparation.

\bibitem[GM22]{pa}
Pierre-Antoine Guih\'eneuf and Emmanuel Militon.
\newblock Homotopic rotation sets for higher genus surfaces.
\newblock {\em to appear in Memoirs of the AMS}, 2022.

\bibitem[GSGL24]{alepablo}
Alejo Garc\'ia-Sassi, Pierre-Antoine Guih\'eneuf, and Pablo Lessa.
\newblock Geodesic tracking and the shape of ergodic rotation sets, 2024.

\bibitem[GT25]{paper1PAF}
Pierre-Antoine Guihéneuf and Fábio~Armando Tal.
\newblock Bounded deviations in higher genus {I}: closed geodesics, 2025.
\newblock arXiv: 2511.14222.

\bibitem[GZ17]{zbMATH06859894}
V.~Grines and E.~Zhuzhoma.
\newblock Around {Anosov}-{Weil} theory.
\newblock In {\em Modern theory of dynamical systems. A tribute to Dmitry
  Victorovich Anosov}, pages 123--154. Providence, RI: American Mathematical
  Society (AMS), 2017.

\bibitem[Kat73]{zbMATH03467479}
A.~B. Katok.
\newblock Invariant measures of flows on oriented surfaces.
\newblock {\em Sov. Math., Dokl.}, 14:1104--1108, 1973.

\bibitem[KT14a]{zbMATH06345227}
Andres Koropecki and Fabio~Armando Tal.
\newblock Area-preserving irrotational diffeomorphisms of the torus with
  sublinear diffusion.
\newblock {\em Proc. Am. Math. Soc.}, 142(10):3483--3490, 2014.

\bibitem[KT14b]{zbMATH06294042}
Andres Koropecki and Fabio~Armando Tal.
\newblock Strictly toral dynamics.
\newblock {\em Invent. Math.}, 196(2):339--381, 2014.

\bibitem[KT16]{zbMATH06908424}
Andres Koropecki and Fabio~Armando Tal.
\newblock Fully essential dynamics for area-preserving surface homeomorphisms.
\newblock {\em Ergodic Theory Dyn. Syst.}, 38(5):1791--1836, 2016.

\bibitem[LC05]{lecalvezfoliations}
Patrice Le~Calvez.
\newblock An equivariant foliated version of {Brouwer}'s translation theorem.
\newblock {\em Publ. Math., Inst. Hautes {\'E}tud. Sci.}, 102:1--98, 2005.

\bibitem[LCT18]{lct1}
Patrice Le~Calvez and Fabio Tal.
\newblock Forcing theory for transverse trajectories of surface homeomorphisms.
\newblock {\em Invent. Math.}, 212(2):619--729, 2018.

\bibitem[LCT22]{lct2}
Patrice Le~Calvez and Fabio Tal.
\newblock Topological horseshoes for surface homeomorphisms.
\newblock {\em Duke Math. J.}, 171(12):2519--2626, 2022.

\bibitem[LT24]{zbMATH07867510}
Xiao-Chuan Liu and F{\'a}bio~Armando Tal.
\newblock On non-contractible periodic orbits and bounded deviations.
\newblock {\em Nonlinearity}, 37(7):26, 2024.
\newblock Id/No 075007.

\bibitem[NZ99]{zbMATH01321275}
Igor Nikolaev and Evgeny Zhuzhoma.
\newblock {\em Flows on 2-dimensional manifolds. {An} overview}, volume 1705 of
  {\em Lect. Notes Math.}
\newblock Berlin: Springer, 1999.

\bibitem[SST22]{zbMATH07488214}
Guilherme Silva~Salom{\~a}o and Fabio~Armando Tal.
\newblock Non-existence of sublinear diffusion for a class of torus
  homeomorphisms.
\newblock {\em Ergodic Theory Dyn. Syst.}, 42(4):1517--1547, 2022.

\end{thebibliography}

\end{document}